\documentclass[11pt]{amsart}

\usepackage{amsmath,amssymb,amsthm}
\usepackage{thmtools,mathtools}
\usepackage{enumerate}
\usepackage{mathrsfs}
\usepackage[dvipsnames]{xcolor}

\usepackage[german,english]{babel}
\usepackage{csquotes}

\usepackage{tikz-cd}
\usetikzlibrary{babel}

\usepackage{todonotes} 
\usepackage{geometry}
\usepackage{tikz-cd}

\usepackage{hyperref}
\hypersetup{
	colorlinks=true,
	linkcolor=NavyBlue,
	urlcolor=teal,
	citecolor=OliveGreen}
\addto\extrasenglish{%
}

\theoremstyle{plain}

\newtheorem{thm}{Theorem}[section]
\newtheorem{prop}[thm]{Proposition}
\newtheorem{lem}[thm]{Lemma}

\theoremstyle{definition}

\newtheorem{dfn}[thm]{Definition}
\newtheorem{exa}[thm]{Example}
\newtheorem{rem}[thm]{Remark}
\newtheorem{question}[thm]{Question}

\newtheorem*{notation}{Notation}

\newcommand{\N}{\mathbb{N}}
\newcommand{\Z}{\mathbb{Z}}
\newcommand{\Q}{\mathbb{Q}}
\newcommand{\R}{\mathbb{R}}
\newcommand{\C}{\mathbb{C}}

\renewcommand{\P}{\mathbb{P}}
\newcommand{\OO}{\mathcal{O}}

\DeclareMathOperator{\codim}{codim}
\DeclareMathOperator{\Exc}{Exc}
\DeclareMathOperator{\Supp}{Supp}
\DeclareMathOperator{\Id}{Id}
\DeclareMathOperator{\Aut}{Aut}
\DeclareMathOperator{\vol}{vol}
\DeclareMathOperator{\Fix}{Fix}
\DeclareMathOperator{\Bl}{Bl}
\DeclareMathOperator{\Locus}{Locus}

\DeclareMathOperator{\ord}{ord}
\DeclareMathOperator{\Pic}{Pic}
\DeclareMathOperator{\Div}{Div}

\DeclareMathOperator{\Ima}{Im}
\DeclareMathOperator{\GL}{GL}
\DeclareMathOperator{\SL}{SL}
\DeclareMathOperator{\Hol}{Hol}
\DeclareMathOperator{\Sp}{Sp}

\DeclareMathOperator{\Spec}{Spec}
\DeclareMathOperator{\Hilb}{Hilb}

\DeclareMathOperator{\Sym}{Sym}

\DeclareMathOperator{\Amp}{\mathrm{Amp}}
\DeclareMathOperator{\Nef}{\mathrm{Nef}}

\DeclareMathOperator{\Pseff}{\overline{\mathrm{Eff}}}

\DeclareMathOperator{\Movb}{\overline{\mathrm{Mov}}}
\DeclareMathOperator{\NE}{\mathrm{NE}}
\DeclareMathOperator{\NEb}{\overline{\mathrm{NE}}}
\DeclareMathOperator{\BgCn}{\mathrm{Big}}
\DeclareMathOperator{\Mob}{Mob}

\DeclareMathOperator{\bMob}{bMob}
\DeclareMathOperator{\bMobc}{\overline{\mathrm{bMob}}}
\DeclareMathOperator{\Pos}{\mathrm{Pos}}
\DeclareMathOperator{\Posb}{\overline{\mathrm{Pos}}}

\DeclareMathOperator{\sbs}{\mathbf{B}} 
\DeclareMathOperator{\dbs}{\mathbf{B}_{--}} 
\DeclareMathOperator{\abs}{\mathbf{B}_{+}} 

\begin{document}

    \title[MMP for Enriques pairs and singular Enriques varieties]{MMP for Enriques pairs \\ and singular Enriques varieties}
    
    \author[F.\ A.\ Denisi]{Francesco Antonio Denisi}
    \address{\foreignlanguage{german}{Fachrichtung Mathematik, Campus, Geb\"aude E2.4, Universit\"at des Saarlandes, 66123 Saarbr\"ucken, Germany}}
    \email{denisi@math.uni-sb.de}

    \author[Á.\ D.\ Ríos Ortiz]{Ángel David Ríos Ortiz}
    \address{Universit\'e Paris-Saclay, CNRS, Laboratoire de Math\'ematiques d'Orsay, B\^at.\ 307, 91405 Orsay, France}
    \email{angel-david.rios-ortiz@universite-paris-saclay.fr}
    
    \author[N.\ Tsakanikas]{Nikolaos Tsakanikas}
    \address{Institut de Math\'ematiques (CAG), \'Ecole Polytechnique F\'ed\'erale de Lausanne (EPFL), 1015 Lausanne, Switzerland}
    \email{nikolaos.tsakanikas@epfl.ch}
    
    \author[Z.\ Xie]{Zhixin Xie}
    \address{Institut \'Elie Cartan de Lorraine, Universit\'e de Lorraine, 54506 Nancy, France}
    \email{zhixin.xie@univ-lorraine.fr}
    

    \begin{abstract}
        We introduce and study the class of primitive Enriques varieties, whose smooth members are Enriques manifolds. We provide several examples and we demonstrate that this class is preserved under the operations of the Minimal Model Program (MMP). In particular, given an Enriques manifold $Y$ and an effective $\R$-divisor $B_Y$ on $Y$ such that the pair $(Y,B_Y)$ is log canonical, we prove that any $(K_Y + B_Y)$-MMP terminates with a minimal model $(Y',B_{Y'})$ of $(Y,B_Y)$, where $Y'$ is a $\Q$-factorial primitive Enriques variety with canonical singularities. Finally, we investigate the asymptotic theory of Enriques manifolds.
	\end{abstract}

    \maketitle
    
	\begingroup
		\hypersetup{linkcolor=black}
		\tableofcontents
	\endgroup

    \section{Introduction}
    
    Irreducible holomorphic symplectic (IHS) manifolds are simply connected, compact K\"ahler manifolds carrying a holomorphic symplectic form that generates the space of holomorphic $2$-forms. According to the Beauville--Bogomolov decomposition theorem \cite{Beauville83a}, they constitute one of the three building blocks of compact K\"ahler manifolds with vanishing (real) first Chern class.
    Enriques manifolds, on the other hand, are projective manifolds that are not simply connected and whose universal covering is a projective IHS manifold. In particular, they can be viewed as finite \'etale quotients of the latter. They were defined and studied independently by Oguiso and Schr\"oer \cite{OS11a,OS11b} and by Boissière, Nieper-Wi{\ss}kirchen and Sarti \cite{BNWS11}.
    IHS and Enriques manifolds are regarded as higher-dimensional generalizations of K3 and Enriques surfaces, respectively. It is natural to extend these two notions to the singular setting as well.
    
    Primitive symplectic varieties (see \autoref{dfn:PSV}) have been introduced as singular analogs of IHS manifolds. Their class constitutes the most general framework for moduli and deformation theory \cite{BL21,BL22}. It is also preserved under the operations of the Minimal Model Program (MMP). Moreover, given a log canonical pair $(X,B)$ such that $X$ is a projective IHS manifold and $B$ is an effective $\R$-divisor on $X$, Matsushita and Zhang \cite[Theorem 4.1]{MZ13} showed that any sequence of flops starting from $(X,B)$ terminates with another log canonical pair $(X',B')$ such that $X'$ is a projective IHS manifold and $B'$ is a nef and effective $\R$-divisor on $X'$. Lehn and Pacienza \cite[Theorem 1.2]{LehnPac16} established later the following more general result.
    
    \begin{thm}
        \label{thm:termination_MMP_symplectic}
        Let $X$ be a projective IHS manifold and let $B$ be an effective $\R$-divisor on $X$ such that $(X,B)$ is a log canonical pair. Then any $(K_X + B)$-MMP terminates with a minimal model $(X',B')$ of $(X,B)$, where $X'$ is a $\Q$-factorial primitive symplectic variety and $B'$ is a nef and effective $\R$-divisor on $X'$.
    \end{thm}
    
    This paper has the following two main objectives: (i) to derive an analogous termination result for log canonical pairs $(Y,B_Y)$ such that $Y$ is an Enriques manifold and $B_Y$ is an effective $\R$-divisor on $Y$; and (ii) 
    to characterize the underlying variety of the resulting minimal model $(Y',B_{Y'})$ of $(Y,B_Y)$. To this end, we introduce \emph{primitive Enriques varieties} (see \autoref{dfn:PEV}) as singular analogs of Enriques manifolds. They are finite quasi-\'etale quotients of primitive symplectic varieties by non-symplectic group actions.

    Recall that the \emph{termination of flips conjecture} is one of the central problems in the MMP. It has been completely resolved in dimension $ 3 $ by Kawamata and Shokurov, as well as in many cases in dimension $ 4 $, but it remains widely open in general. We refer to \cite[Section 1]{ChenTsak23} for an overview and to \cite[Corollary 1.4.2]{BCHM10} for a partial result that is valid in any dimension. Our main result is the following and confirms in particular the termination of flips conjecture for the class of Enriques manifolds.
    
    \begin{thm}
        \label{thm:termination_MMP_Enriques}
        Let $Y$ be an Enriques manifold and let $B_Y$ be an effective $\R$-divisor on $Y$ such that $(Y,B_Y)$ is a log canonical pair. Then any $(K_Y + B_Y)$-MMP terminates with a minimal model $(Y',B_{Y'})$ of $(Y,B_Y)$, where $Y'$ is a $\Q$-factorial primitive Enriques variety with canonical singularities and $B_{Y'}$ is a nef and effective $\R$-divisor on $Y'$.
    \end{thm}

    We present now the strategy for the proof of \autoref{thm:termination_MMP_Enriques}. Given an Enriques pair $(Y,B_Y)$ as above, run any $(K_Y+B_Y)$-MMP. Consider the universal covering $\gamma \colon X \to Y$ of $Y$ and the corresponding IHS pair $(X,B)$, where $B \coloneqq \gamma^* B_Y$. To prove that the given $(K_Y+B_Y)$-MMP terminates, we first show that it can be lifted to a $\pi_1(Y)$-equivariant $(K_X+B)$-MMP. 
    We refer to Prokhorov's survey \cite{Pro21} for the equivariant MMP. We also stress that a step of an equivariant MMP need not be a composition of steps of a usual MMP, so \autoref{thm:termination_MMP_symplectic} cannot yet be applied to $(X,B)$ to derive \autoref{thm:termination_MMP_Enriques}.
    Therefore, we set $(X',B') \coloneqq (X, B)$ and we prove next that the $\pi_1(Y)$-equivariant $(K_X+B)$-MMP can in turn be lifted to a usual $(K_{X'} + B')$-MMP. 
    Since the $(K_{X'} + B')$-MMP terminates by \autoref{thm:termination_MMP_symplectic}, we conclude that the original $(K_Y+B_Y)$-MMP also terminates.
    
    \medskip

    Most of the examples of primitive Enriques varieties that we construct in this paper are finite quasi-\'etale quotients of irreducible symplectic varieties (see \autoref{dfn:IS_variety}) by non-symplectic group actions. We call these quotients \emph{irreducible Enriques varieties} (see \autoref{dfn:IEV}). 
    Our definition was inspired by the fundamental paper \cite{GGK19} of Greb, Guenancia and Kebekus, who characterized irreducible symplectic varieties in terms of holonomy groups. Building crucially on their results, we obtain a similar characterization of irreducible Enriques varieties. We refer to \autoref{subsection:def_Enriques} for the details.
    
    We point out that Samuel Boissi\`ere, Chiara Camere and Alessandra Sarti introduced independently in \cite{BCS24} another notion of singular Enriques varieties, which they call \emph{logarithmic Enriques varieties}. Our definitions of \emph{irreducible} and \emph{primitive} Enriques varieties differ slightly from their definitions of logarithmic Enriques varieties \emph{of ISV type} and \emph{of PSV type}, respectively, because singular Enriques varieties in our sense are also allowed to have trivial canonical class, whereas they exclude this possibility.
    
    \medskip
    
    The main body of the existing theory of Enriques manifolds is contained in the papers \cite{BNWS11,OS11a,OS11b}. Our last goal in this paper is to further develop this theory by studying the asymptotic properties of real divisors on Enriques manifolds.
    For instance, the following result is an analog of \cite[Theorem 1.6(2)]{Denisi23} for Enriques manifolds and characterizes their volume function, demonstrating that it is piecewise polynomial (see \autoref{dfn:piecewise_polynomial}).
    Our result should be compared with the second theorem on \cite[p.\ 210]{BKS04}, see also \cite[Subsection 3.3]{BKS04}.
    
    \begin{thm}
        \label{thm:volume_function_Enriques}
        If $Y$ is an Enriques manifold of dimension $2n$, then its volume function
        \[
            \vol_Y(-) \colon N^1(Y)_\R \to \R_{\geq 0} , \ \alpha \mapsto \vol_Y(\alpha)
        \]
        is homogeneous of degree $2n$ and piecewise polynomial.
    \end{thm}
    %
    
    Additionally, we address \cite[Problem]{Payne06} in the setting of Enriques manifolds. More precisely, Payne asks whether the closure of the cone of $k$-ample classes (see \autoref{dfn:amp_k-cone}) and the cone of birationally $k$-movable curves (see \autoref{dfn:birationally_k-movable_cone}) are dual. This problem can be viewed as an attempt to generalize the duality between the pseudoeffective cone and the cone of movable curves, obtained by \cite[Theorem 2.2]{BDPP13}. We establish here the following duality result, which is an analog of \cite[Theorem B]{DRO25} for Enriques manifolds.
    
    \begin{thm}
        \label{thm:cone_duality_Enriques}
        If $Y$ is an Enriques manifold of dimension $2n$, then for any integer
        $k \in \{ 1, \dots, 2n \}$
        we have
        \[
            \Amp^k(Y)^{\vee} = \bMobc_k(Y) . 
        \]
        In particular, if $1 \leq k \leq n$, then 
        \[ 
            \Amp^k(Y)^{\vee} = \NEb(Y) .
        \]
    \end{thm}
    
    The proof of the above theorem follows closely the strategy of the proof of \cite[Theorem B]{DRO25}, which settled Payne's question for IHS manifolds, while the new inputs are the MMP lifting constructions used to establish \autoref{thm:termination_MMP_Enriques}. It is worth mentioning that, even though the statement of \autoref{thm:cone_duality_Enriques} only involves an Enriques manifold, it is currently indispensable to pass to possibly singular birational models of the given Enriques manifold to conclude. Such models are $\Q$-factorial primitive Enriques varieties with terminal singularities in this context.

    \medskip

    \textbf{Organization of the paper}:
    In \autoref{section:preliminaries} we recall some definitions and facts about quasi-\'etale and Galois covers, the augmented irregularity, varieties with numerically trivial canonical divisor, and the usual or equivariant MMP.
    
    In \autoref{section:MMP} we establish the main ingredients for the proof of \autoref{thm:termination_MMP_Enriques}. Specifically, we first show in \autoref{prop:lifting_MMP} that any usual $(K_V + B)$-MMP can be lifted to an equivariant 
    MMP on a Galois quasi-\'etale cover of $V$.
    We then discuss in \autoref{prop:G-equivariant_termination} how the termination of any equivariant MMP can be deduced from the termination of usual MMPs.
    
    In \autoref{section:symplectic_varieties} we review the basic properties and some examples of irreducible and primitive symplectic varieties. We discuss next finite non-symplectic group actions on primitive symplectic varieties and the MMP for symplectic log canonical pairs. Everything will be used later in the study of primitive Enriques varieties.
    
    The heart of the paper, \autoref{section:Enriques_varieties}, is devoted to the definition and investigation of the singular generalizations of Enriques manifolds. In \autoref{subsection:def_Enriques} we define irreducible and primitive Eniques varieties and we establish their basic properties; see Propositions \ref{prop:characterization_IEV}, \ref{prop:algebra_reflexive_forms_IEV}, \ref{prop:basic_prop_PEV} and \ref{prop:PEV_closed_under_bir_contraction}.
    In \autoref{subsection:MMP_Enriques} we prove our \autoref{thm:termination_MMP_Enriques}, while in \autoref{subsection:examples_Enriques} we provide various examples of primitive Enriques varieties.

    In \autoref{section:asymptotic_theory_EM} we deal with the asymptotic theory of divisors on Enriques manifolds. In particular, we prove our Theorems \ref{thm:volume_function_Enriques} and \ref{thm:cone_duality_Enriques}.
    
    \medskip
    
    \textbf{Acknowledgements}: 
    The authors thank Chenyu Bai, Stefano Filipazzi, C\'ecile Gachet, Franco Giovenzana, Daniel Greb, Vladimir Lazi\'c, Alessandra Sarti, Martin Schwald and Sokratis Zikas for many valuable conversations related to this project. They are also grateful to Mirko Mauri not only for showing them \autoref{lem:GGK_cover} and \autoref{example:uniruled_Enriques} but also for suggesting many useful references, and to Gianluca Pacienza for a valuable conversation at an early stage of this work and for useful comments on a draft of the paper. The authors would finally like to thank the anonymous referees for pointing out some inaccuracies in the previous version of \autoref{lem:lifting_MMP_step}, for suggesting the current proof of \autoref{prop:PEV_closed_under_bir_contraction}, as well as for providing valuable feedback, which resulted in improving the exposition and polishing the paper.
    
    FAD and ÁDRO were supported by the European Research Council (ERC) under the European Union's Horizon 2020 research and innovation programme (ERC-2020-SyG-854361-HyperK). NT acknowledges support by the ERC starting grant $\#804334$.

    \section{Preliminaries}
    \label{section:preliminaries}

    We work throughout the paper over the field $\C$ of complex numbers, and a variety is an integral separated $\C$-scheme of finite type.
    We use the terminology of \cite{Laz04book_I} for notions of positivity and of \cite{KM98,Fuj17book} for standard notions in birational geometry, e.g., singularities of pairs and the Minimal Model Program (MMP). In particular, a \emph{pair} $(X,B)$ consists of a normal projective variety $X$ and an effective $\R$-divisor $B$ on $X$ such that $K_X+B$ is $\R$-Cartier.
    
    We frequently work also in the equivariant setting and we use \cite{Pro21} as our main reference for the equivariant MMP.
    In particular, a \emph{$G$-pair} consists of the following data: a pair $(X,B)$ as above, together with a finite subgroup $G$ of the automorphism group $\Aut(X)$ of $X$, such that $X$ is a $G$-variety and $B$ is a $G$-invariant divisor. In this setting, the $G$-variety $X$ is said to be \emph{$G\Q$-factorial} if any $G$-invariant Weil divisor on $X$ is $\Q$-Cartier.

    Given a normal projective variety $X$ and a finite subgroup $G$ of $\Aut(X)$, 
    we denote by $G_x$ the \emph{stabilizer} of a point $x\in X$, by $\Fix(g)$ the \emph{fixed locus} of an element $g\in G$, and by $X^\text{nf}$ the \emph{non-free locus} of the $G$-action.
    We say that $G$ acts \emph{freely in codimension one} on $X$ if $X^\text{nf}$ is a closed subset of $X$ with $\codim_X (X^\text{nf}) \geq 2$. If additionally $X^\text{nf} = \emptyset$, then we say that $G$ acts \emph{freely} on $X$.
    Finally, note that if $g \in \Aut(X)$ is a prime order automorphism of $X$, then by setting $G \coloneqq \langle g \rangle$, we have $ \Fix(g) = X^\mathrm{nf} $.

    \medskip
    
    A normal projective variety $X$ is said to be \emph{uniruled} if it is covered by rational curves. For the basic properties of uniruled varieties we refer to \cite{Debarre01book}, remarking only that uniruledness is a birationally invariant property. We also recall the following immediate corollary of \cite[Lemma 3.18]{LMPTX23}, which relies on \cite[Corollary 0.3]{BDPP13}.
    
    \begin{lem}
        \label{lem:uniruled_K-trivial}
        If $X$ is a normal projective variety such that $K_X \sim_\Q 0$, then $X$ is uniruled if and only if $X$ does not have canonical singularities.
    \end{lem}
    
    Given normal varieties $X$ and $Y$, we frequently consider the following types of maps between them. A \emph{fibration} is a projective surjective morphism $X \to Y$ with connected fibers. A \emph{birational contraction} is a birational map $X \dashrightarrow Y$ whose inverse does not contract any divisor. A \emph{small morphism} is a projective birational morphism $X \to Y$ which is an isomorphism in codimension one.
    Finally, a \emph{resolution} of $X$ is a projective birational morphism $W \to X$ from a smooth variety $W$.
    
    \begin{rem}
        \label{rem:contraction_from_K-trivial}
        Let $f \colon X \to Y$ be a proper birational morphism between normal varieties. If $K_X \sim_\Q 0$, then $K_Y \sim_\Q 0$, since $\codim_Y f \big( \Exc(f) \big) \geq 2$, and the birational morphism $f$ is actually crepant by the Negativity lemma \cite[Lemma 3.39(1)]{KM98}.
    \end{rem}

    \begin{rem}
        \label{rem:Q-fact_very_sing}
        If $X$ is a normal projective variety with klt but not canonical singularities and $K_X \sim_\Q 0$, then the underlying variety $T$ of any $\Q$-factorial terminal modification $h \colon (T, \Delta_T) \to X$ of $X$ satisfies $K_T \not\sim_\Q 0$. Indeed, by assumption and by construction (see \cite[Corollary 1.4.3]{BCHM10}) we know that $\Delta_T \neq 0$ is an effective $h$-exceptional $\Q$-divisor on $T$ and we also have $K_T + \Delta_T \sim_\Q h^* K_X \sim_\Q 0$, which yields the claim.  
    \end{rem}

    \subsection{Reflexive differential forms}

    Let $X$ be a normal variety. We denote by $\Omega^1_X$ the sheaf of K\"ahler differentials on $X$ and by $X_\text{reg}$ the regular locus of $X$, and we consider the natural inclusion $\iota \colon X_\text{reg} \hookrightarrow X$. For any integer $p \in \{ 0, \dots, \dim X \}$ we set
    \[ \Omega_X^{[p]} \coloneqq \big( \Lambda^p \Omega_X^1 \big)^{**} = \big( \Omega_X^p \big)^{**} \simeq \iota_* \Omega_{X_\text{reg}}^p , \]
    and we refer to a global section of $\Omega_X^{[p]}$ as a \emph{reflexive $p$-form} on $X$. 
    Note that
    \[ H^0 \big( X, \Omega_X^{[p]} \big) = H^0 \big( X_\text{reg}, \Omega^p_{X_\text{reg}} \big) . \]
    We will frequently use the next result, which was established by Greb, Kebekus and Peternell, see \cite[Proposition 6.9]{GKP16b}.

    \begin{prop}[Hodge duality for klt spaces]
        \label{prop:Hodge_duality}
        Let $X$ be a normal projective variety which admits the structure of a klt pair.\footnote{The phrase \enquote{$X$ admits the structure of a klt pair} means that there exists an effective $\R$-divisor $\Delta$ on $X$ such that $(X,\Delta)$ is a klt pair. By \cite[Lemma 2.14(i)]{TX25} this is equivalent to assuming that there exists an effective $\Q$-divisor $\Delta'$ on $X$ such that $(X,\Delta')$ is a klt pair.} For any integer $p \in \{ 0, \dots, \dim X \}$ there are $\C$-linear isomorphisms
        \[ 
            H^0 \big( X, \Omega_X^{[p]} \big) \simeq \overline{H^p \big( X,\OO_X \big)} . 
        \]
    \end{prop}

    \subsection{Galois and quasi-\'etale covers}

    The following definitions are taken from \cite[Section 3]{GKP16a}.
    
    \begin{dfn}
        Let $X$ and $Y$ be normal varieties.
        \begin{enumerate}[(a)]
            \item A morphism $\gamma \colon X \to Y$ is called \emph{quasi-\'etale} if $ \dim X = \dim Y $ and if there exists a closed subset $ Z \subset X $ of codimension $\codim_X Z \geq 2$ such that $ \gamma |_{X \setminus Z} \colon X \setminus Z \to Y $ is \'etale.
            
            \item A \emph{cover} or \emph{covering map} is a finite surjective morphism $\gamma \colon X \to Y$.

            \item Assume now that $X$ and $Y$ are also quasi-projective. A covering map $\gamma \colon X \to Y$ is called \emph{Galois} if there exists a finite subgroup $G$ of $\Aut(X)$ such that $\gamma$ is isomorphic to the quotient map $X \twoheadrightarrow X / G$.
        \end{enumerate}
    \end{dfn}

    
    \begin{rem}~
        \label{rem:Galois_and_quasi-etale}
        \begin{enumerate}[(i)]
            \item If $\gamma \colon X \to Y$ is a quasi-\'etale cover of a normal variety $Y$, then $\gamma$ is \'etale over $Y_\text{reg}$ by purity of the branch locus \cite[Tag 0BMB]{Stacks}. Therefore, a quasi-\'etale cover of a smooth variety is \'etale.

            \item If $\gamma \colon X \to Y$ is a Galois cover determined by a finite group $G \subseteq \Aut(X)$, then $\gamma$ is \'etale (resp.\ quasi-\'etale) if and only if $G$ acts freely (resp.\ freely in codimension one) on $X$.

            \item If $\gamma$ and $G$ are as in part (ii), then $X$ is $G\Q$-factorial if and only if $Y \simeq X /G$ is $\Q$-factorial by \cite[Lemma 1.1.3]{Pro21}.
            
            \item If $\gamma \colon X \to Y$ is a Galois quasi-\'etale cover determined by a finite group $G \subseteq \Aut(X)$, then 
            \[ 
                H^0 \big(Y,\Omega_Y^{[k]} \big) \simeq H^0 \big(X,\Omega_X^{[k]} \big)^G \ \text{ for any integer } \, 0 \leq k \leq \dim X = \dim Y .
            \]
        \end{enumerate}
    \end{rem}
    
    \begin{rem}
        \label{rem:singularities_quasi-etale_cover}
        Given a quasi-\'etale cover $\gamma \colon X \to Y$ of a normal variety $Y$, by \cite[Proposition 5.20]{KM98}
        we obtain the following:
        \begin{enumerate}[(i)]
            \item $K_X$ is $\Q$-Cartier if and only if $K_Y$ is $\Q$-Cartier, and we also have $K_X \sim_\Q \gamma^* K_Y$. In particular, $K_Y \sim_\Q 0$ (resp.\ $K_Y \equiv 0)$ if and only if $K_X \sim_\Q 0$ (resp.\ $K_X \equiv 0)$; see \cite[Chapter 7, Theorem 2.18]{Liu02book} and \cite[Tag 02RT]{Stacks} for the implication $K_X \sim_\Q 0 \implies K_Y \sim_\Q 0$.
            
            Moreover, if $Y$ has canonical (resp.\ terminal) singularities, then $X$ has canonical (resp.\ terminal) singularities. On the other hand, if $X$ has canonical singularities, then $Y$ has at worst klt singularities. An example of a klt quotient of a smooth surface is given in \autoref{example:uniruled_Enriques}.

            \item If $B_Y$ is an effective $\R$-divisor on $Y$ and if we set $B_X \coloneqq \gamma^* B_Y$, then we have 
            \[ K_X + B_X \sim_\Q \gamma^* (K_Y + B_Y) . \]
            Furthermore, $(X,B_X)$ is a log canonical (resp.\ klt) pair if and only if $(Y,B_Y)$ is a log canonical (resp.\ klt) pair.
        \end{enumerate}
    \end{rem}    
    
    \begin{lem}
        \label{lem:NeronSeveri_GaloisCovering}
        Let $\gamma \colon X\to Y \simeq X/G$ be a Galois cover between normal projective varieties determined by a finite subgroup $G$ of $\Aut(X)$. Then 
        \[ 
            N^1(X)_{\R}^G \simeq N^1(Y)_{\R} \quad \text{ and } \quad N_1(X)_{\R}^G \simeq N_1(Y)_{\R} .
        \]
    \end{lem}
        
    \begin{proof}
        First, as $\gamma$ is the quotient map $X \to X/G \simeq Y$, note that the pullback of Cartier divisors induces an injective map 
        \[ 
            \gamma^* \colon N^1(Y)_{\R} \hookrightarrow N^1(X)_{\R}^G .
        \]
        To prove its surjectivity, and establish thus the isomorphism $N^1(X)_{\R}^G \simeq N^1(Y)_{\R}$, let $D$ be a $G$-invariant $\R$-Cartier $\R$-divisor on $X$. As 
        \[ 
            N^1(X)_{\R}^G \simeq N^1(X)^G \otimes \R , 
        \]
        we may assume that $D$ is an integral Cartier divisor.
        Since $\gamma$ is finite, arguing as in the first part of the proof of \cite[Lemma 1.1.3]{Pro21}, we infer from \cite[Chapter 7, Remark 2.19]{Liu02book} that $\gamma_* (D)$ is a $\Q$-Cartier divisor on $Y$, 
        and we also have $\gamma^* \gamma_* (D) = |G|D$,
        %
        %
        %
        as shown in the second part of the proof of \cite[Lemma 1.1.3]{Pro21}. This implies that $\gamma^*$ is surjective, as desired.
        
        Second, note that the usual proper push-forward of $1$-cycles induces a surjective map of finite-dimensional $\R$-vector spaces 
        \[ \gamma_*\colon N_1(X)_{\R}^G \twoheadrightarrow N_1(Y)_{\R} . \]
        To prove the isomorphism $N_1(X)_{\R}^G \simeq N_1(Y)_{\R}$, it remains to check that $\gamma_*$ is injective. Supposing that $\alpha\in N_1(X)_{\R}^G$ is nonzero, we now show that so is $\gamma_*\alpha \in N_1(Y)_\R$. Indeed, by assumption, there exists a Cartier divisor $D'$ on $X$ such that $D'\cdot \alpha\neq 0$. Since $\alpha$ is $G$-invariant, for every $g\in G$ we have $D'\cdot \alpha = g^*D'\cdot g^*\alpha = g^*D'\cdot \alpha$. In particular, if we set $D_X \coloneqq \sum_{g\in G}g^*D'$, then we have $D_X\cdot \alpha \neq 0$. Since $D_X$ is $G$-invariant by construction, there exists a Cartier divisor $D_Y$ on $Y$ such that $\gamma^*D_Y=D_X$. The projection formula gives $D_Y\cdot \gamma_*\alpha=D_X\cdot \alpha \neq 0$, and we are done.
    \end{proof}

    \subsection{Augmented irregularity}
    
    Given a normal projective variety $X$, denote by 
    $q(X) \coloneqq h^1(X,\OO_X)$ 
    the \emph{irregularity} of $ X $, and by
	\[ \widetilde{q}(X) \coloneqq \sup \big\{ q(W) \mid W \text{ is a quasi-\'etale cover of } X \big\} \in \N \cup \{\infty\} \]
	the \emph{augmented irregularity} of $ X $. 
    If $X$ has klt singularities, then 
    $ q(X) = h^0 \big( X, \Omega_X^{[1]} \big) $
    by \autoref{prop:Hodge_duality}, and if additionally $ X $ has numerically trivial canonical divisor, then $\widetilde{q}(X) \leq \dim X$ by \cite[Lemma 2.19.2]{GGK19}.
    
    \begin{rem}~
        \label{rem:augmented_irreg}
        \begin{enumerate}[(i)]
            \item The irregularity is a birational invariant of normal projective varieties with rational singularities
            (e.g., with klt singularities by \cite[Theorem 5.22]{KM98}).

            \item The augmented irregularity is a birational invariant for smooth projective varieties; see \cite[Section 2]{Anella21b} for the proof of this fact. However, this is no longer true for normal projective varieties even with canonical singularities; see \autoref{example:singular_Kummer_surface}.

            \item The augmented irregularity is invariant under quasi-\'etale covers by \cite[Lemma 2.19.1]{GGK19}.
        \end{enumerate}
    \end{rem}    
    
    %

    \subsection{Reminders about an MMP step}
    \label{subsection:MMP_step}

    Let $(X,B)$ be a log canonical pair. Assume that $K_X + B$ is not nef and consider a step of a $(K_X+B)$-MMP:
    \begin{center}
        \begin{tikzcd}
            (X, B) \arrow[dr, "\eta" swap] \arrow[rr, "\varphi", dashed] && (X^+, B^+) \arrow[dl, "\eta^+"] \\
            & Z ,
        \end{tikzcd}
    \end{center}
    where $B^+ = \varphi_* B$. The pair $(X^+, B^+)$ is log canonical by \cite[Lemma 2.3.27]{Fuj17book}, and if $X$ is $\Q$-factorial, then so is $X^+$ by \cite[Propositions 4.8.18 and 4.8.20]{Fuj17book}.
        
    Assume now that $K_X \equiv 0$. It follows from \cite[Theorem 4.5.2(4)(iii)]{Fuj17book} that $K_{X^+} \equiv 0$.
    Hence, if $X$ has canonical (resp.\ klt) singularities, then \cite[Lemma 3.38]{KM98} implies that $X^{+}$ also has canonical (resp.\ klt) singularities, and thus both $X$ and $X^+$ have rational singularities by \cite[Theorem 5.22]{KM98}.
    
    \medskip
    
    We demonstrate next that the same statements as above also hold in the equivariant setting, but we also provide more information about the equivariant MMP, which are necessary for our purposes in this paper.

    \medskip

    Let $(X,B)$ be a log canonical $G$-pair.
    Assume that $K_X+B$ is not nef and let $\eta \colon X \to Z$ be the contraction of a $(K_X+B)$-negative extremal ray $R$ of $\NEb(X)^G$, which exists by \cite[Theorem 3.3.1]{Pro21}. We distinguish two cases:

    \medskip

    \emph{Case A}:
    If $\eta$ is a flipping $G$-contraction (i.e., $\codim_X \Exc(\eta) \geq 2$), then by \cite[Theorems 3.4.2 and 3.4.3]{Pro21} we obtain the following $G$-equivariant commutative diagram:
    \begin{center}
        \begin{tikzcd}
            (X, B) \arrow[dr, "\eta" swap] \arrow[rr, "\varphi", dashed] && (X^+, B^+) \arrow[dl, "\eta^+"] \\
            & Z ,
        \end{tikzcd}
    \end{center}
    where $B^+ = \varphi_* B$, the $G$-pair $(X^+, B^+)$ is log canonical, and when $X$ is $G\Q$-factorial, then so is $X^+$ by \cite[Theorem 3.4.2(ii)]{Pro21}. Furthermore, in case that $G$ acts freely in codimension one on $X$, then it also acts freely in codimension one on both $X^+$ and $Z$, since the maps $\eta$, $\eta^+$ and $\varphi$ are all isomorphisms in codimension one.

    \medskip

    \emph{Case B}:
    If $\eta$ is a divisorial $G$-contraction (i.e., $\codim_X \Exc(\eta) = 1$, but the closed subscheme $\Exc(\eta) \subset X$ need not be irreducible or equidimensional in general), then by \cite[Theorem 3.3.2]{Pro21} we obtain the following commutative diagram:
    \begin{center}
        \begin{tikzcd}
            (X, B) \arrow[dr, "\eta" swap] \arrow[rr, "\varphi \, = \, \eta"] && (Z,B_Z) \arrow[dl, "\eta^+ = \, \Id_Z"] \\
            & Z ,
        \end{tikzcd}
    \end{center}
    where $B_Z = \varphi_* B$, the $G$-pair $(Z,B_Z)$ is log canonical, and if $X$ is $G\Q$-factorial, then so is $Z$ and additionally $\Exc(\eta)$ is a $G$-prime divisor on $X$ by \cite[Theorem 3.3.2(ii)]{Pro21}. Furthermore, in case that $G$ acts freely in codimension one on $X$, then it also acts freely in codimension one on $Z$, since the indeterminacy locus of the inverse $\varphi^{-1} = \eta^{-1} \colon Z \dashrightarrow X$ has codimension at least two in $Z$.
        
    \medskip
        
    In conclusion, we have in any case a \emph{$G$-equivariant MMP step}, that is, a commutative diagram of the form
    \begin{center}
        \begin{tikzcd}
            (X, B) \arrow[dr, "\eta" swap] \arrow[rr, "\varphi", dashed] && (X^+, B^+) \arrow[dl, "\eta^+"] \\
            & Z .
        \end{tikzcd}
    \end{center}
    We now make the following observations about it:
    
    \begin{rem}~
        \label{rem:G-equivariant_K-trivial_canonical_sing}
        \begin{enumerate}[(i)]
            \item The pair $(X^+,B^+)$ is the \emph{canonical model} of $(X,B)$ over $Z$; see \cite[Definition 4.8.1]{Fuj17book} and the proof of \cite[Proposition 4.2.2]{Pro21}.

            \item The pair $(X,B)$ has a minimal model $(X^m, B^m)$ over $Z$, which is obtained by running some $(K_X+B)$-MMP over $Z$ with scaling of an ample divisor. Indeed, by (i) and since a relative canonical model is by definition a relative weak canonical model, the assertion follows from \cite[Corollary 3.7]{Bir12a} and \cite[Theorem 1.7]{HH20}.
            Moreover, by (i) and by \cite[Lemma 4.8.4]{Fuj17book} there exists a projective birational morphism $ f \colon X^m \to X^+ $ such that 
            \[ 
                K_{X^m} + B^m \sim_\R f^* (K_{X^+} + B^+). 
            \]
            Thus, $(X^m, B^m)$ is actually a good minimal model of $(X,B)$ over $Z$.
            
            \item Assume now that $K_X \equiv 0$. 
            We claim that $K_{X^+} \equiv 0$. Indeed, since $G \subseteq \Aut(X)$ is finite, by \cite[Lemma 1.1.2]{Pro21} there exists a $G$-invariant $\Q$-divisor $D_X$ on $X$ such that $D_X \sim_\Q K_X$. 
            Since $K_X \equiv 0$, by \cite[Theorem 3.3.1(iii)]{Pro21} there exists a $G$-invariant $\Q$-Cartier $\Q$-divisor $D_Z$ on $Z$ such that $K_X \sim_\Q D_X \sim_\Q \eta^* D_Z$, whence $D_Z \equiv 0$. It follows now that $K_{X^+} \equiv 0$, as claimed. 
            
            Thus, by \cite[Lemma 3.38]{KM98} we deduce that if $X$ has canonical (resp.\ klt) singularities, then $X^+$ also has canonical (resp.\ klt) singularities, and hence both $X$ and $X^+$ have rational singularities by \cite[Theorem 5.22]{KM98}.
        \end{enumerate}
    \end{rem}

    \section{Equivariant MMP and Galois quasi-\'etale covers}
    \label{section:MMP}
            
    \begin{lem}
        \label{lem:invariant_extremal_ray}
        Let $\gamma \colon X\to Y \simeq X/G$ be a Galois cover between normal projective varieties determined by a finite subgroup $G$ of $\Aut(X)$.
        \begin{enumerate}[\normalfont (i)]
            \item The push-forward of $1$-cycles induces a bijection $\NEb(X)^G \to \NEb(Y)$.
            
            \item Let $R$ be an extremal ray of $\NEb(Y)$. If $\alpha \in \NEb(X)^G$ satisfies $\gamma_*(\alpha) \in R$, then it spans an extremal ray $R^G \coloneqq \R_{\geq 0} \, \alpha$ of $\NEb(X)^G$. In particular, if $R = \R_{\geq 0} [C]$ for some irreducible curve $C$ on $Y$, then $R^G = \R_{\geq 0} \big[ \sum_{g\in G} g(C') \big]$ for some irreducible curve $C'$ on $X$ such that $\gamma(C') = C$.
            
            \item Suppose now that $\gamma$ is quasi-\'etale. Let $B_Y$ be an effective $\R$-divisor on $Y$ such that $(Y, B_Y)$ is a log canonical pair, and consider the log canonical pair $(X,B_X)$, where $B_X \coloneqq \gamma^* B_Y$.
            If $R$ is a $(K_Y+B_Y)$-negative extremal ray of $\NEb(Y)$, then $R^G$ is a $(K_X+B_X)$-negative extremal ray of $\NEb(X)^G$.
            Moreover, 
            \[ 
                \Locus \left( R^G \right) = \gamma^{-1} \big( \Locus(R) \big) . 
            \]  
        \end{enumerate}
    \end{lem}

    Before proceeding with the proof of the previous lemma, we recall the notation introduced in part (iii) above, which is inspired from \cite[Section 6.5, p.\ 154]{Debarre01book}: Given a log canonical pair $(V,B_V)$, a $(K_V+B_V)$-negative extremal ray $R$ of $\NEb(V)$, and the associated contraction morphism $c_R \colon V \to W$ which contracts an irreducible curve $C \subseteq V$ if and only if $[C] \in R$, see \cite[Theorem 4.5.2]{Fuj17book}, we call the \emph{locus} of $R$ the following closed subset of $V$:
    \[ 
        \Locus(R) \coloneqq \bigcup_{\substack{C: \; \textrm{irred. curve} \\[0.1em] \textrm{on $V$ s.t.} \, [C] \in R}} \Supp(C) .
    \]
    
    \begin{proof}
        We denote by 
        \[ \gamma_* \colon N_1(X)_\R \to N_1(Y)_\R \] 
        the usual proper push-forward of $1$-cycle classes and we note that 
        \begin{equation}
            \label{eq:ker_pushforward}
            \ker (\gamma_*) \subseteq \big( N_1(X)_\R \setminus \NEb(X) \big) \cup \{ 0 \} ,
        \end{equation}
        because a limit of effective classes intersects positively ample classes by \cite[Theorem 1.44]{KM98}. Indeed, take a nonzero element $\overline{\zeta} \in \NEb(X)$. There is a sequence $\{\zeta_n\}_n$, where $\zeta_n \in \NE(X)$ for every $n \geq 1$, such that $ \zeta_n \to \overline{\zeta}$. The push-forward of $1$-cycles gives a sequence $\big\{ \gamma_*(\zeta_n) \big\}_n$, where $\gamma_* (\zeta_n) \in \NE(Y)$ for every $n \geq 1$, such that $\gamma_*(\zeta_n) \to \gamma_*(\overline{\zeta})$. Thus, for any ample divisor $A$ on $Y$, the projection formula yields
        \[ 
            0\neq \gamma^*[A]\cdot \overline{\zeta}=[A]\cdot \gamma_*(\overline{\zeta}) ,
        \]
        whence $\gamma_*(\overline{\zeta})$ is not the zero element. This proves \eqref{eq:ker_pushforward}.
        
        \medskip

        \noindent (i)
        By \autoref{lem:NeronSeveri_GaloisCovering} and by \eqref{eq:ker_pushforward} the linear map $\NEb(X)^G\to \NEb(Y)$ induced by $\gamma_*$ is injective. To show its surjectivity, pick an element $\overline{\zeta'} \in \NEb(Y)$ and a sequence $\{\zeta'_n\}_n$ of effective cycles such that $\zeta'_n \to \overline{\zeta'}$. Since $\gamma_*$ is surjective on effective cycles, for any element $\zeta'_n \in \NE(Y)$ there is a unique element $\zeta_n\in \NE(X)^G$ such that $\zeta_n \mapsto \zeta_n'$. 
        Then the limit $\overline{\zeta}$ of $\{\zeta_n\}_n$ exists in $\NEb(X)^G$ and it is the preimage of $\overline{\zeta'}$ by the continuity of $\gamma_*$.

        \medskip

        \noindent (ii)
        The second item is a direct consequence of the first one. 
        Indeed, write $\alpha = \beta_1 + \beta_2$ with $\beta_1, \beta_2 \in \NEb(X)^G \setminus \{0\}$. Then $\gamma_*(\beta_1 + \beta_2)=\gamma_*(\beta_1) + \gamma_*(\beta_2) \in R$ by assumption, so both $\gamma_*(\beta_1)$ and $ \gamma_*(\beta_2)$ lie in $R$ by the injectivity of $\gamma_* |_{\NEb(X)^G}$ and the extremality of $R$. This implies that $\alpha=l_1 \beta_1$ and $\alpha=l_2 \beta_2$ for some $l_1, l_2 \in \R_{>0}$, 
        whence both $\beta_1$ and $\beta_2$ belong to $R^G = \R_{\geq 0} \, \alpha$. 
        
        Suppose now that the ray $R$ is spanned by the class of an irreducible curve $C$ in $Y$. Let $C'$ be an irreducible curve in $X$ mapping onto $C$. Then the class $\beta$ of the $1$-cycle $\sum_{g \in G} g(C')$ is $G$-invariant and maps to an element of $R$ under $\gamma_*$. Thus, by the argument above, $\beta$ spans the extremal ray $R^G$ constructed above.

        \medskip

        \noindent (iii)
        Let $C'$ be any irreducible curve in $\Locus \left( R^G \right)$. We have 
        \[ 
            0 > (K_X+B_X) \cdot C' = \gamma^*(K_Y+B_Y) \cdot C' =(K_Y+B_Y) \cdot \gamma_*(C') 
        \]
        and $[\gamma_* (C')] \in R$ by construction of $R^G$.
        Therefore, $\gamma\big( \Locus(R^G) \big) \subseteq \Locus(R)$, so 
        \[ 
            \Locus \left( R^G \right) \subseteq \gamma^{-1} \big( \Locus(R) \big) .
        \]
        Conversely, let $C$ be an irreducible curve in $\Locus(R)$, consider its preimage $\gamma^{-1}(C)$ in $X$, and let $C_1'$ be an irreducible component of $\gamma^{-1}(C)$. We have 
        \[
            (K_X+B_X) \cdot C_1' = \gamma^*(K_Y+B_Y) \cdot C_1' = (K_Y+B_Y) \cdot \gamma_*(C_1') < 0 .
        \]
        Since this can be done for every irreducible component of $\gamma^{-1}(C)$, by construction of $R^G$ we deduce that 
        \[
            \Locus \left( R^G \right) \supseteq \gamma^{-1} \big( \Locus(R) \big) .
        \]
        This concludes the proof.
    \end{proof}
    
    \begin{lem}
        \label{lem:lifting_MMP_step}
        Let $\gamma_1 \colon X'_1 \to X_1$ be a Galois quasi-\'etale cover between normal projective varieties with numerically trivial canonical divisor. Denote by $G$ the finite subgroup of $\Aut(X'_1)$ defining $\gamma_1$ and assume that $X_1'$ is $G\Q$-factorial or, equivalently, that $X_1 \simeq X_1' /G$ is $\Q$-factorial. Let $B_1$ be an effective $\R$-divisor on $X_1$ such that $(X_1, B_1)$ is a log canonical pair, and consider the log canonical pair $(X'_1,B'_1)$, where $B'_1 \coloneqq (\gamma_1)^* B_1$. Assume that $K_{X_1} + B_1$ is not nef and consider a step of a $(K_{X_1}+B_1)$-MMP:
        \begin{center}
            \begin{tikzcd}
                (X_1, B_1) \arrow[rr, "\varphi_1", dashed] \arrow[dr, "\eta_1" swap] && (X_2, B_2) \arrow[dl, "\eta_1^+"] \\
                & Z_1 ,
            \end{tikzcd}
        \end{center}
        where $B_2 = (\varphi_1)_* B_1$.
        Then there exists a commutative diagram:
        \begin{center}
            \begin{tikzcd}
                (X'_1,B'_1) \arrow[dd, "\gamma_1" swap] \arrow[rr, "\varphi'_1", dashed] \arrow[dr, "\eta'_1" swap] && (X'_2, B'_2) \arrow[dl, "(\eta'_1)^+"] \arrow[dd, "\gamma_2"] \\
                & Z'_1 \arrow[dd, "\mu_1", pos=0.3, swap] \\
                (X_1, B_1) \arrow[rr, "\varphi_1", pos=0.65, dashed, crossing over] \arrow[dr, "\eta_1" swap] && (X_2, B_2) \arrow[dl, "\eta_1^+"] \\
                & Z_1 ,
            \end{tikzcd}
        \end{center}
        where 
        \begin{itemize}
            \item $\varphi'_1 \colon X'_1 \dashrightarrow X'_2$ is a step of a $G$-equivariant $(K_{X'_1}+B'_1)$-MMP,

            \item $\mu_1 \colon Z'_1 \to Z_1$ and $\gamma_2 \colon X'_2 \to X_2$ are Galois quasi-\'etale covers defined by $G$ between normal projective varieties with numerically trivial canonical divisor, and

            \item $B'_2 \coloneqq (\varphi'_1)_* B'_1 = (\gamma_2)^* B_2$.
        \end{itemize}
        In particular, if $\varphi_1$ is a divisorial contraction (in which case we have $X_2 = Z_1$, $(\eta_1)^+ = \Id_{Z_1}$ and $\varphi_1 = \eta_1$) or a flip, then $\varphi'_1$ is a $G$-equivariant divisorial contraction (in which case we have $X'_2 = Z'_1$, $(\eta'_1)^+ = \Id_{Z'_1}$, $\varphi'_1 = \eta'_1$ and $\mu_1 = \gamma_2$) or a $G$-equivariant flip, respectively.
    \end{lem}
    
    \begin{proof}
        By \autoref{rem:singularities_quasi-etale_cover}(ii) we know that $(X_1',B_1')$ is a log canonical pair such that 
        \[ K_{X_1'} + B_1' = (\gamma_1)^* (K_{X_1} + B_1) . \]
        Set $d\coloneqq |G|$, observe that $\gamma_1$ is a finite morphism of degree $d$, and note that $B_1' = (\gamma_1)^* B_1$ is a $G$-invariant divisor; indeed, since $\gamma_1$ is Galois and \'{e}tale in codimension one, for any irreducible component $P$ of $B_1$ we have $\gamma_1^* (P)=\sum_{k=1}^d P'_k$, where 
        $\big\{ P_k' \mid 1 \leq k \leq d \big\}$ 
        is the $G$-orbit of a prime divisor $P_1'$ on $X_1'$ which maps onto $P$ under $\gamma_1$.
    
        Observe now that the given $(K_{X_1}+B_1)$-MMP step $\varphi_1 \colon (X_1,B_1)\dashrightarrow (X_2,B_2)$ is $B_1$-negative, since $K_{X_1} \equiv 0$. Denote by $C_1$ an integral curve on $X_1$ whose class spans the $(K_{X_1}+B_1)$-negative extremal ray $R_1$ that corresponds to $\varphi_1$. By \autoref{lem:invariant_extremal_ray}(iii) we obtain a $(K_{X_1'}+B_1')$-negative extremal ray $R_1^G$ of $\NEb(X_1')^G$ such that $\Locus \left( R_1^G \right) = \gamma_1^{-1} \big( \Locus(R_1) \big)$, and by \autoref{lem:invariant_extremal_ray}(ii) we know that if $C_1'$ is any irreducible curve on $X_1'$ mapping onto $C_1$, then the class $\big[ \sum_{g\in G}g(C_1') \big]$ spans the extremal ray $R^G$.
        Therefore, due to \cite[Theorems 3.3.1, 3.3.2 and 3.4.2]{Pro21}, there exists either a $G$-equivariant divisorial contraction: 
        \begin{center}
    		\begin{tikzcd}[column sep = 2em, row sep = large]
    			(X_1',B_1') \arrow[dr, "\eta'_1" swap] \arrow[rr, "\varphi'_1 \, \coloneqq \, \theta'_1"] && (X_2',B_2') \arrow[dl, "(\eta_1')^+ \, \coloneqq \, \Id_{X_2'}"]  \\
    			& Z'_1 = X_2' ,
    		\end{tikzcd}
    	\end{center}
        or a $G$-equivariant $(K_{X_1'}+B_1')$-flip (which is a $G$-equivariant $B_1'$-flop):
        \begin{center}
    		\begin{tikzcd}[column sep = 2em, row sep = large]
    			(X_1',B_1') \arrow[dr, "\eta'_1" swap] \arrow[rr, dashed, "\varphi'_1"] && (X_2',B_2') \arrow[dl, "(\eta_1')^+"]  \\
    			& Z'_1 ,
    		\end{tikzcd}
    	\end{center}
        associated with the extremal ray $R_1^G$. Note that, in any case, the following hold: 
        \[ K_{X'_2} \equiv 0 \] 
        by \autoref{rem:G-equivariant_K-trivial_canonical_sing}(iii),
        \begin{equation}
            \label{eq:2_lifting_MMP}
            B_2'\coloneq (\varphi'_1)_* B_1' \ \text{ is } \, (\eta'_1)^+ \text{-ample}
        \end{equation}
        by construction, and $G$ acts freely in codimension one on $X_1'$, $Z_1'$ and $X_2'$ by \autoref{rem:Galois_and_quasi-etale}(ii), see also \autoref{subsection:MMP_step}. In particular, the quotient varieties $X_1' / G \simeq X_1$ and $X_2' / G$ have numerically trivial canonical divisor by \autoref{rem:singularities_quasi-etale_cover}(i), and the quotient maps 
        \[  
            \gamma_1 \colon X_1' \to X_1' / G , \quad \mu_1 \colon Z_1' \to Z_1' / G , \ \text{ and } \ \gamma_2 \colon X_2' \to X_2' / G
        \]
        are all Galois quasi-\'etale covers of the same degree $d$.
        
        We now distinguish two cases:
        
        \medskip

        \emph{Case A}: 
        If the obtained map $\varphi'_1 \colon X_1' \to X_2'$ is a $G$-equivariant divisorial contraction, by considering the quotients $X_1' / G \simeq X_1$ and $X_2' / G$ we get a unique birational morphism $\psi_1 \colon X_1' / G \to X_2' / G$ such that the following diagram commutes:
        \begin{center}
            \begin{tikzcd}[row sep = large, column sep = large]
                X_1' \arrow[d, "\gamma_1" swap] \arrow[r, "\varphi'_1"] & X_2' \arrow[d, "\gamma_2"] \\
                X_1' / G \arrow[r, "\psi_1" swap] & X_2'/G .
            \end{tikzcd}
        \end{center}
        By construction, the map
        $\psi_1 \colon X_1' / G \to X_2' / G$ contracts the extremal ray $R_1 = \R_{\geq 0}[C_1]$, so it coincides with the given divisorial contraction $\varphi_1 \colon X_1 \to X_2$ by \cite[Proposition 1.14(b)]{Debarre01book}.
        This completes the proof in this case, where we clearly also have that $\mu_1 = \gamma_2$.
        
        \medskip

        \emph{Case B}: 
        If the obtained map $\varphi'_1 \colon X_1' \dashrightarrow X_2'$ is a $G$-equivariant flip, by using the fact that both $\eta_1'$ and $(\eta_1')^+$ are $G$-equivariant and by considering the quotients $X_1' / G \simeq X_1$, $X_2' / G$ and $Z_1' / G$ we get unique birational morphisms $\zeta_1 \colon X_1' / G \to Z_1' / G$ and $\zeta_1^+ \colon X_2' / G \to Z_1' / G$ such that the following diagram commutes:
        \begin{center}
            \begin{tikzcd}
                X'_1 \arrow[dd, "\gamma_1" swap] \arrow[rr, "\varphi'_1", dashed] \arrow[dr, "\eta'_1" swap] && X'_2 \arrow[dl, "(\eta'_1)^+"] \arrow[dd, "\gamma_2"] \\
                & Z'_1 \arrow[dd, "\mu_1", pos=0.3, swap] \\
                X_1' / G \arrow[rr, "\psi_1", pos=0.65, dashed, crossing over] \arrow[dr, "\zeta_1" swap] && X_2' / G \arrow[dl, "\zeta_1^+"] \\
                & Z_1' / G ,
            \end{tikzcd}
        \end{center}
        and in this case we set $\psi_1 \coloneqq (\zeta_1)^+ \circ \zeta_1^{-1} \colon X_1' / G \dashrightarrow X_2' / G$. By construction, the map
        $\zeta_1 \colon X_1' / G \to Z_1' / G$ contracts the extremal ray $R_1 = \R_{\geq 0}[C_1]$, so it coincides with the given flipping contraction $\theta_1 \colon X_1 \to Z_1$. Moreover, since $\gamma_1$ and $\gamma_2$ are finite surjective morphisms of degree $d$, we have
        \[
            (\varphi_1)_* B_1 = \frac{1}{d} (\gamma_2)_*(B_2') ,
        \]
        which is ample over $Z_1'/G$ by \eqref{eq:2_lifting_MMP}. By uniqueness of a $(K_{X_1}+B_1)$-flip (or a $B_1$-flop), we conclude that the map $\psi_1 \colon X_1 \simeq X_1' / G \dashrightarrow X_2' / G$ coincides with $\varphi_1 \colon X_1 \dashrightarrow X_2$. This completes the proof in this case as well.
    \end{proof}
    
    \begin{prop}
        \label{prop:lifting_MMP}
        Let $\gamma_1\colon X'_1\to X_1$ be a Galois quasi-\'etale cover between normal projective varieties with numerically trivial canonical divisor. Denote by $G$ the finite subgroup of $\Aut(X'_1)$ defining $\gamma_1$ and assume that $X_1'$ is $G\Q$-factorial or, equivalently, that $X_1 \simeq X_1' /G$ is $\Q$-factorial. Let $B_1$ be an effective $\R$-divisor on $X_1$ such that $(X_1, B_1)$ is a log canonical pair, and consider the log canonical pair $(X'_1,B'_1)$, where $B'_1 \coloneqq (\gamma_1)^* B_1$. Consider a $(K_{X_1}+B_1)$-MMP:
        \begin{center}
            \begin{tikzcd}[column sep = 2em]
            (X_1,B_1) \arrow[rr, dashed, "\varphi_1"] && (X_2,B_2) \arrow[rr, dashed, "\varphi_2"] && (X_3,B_3) \arrow[rr, dashed, "\varphi_3"] && \dots 
            \end{tikzcd}
        \end{center}
        Then there exists a commutative diagram:
        \begin{center}
    		\begin{tikzcd}[column sep = 2em, row sep = large]
    			(X_1',B_1') \arrow[d, "\gamma_1" swap] \arrow[rr, dashed, "\varphi'_1"] && (X_2',B_2') \arrow[d, "\gamma_2" swap] \arrow[rr, dashed, "\varphi'_2"] && (X_3',B_3') \arrow[d, "\gamma_3" swap] \arrow[rr, dashed, "\varphi'_3"] && \dots 
    			\\ 
    			(X_1,B_1) \arrow[rr, dashed, "\varphi_1"] && (X_2,B_2) \arrow[rr, dashed, "\varphi_2"] && (X_3,B_3) \arrow[rr, dashed, "\varphi_3"] && \dots
    		\end{tikzcd}
	    \end{center}
        where, for each $i\geq 1$, the map $\gamma_i \colon X_i' \to X_i$ is a Galois quasi-\'etale cover of $X_i$, the map $\varphi'_i \colon X_i'\dashrightarrow X_{i+1}'$ is a step of a $G$-equivariant $(K_{X_i'}+B_i')$-MMP, and $B_i' = \gamma_i^* B_i = (\varphi_i')_* (B_{i-1}')$. In particular, if $\varphi_i \colon X_i \dashrightarrow X_{i+1}$ is a divisorial contraction or a flip, then $\varphi'_i \colon X_i'\dashrightarrow X_{i+1}'$ is a $G$-equivariant divisorial contraction or a $G$-equivariant flip, respectively, and the sequence on top of the above diagram is a $G$-equivariant $(K_{X_1'} + B_1')$-MMP.
    \end{prop}

    \begin{proof}
        The statement is established by applying step by step the procedure described in \autoref{lem:lifting_MMP_step}.
    \end{proof}

    Finally, according to \cite[Proposition 4.2.2]{Pro21}, if any usual MMP (that is, there is no group action involved) for $\Q$-factorial dlt pairs terminates, then any $G$-equivariant MMP for $G\Q$-factorial dlt $G$-pairs terminates as well. Using similar ideas, together with recent developments in the MMP for log canonical pairs, we prove the following statement, which refines \cite[Proposition 4.2.2]{Pro21}.
    
    \begin{prop}
        \label{prop:G-equivariant_termination}
        Let $(X_1,B_1)$ be a log canonical $G$-pair. Consider a $G$-equivariant $(K_{X_1}+B_1)$-MMP:
        \begin{center}
    		\begin{tikzcd}[column sep = 2em, row sep = large]
    			(X_1,B_1) \arrow[dr, "\eta_1" swap] \arrow[rr, dashed, "\varphi_1"] && (X_2,B_2) \arrow[dl, "\eta_1^+"] \arrow[dr, "\eta_2" swap] \arrow[rr, dashed, "\varphi_2"] && (X_3,B_3) \arrow[dl, "\eta_2^+"] \arrow[rr, dashed, "\varphi_3"] && \dots \\
    			& Z_1 && Z_2
    		\end{tikzcd}
    	\end{center}
        Then there exists a commutative diagram:
        \begin{center}
        	\begin{tikzcd}[column sep = 2em, row sep = large]
        		(X_1',B_1') \arrow[d, "f_1" swap] \arrow[rr, dashed, "\rho_1"] && (X_2',B_2') \arrow[d, "f_2" swap] \arrow[rr, dashed, "\rho_2"] && (X_3',B_3') \arrow[d, "f_3" swap] \arrow[rr, dashed, "\rho_3"] && \dots 
        		\\ 
        		(X_1,B_1) \arrow[dr, "\eta_1" swap] \arrow[rr, dashed, "\varphi_1"] && (X_2,B_2) \arrow[dl, "\eta_1^+"] \arrow[dr, "\eta_2" swap] \arrow[rr, dashed, "\varphi_2"] && (X_3,B_3) \arrow[dl, "\eta_2^+"] \arrow[rr, dashed, "\varphi_3"] && \dots 
        		\\
        		& Z_1 && Z_2
        	\end{tikzcd}
        \end{center}
        where $(X_1',B_1') = (X_1,B_1)$, $f_1 = \Id_{X_1}$, and for each $i \geq 1$,
        \begin{itemize}
            \item the map $ \rho_i \colon X_i' \dashrightarrow X_{i+1}' $ is a $(K_{X_i'}+B_i')$-MMP over $ Z_i $, and

            \item the map $f_i \colon X_i' \to X_i$ is a projective birational morphism such that 
            \[ 
                K_{X_i'} + B_i' \sim_\R f_i^* (K_{X_i} + B_i) . 
            \]
        \end{itemize}
        In particular, the sequence on top of the above diagram is a $(K_{X_1} + B_1) = (K_{X_1'} + B_1')$-MMP.
        
        Consequently, if any usual $(K_{X_1}+B_1)$-MMP terminates, then also any $G$-equivariant $(K_{X_1}+B_1)$-MMP terminates.
    \end{prop}

    \begin{proof}
        To prove the assertions, we follow the strategy of the proof of \cite[Theorem 3.2]{TX24} in the setting of usual pairs. However, instead of beginning the lifting construction in op.\ cit.\ with a dlt blow-up of $(X_1,B_1)$, we set $(X_1', B_1') \coloneqq (X_1,B_1)$ and $f_1 \coloneqq \Id_{X_1}$, and we proceed as follows.
        
        By \autoref{rem:G-equivariant_K-trivial_canonical_sing}(i)(ii) there exists a $ (K_{X_1'} + B_1') $-MMP over $Z_1$ with scaling of an ample divisor which terminates with a minimal model $ (X_2',B_2') $ of $ (X_1',B_1') $ over $Z_1$, yielding a map $ \rho_1 \colon X_1' \dashrightarrow X_2' $, and a projective birational morphism $ f_2 \colon X_2' \to X_2 $ such that 
		\begin{equation}
            \label{eq:1_G-equiv_termination}
			K_{X_2'} + B_2' \sim_\R f_2^* (K_{X_2} + B_2) .
		\end{equation}
        We obtain thus the following diagram:
        \begin{center}
        	\begin{tikzcd}[column sep = 2em, row sep = large]
        		(X_1,B_1) \arrow[d, "\Id_{X_1}" swap] \arrow[rr, dashed, "\rho_1"] && (X_2',B_2') \arrow[d, "f_2" swap] 
        		\\ 
        		(X_1,B_1) \arrow[dr, "\eta_1" swap] \arrow[rr, dashed, "\varphi_1"] && (X_2,B_2) \arrow[dl, "\eta_1^+"] \arrow[dr, "\eta_2" swap] \arrow[rr, dashed, "\varphi_2"] && (X_3,B_3) \arrow[dl, "\eta_2^+"] \arrow[rr, dashed, "\varphi_3"] && \dots 
        		\\
        		& Z_1 && Z_2
        	\end{tikzcd}
        \end{center}
        and we want to repeat the same procedure with $(X_2',B_2')$ in place of $(X_1',B_1') = (X_1,B_1)$. To this end, we first note that $(X_3,B_3)$ is the canonical model of $(X_2', B_2')$ over $Z_2$, taking \autoref{rem:G-equivariant_K-trivial_canonical_sing}(i) and \eqref{eq:1_G-equiv_termination} into account. Therefore, again by \autoref{rem:G-equivariant_K-trivial_canonical_sing}(i)(ii) there exists a $ (K_{X_2'} + B_2') $-MMP over $Z_2$ with scaling of an ample divisor which terminates with a minimal model $ (X_3',B_3') $ of $ (X_2',B_2') $ over $Z_2$, yielding a map $\rho_2 \colon X_2' \dashrightarrow X_3'$, and a projective birational morphism $ f_3 \colon X_3' \to X_3 $ such that 
        \[ 
            K_{X_3'} + B_3' \sim_\R f_3^* (K_{X_3} + B_3) .
        \]
        
        By repeating the above procedure, we eventually obtain the following diagram:
        \begin{center}
        	\begin{tikzcd}[column sep = 2em, row sep = large]
        		(X_1',B_1') = (X_1,B_1) \arrow[d, "f_1 = \Id_{X_1}" swap] \arrow[rr, dashed, "\rho_1"] && (X_2',B_2') \arrow[d, "f_2" swap] \arrow[rr, dashed, "\rho_2"] && (X_3',B_3') \arrow[d, "f_3" swap] \arrow[rr, dashed, "\rho_3"] && \dots 
        		\\ 
        		(X,B) = (X_1,B_1) \arrow[dr, "\eta_1" swap] \arrow[rr, dashed, "\varphi_1"] && (X_2,B_2) \arrow[dl, "\eta_1^+"] \arrow[dr, "\eta_2" swap] \arrow[rr, dashed, "\varphi_2"] && (X_3,B_3) \arrow[dl, "\eta_2^+"] \arrow[rr, dashed, "\varphi_3"] && \dots 
        		\\
        		& Z_1 && Z_2
        	\end{tikzcd}
        \end{center}
        where, for each $i \geq 1$, the map $ \rho_i \colon X_i' \dashrightarrow X_{i+1}' $ is a $(K_{X_i'}+B_i')$-MMP over $ Z_i $ with scaling of an ample divisor and the map $f_i \colon X_i' \to X_i$ is a projective birational morphism such that 
        $ K_{X_i'} + B_i' \sim_\R f_i^* (K_{X_i} + B_i) $.
        In particular, the sequence on top of the above diagram is a $(K_{X_1'} + B_1')$-MMP. 
        
        We have thus established the first part of the statement. The last part of the statement follows immediately from the first one and from our construction.
    \end{proof}

    \section{Symplectic varieties}
    \label{section:symplectic_varieties}

    \subsection{Definitions and basic properties}
    \label{subsection:def_symplectic}
    
    A \emph{holomorphic symplectic manifold} is a complex manifold $X$ which carries a \emph{holomorphic symplectic form}, i.e., a holomorphic $2$-form $\sigma \in H^0 (X, \Omega^2_X)$ which is closed and everywhere non-degenerate. Note that if $(X,\sigma)$ is a holomorphic symplectic manifold, then its complex dimension is even, say $\dim_\C X = 2n$, and the top exterior power $\sigma^n \in H^0 (X, \Omega^{2n}_X)$ of $\sigma$ is nowhere vanishing, 
    whence $\omega_X \simeq \OO_X$. 
    An \emph{irreducible holomorphic symplectic} (IHS) \emph{manifold} is a simply connected, compact K\"ahler manifold $X$ such that $H^0 (X, \Omega^2_X) = \C \cdot \sigma$, where $\sigma$ is a holomorphic symplectic form on $X$.
    
    \medskip
    
    In the remainder of this subsection we discuss singular analogs of the aforementioned classes of complex manifolds.

    \subsubsection{Symplectic varieties}
    
    A \emph{symplectic form} on a normal variety $X$ of dimension at least two is a reflexive $2$-form $\sigma \in H^0 \big( X, \Omega_X^{[2]} \big)$ such that $\sigma |_{X_\text{reg}}$ is closed and everywhere non-degenerate. 
    The next definition was introduced by Beauville \cite{Beauville00a}.
    
    \begin{dfn}
        \label{dfn:symplectic_variety}
        A pair $(X,\sigma)$, where $X$ is a normal variety of dimension $\dim X \geq 2$ and $\sigma$ is a symplectic form on $X$, is called a \emph{symplectic variety} if the holomorphic symplectic form $\sigma |_{X_\text{reg}}$ extends to a holomorphic $2$-form on any resolution of singularities of $X$.
    \end{dfn}

    Consider a pair $(X,\sigma)$, where $X$ is a normal variety and $\sigma$ is a symplectic form on $X$. Note that $\dim X = 2n$ for some $n \geq 1$. Observe also that $K_X \sim 0$, since $(\sigma |_{X_\text{reg}})^n$ trivializes the canonical bundle $\omega_{X_\text{reg}}$ of the regular locus $X_\text{reg}$ of $X$ and since $\codim_X \big( X \setminus X_\text{reg} \big) \geq 2$.
    To examine whether $(X,\sigma)$ is a symplectic variety, it suffices to check whether $\sigma |_{X_\text{reg}}$ extends regularly to one resolution of singularities of $X$, since any two resolutions are dominated by a common one.

    If $f \colon W \to X$ is a resolution of singularities of a symplectic variety $(X,\sigma)$ and if the extended holomorphic $2$-form $f^* \sigma$ on $W$ is symplectic, then we say that $f$ is a \emph{symplectic resolution} of $X$. For a discussion about symplectic resolutions we refer to \cite{Fu06} and \cite[Remark 3.3(2) and Example 3.4(5)]{BL21}. Here, we only recall that a resolution of singularities of a symplectic variety is symplectic if and only if it is crepant, and that symplectic resolutions do not always exist.

    If $(X,\sigma)$ is a symplectic variety, then a subvariety $L$ of $X$ is called \emph{Lagrangian} with respect to $\sigma$ if $\dim L = \frac{1}{2} \dim X$, $L \cap X_{\text{reg}} \neq \emptyset$ and $\sigma |_{L_\text{reg} \cap X_{\text{reg}}} = 0$; see \cite[Lemma 15]{Schwald20}. When $L$ is smooth, we also say that $L$ is a \emph{Lagrangian submanifold} of $X$.
    
    Now, in order to give an important characterization of symplectic varieties, we recall the following facts: According to \cite[Corollary 5.24]{KM98}, a normal variety $V$ such that $K_V$ is Cartier has rational singularities if and only if it has canonical singularities, while $V$ is Cohen--Macaulay and canonical of index $1$ if and only if it has rational Gorenstein singularities; see \cite[p.\ 286]{Reid80}. In view of the above and since rational singularities are Cohen-Macaulay by \cite[Theorem 5.10]{KM98}, we can apply \cite[Proposition 1.3]{Beauville00a} and \cite[Theorem 4]{Nam01b} to deduce the following statement: \emph{A normal variety is symplectic if and only if it has rational singularities and admits a symplectic form}. 
    Thus, symplectic varieties have canonical singularities, and hence they are not uniruled by \autoref{lem:uniruled_K-trivial}. 
    
    Finally, we recall the following well-known lemma, whose proof relies essentially on the aforementioned characterization of symplectic varieties.
    
    \begin{lem}
        \label{lem:symplectic_under_maps}
        Let $(X, \sigma)$ be a symplectic variety.
        \begin{enumerate}[\normalfont (i)]
            \item If $\gamma \colon X'\to X$ is a quasi-\'etale cover of $(X, \sigma)$, then $(X', \sigma')$ is also a symplectic variety, where $\sigma' \coloneqq \gamma^{[*]} \sigma \in H^0 \big( X',\Omega_{X'}^{[2]} \big)$.

            \item If $f \colon X \to X'$ is a proper birational morphism to a normal variety $X'$, then $X'$ is also a symplectic variety.

            \item If $g \colon X \dashrightarrow X''$ is a birational contraction to a normal variety $X''$ with rational singularities, then $X''$ is also a symplectic variety.
        \end{enumerate}
    \end{lem}
    
    \subsubsection{Irreducible symplectic varieties}
    
    The following subclass of symplectic varieties was introduced by Greb, Kebekus and Peternell \cite{GKP16b}.
    
    \begin{dfn}
        \label{dfn:IS_variety}
        A projective symplectic variety $(X,\sigma)$ is called \emph{irreducible} if for every quasi-\'etale cover $\gamma \colon W \to X$ of $X$, in particular for $X$ itself, the exterior algebra of global reflexive forms on $W$ is generated by the reflexive pullback $\gamma^{[*]} \sigma$:
		\[ \bigoplus_{p=0}^{2n} H^0 \big( W, \Omega_W^{[p]} \big) = \C [\gamma^{[*]} \sigma] . \]
    \end{dfn}

    Irreducible symplectic varieties appear as factors in the singular version of the Beauville-Bogo\-molov decomposition theorem, which was established by H\"oring and Peternell \cite{HoerPet19}, building on earlier works of Druel \cite{Druel18} and Greb, Guenancia and Kebekus \cite{GGK19}. They are thus one of the three building blocks of normal projective varieties with klt singularities and numerically trivial canonical divisor.
    
    We collect below some basic properties of irreducible symplectic varieties, which were established in \cite{GKP16b,GGK19}.
    
    \begin{prop}
        \label{prop:basic_prop_ISV}
        Let $X$ be an irreducible symplectic variety of dimension $2n$. The following statements hold:
        \begin{enumerate}[\normalfont (i)]
            \item $\widetilde{q}(X)=0$.

            \item $\pi_1(X) = \{ 1 \}$.
            
            \item For any quasi-\'etale cover $\gamma \colon X'\to X$ of $X$, the variety $X'$ is also an irreducible symplectic variety.

            \item The tangent sheaf of $X$ is strongly stable in the sense of \cite[Definition 7.2]{GKP16b}.

            \item $\chi(X, \OO_X) = n+1$.
        \end{enumerate}
    \end{prop}

    \begin{proof}
        Parts (i), (ii) and (iv) of the statement are \cite[Remark 8.17]{GKP16b}, \cite[Corollary 13.3]{GGK19} and \cite[Proposition 8.20]{GKP16b}, respectively. Part (iii) (resp.\ part (v)) follows immediately from \autoref{lem:symplectic_under_maps}(i) (resp.\ \autoref{prop:Hodge_duality}) and the definition of an irreducible symplectic variety.
    \end{proof}

    \subsubsection{Primitive symplectic varieties}
    
    Another subclass of symplectic varieties was defined more recently by Schwald \cite{Schwald20}, albeit first studied by Namikawa and Matsushita under stronger hypotheses. They were originally called \enquote{irreducible} \cite{Schwald20}, but are currently called \enquote{primitive} \cite{BL22}. We also stick to the latter terminology in this paper and, in what follows, we recall their definition and gather some basic properties.
    
    \begin{dfn}
        \label{dfn:PSV}
        A projective symplectic variety $(X,\sigma)$ is called \emph{primitive} if 
        \[ H^1(X,\OO_X) = 0 \ \text{ and } \ H^0 \big( X, \Omega_X^{[2]} \big) = \C \cdot \sigma . \]
    \end{dfn}


    \begin{prop}
        \label{prop:basic_prop_PSV}
        Let $X$ be a primitive symplectic variety. The following statements hold:
        \begin{enumerate}[\normalfont (i)]
            \itemsep 2pt
            
            \item 
            $\Pic(X)$ is finitely generated.

            \item 
            $H^0 \big( X,\Omega_X^{[1]} \big) = \{ 0 \} = H^0 \big( X,\Omega_X^{[\dim X - 1]} \big)$.

            \item 
            $H^0 \big( X,\Omega_X^{[\dim X - 2]} \big) = \C \cdot \sigma^{\dim X -1}$.

            \item 
            $H^0(X,\mathscr{T}_X) = \{ 0 \}$, where $\mathscr{T}_X$ denotes the tangent sheaf of $X$.
            
            \item 
            $\Aut(X)$ is a discrete group.
        \end{enumerate}
    \end{prop}

    \begin{proof}
        Part (i) is an easy consequence of 
        the facts that $q(X) = 0$ and that the symplectic variety $X$ has rational singularities, see \cite[Proposition 1.3]{Beauville00a}. Parts (ii) and (iii) follow from the definition of primitive symplectic varieties and \cite[Proposition 6.9 and Corollary 6.10]{GKP16b}. Part (iv) is \cite[Lemma 4.6]{BL22}, while part (v) follows from part (iv) and \cite[Lemma 3.4]{MO67}.
    \end{proof}
    
    \begin{rem}
        \label{rem:irreducible_vs_primitive_symplectic}
        We compare here Definitions \ref{dfn:IS_variety} and \ref{dfn:PSV}.
        \begin{enumerate}[(i)]
            \item In the smooth case, being irreducible symplectic is equivalent to being primitive symplectic, since then both notions recover the notion of an IHS manifold by \cite[Theorem 1]{Schwald22}; see also \cite[Remark 8.19]{GKP16b} and \cite[Remark 2]{Schwald22} for a proof of this assertion for smooth irreducible symplectic varieties without invoking \cite[Theorem 1]{Schwald22}.
        
            \item In the singular setting, irreducible symplectic varieties are primitive symplectic by \autoref{prop:Hodge_duality}, but the converse does not hold: Examples \ref{example:symmetric_product_K3s} and \ref{example:singular_Kummer_surface} below and \cite[Example 26]{Schwald20} constitute examples of singular symplectic varieties which are primitive, but not irreducible.

            \item In contrast with the case of irreducible symplectic varieties, a quasi-\'etale cover of a primitive symplectic variety is not necessarily a primitive symplectic variety; see Examples \ref{example:symmetric_product_K3s} and \ref{example:singular_Kummer_surface}.
        \end{enumerate}
    \end{rem} 
    
    To the best of our knowledge, all known examples of primitive symplectic varieties are simply connected. For instance, this holds for all irreducible symplectic varieties by \autoref{prop:basic_prop_ISV}(ii), as well as for the symplectic torus quotients from \cite[Lemma 5]{Schwald22}, see \cite[Lemma 1.2]{Fujiki83b} and \cite[Section 6]{Schwald22}. We would like to thank Martin Schwald for the relevant discussions about his result.

    \begin{exa}
        \label{example:symmetric_product_K3s}
        Fix a projective K3 surface $S$ and an integer $n \geq 2$. Set $S^n \coloneqq S \times \cdots \times S$ and denote by $\mathfrak{S}_n$ the symmetric group of degree $n$. By \cite[Section 6]{Beauville83a}, the punctual Hilbert scheme $S^{[n]} \coloneqq \Hilb^n(S)$, which is a projective IHS manifold of dimension $2n$, yields a crepant
        resolution of singularities $S^{[n]} \to S^{(n)}$ of the symmetric product $S^{(n)} \coloneqq \Sym^n(S) \coloneqq S^n / \mathfrak{S}_n$, which is thus a primitive symplectic variety.
        However, $S^{(n)}$ is \emph{not} an irreducible symplectic variety, since it admits a quasi-\'etale cover $S^n \to S^{(n)}$ such that $h^0 \big( S^n,\Omega^2_{S^n} \big) = n \neq 1$.
        
        Since $S^{[n]}$ is simply connected, by Remarks \ref{rem:Galois_and_quasi-etale}(i) and \ref{rem:augmented_irreg}(iii) we obtain 
        \[ 
            \widetilde{q} \big( S^{(n)} \big) = \widetilde{q} \big( S^n \big) = q \big( S^n \big) = 0 .
        \]
        Moreover, $S^{(n)}$ is simply connected by \cite[Section 6, Lemma 1]{Beauville83a}. On the other hand, the regular locus of $S^{(n)}$ is not simply connected: we have 
        the short exact sequence
        \[ 
            1 \to \pi_1(S^n) \to \pi_1 \big( {S^{(n)}}_\text{reg} \big) \to \mathfrak{S}_n \to 1 , 
        \]
        whence $\pi_1 \big( {S^{(n)}}_\text{reg} \big) \simeq \mathfrak{S}_n$.
        We refer to \cite[p.\ 156 (symmetric products of K3 surfaces), p.\ 164 (paragraph after the proof of Proposition 3), and Proposition 5]{Perego20} for more information about this example.
    \end{exa}
    
    \begin{exa}
        \label{example:singular_Kummer_surface}
        Consider an abelian surface $A$ and the singular Kummer surface $S \coloneqq A / \{ \pm \Id_A \}$, which has canonical singularities. 
        Recall that the minimal resolution $S' \to S$ of $S$ is a K3 surface $S'$. Therefore, $S$ is a primitive symplectic variety, but we have
        \[ 
            \widetilde{q}(S) = \widetilde{q}(A) = 2 ,
        \]
        see \cite[Lemma 2.19]{GGK19},
        and thus $S$ is \emph{not} an irreducible symplectic variety due to \autoref{prop:basic_prop_ISV}(i).
        
        Furthermore, since the K3 surface $S'$ is simply connected, we have $\widetilde{q}(S') = q(S') = 0$ by \autoref{rem:Galois_and_quasi-etale}(i), and $S$ is also simply connected by \cite[Theorem 1.1]{Tak03}. On the other hand, the regular locus of $S$ is not simply connected: the fundamental group of $S_\text{reg}$ is infinite, since there is a short exact sequence
        \[ 
            1 \to \Z^4 \to \pi_1 \big( S_\text{reg} \big) \to \Z / 2\Z \to 1 . 
        \]
    \end{exa}
    
    Next, we prove a criterion for a projective symplectic variety to be \emph{primitive}, which is similar to \cite[Proposition 3.15]{BCGPSV26}. This result plays a key role in the paper and slightly generalizes the recent \cite[Proposition 3.14]{BCGPSV26}.
    
    \begin{prop}
        \label{prop:dom_rational_map_from_PSV}
        If there exists a dominant rational map $P \dashrightarrow X$ from a primitive symplectic variety $P$ to a projective symplectic variety $X$, then $X$ is a primitive symplectic variety as well.
    \end{prop}

    \begin{proof}
        The proof is similar to the one of \cite[Proposition 3.14]{BCGPSV26}, but we provide all the details for the reader's convenience. We also use tacitly but repeatedly the fact that symplectic varieties have rational singularities, see \cite[Proposition 1.3]{Beauville00a}.
        
        Consider first a resolution of singularities $\rho \colon Y \rightarrow X$ of $X$ and then a resolution of indeterminacies $(r,s) \colon Q \to P \times Y$ of the induced dominant rational map $P \dashrightarrow Y$ such that $Q$ is smooth. We obtain thus the following commutative diagram:
        \begin{center}
            \begin{tikzcd}[row sep = normal]
                Q \arrow[d, "r" swap] \arrow[r, "s"] & Y \arrow[d, "\rho"] \\
                P \arrow[r, dashed] & X .
            \end{tikzcd} 
        \end{center}
        Since $Q$ is by construction birational to the primitive symplectic variety $P$, 
        we obtain 
        \begin{equation}
            \label{eq:1_dom_rational_map_from_PSV}
            q(Q) = h^1(Q,\OO_Q) = h^1(P, \OO_P) = 0
        \end{equation}
        by \autoref{rem:augmented_irreg}(i) and 
        \begin{equation}
            \label{eq:2_dom_rational_map_from_PSV}
            h^0 \big( Q,\Omega^{2}_Q \big) = h^2(Q, \OO_Q) = h^2(P, \OO_P) = h^0 \big( P,\Omega^{[2]}_P \big) = 1
        \end{equation}
        by \autoref{prop:Hodge_duality}. Since the morphism $s \colon Q \to Y$ is surjective, by \cite[Lemma 7.28]{Voisin07book_I} and the preceding paragraph in op.\ cit.\ we obtain $q(Y) = h^1(Y,\OO_Y) = 0$ due to \eqref{eq:1_dom_rational_map_from_PSV}, and hence 
        \[ q(X) = h^1(X,\OO_X) = h^1(Y,\OO_Y) = 0 \]
        by \autoref{rem:augmented_irreg}(i).
        Finally, since the composite map $\rho \circ s \colon Q \to X$ is dominant and since $X$ admits a symplectic form by assumption, it follows from \cite[Proposition 5.8]{Kebekus13a} and \eqref{eq:2_dom_rational_map_from_PSV} that 
        \[ 1 \leq h^0 \big( X,\Omega^{[2]}_X \big) \leq  h^0 \big( Q,\Omega^{2}_Q \big) = 1 , \] 
        whence $h^0 \big( X,\Omega^{[2]}_X \big) = 1$. This finishes the proof.
    \end{proof}

    We conclude this subsection by mentioning two immediate consequences of the previous result. Firstly, \autoref{prop:dom_rational_map_from_PSV} together with \autoref{lem:symplectic_under_maps}(ii) imply that the class of primitive symplectic varieties is preserved under birational morphisms; namely, if $X \to X'$ is a birational morphism from a primitive symplectic variety $X$ to a normal projective variety $X'$, then $X'$ is also a primitive symplectic variety. Secondly, \autoref{prop:dom_rational_map_from_PSV} shows that if a primitive symplectic variety $X$ admits a symplectic resolution $W \to X$, then it is simply connected by \cite[Theorem 1.1]{Tak03}, since $W$ is an IHS manifold in this case.

    \subsection{Finite group actions on primitive symplectic varieties}
    \label{subsection:finite_group_actions}

    We first recall the following standard definitions.
    
    \begin{dfn}
        Let $(X,\sigma)$ be a primitive symplectic variety.
        \begin{enumerate}[(i)]
            \item An automorphism $g \in \Aut(X)$ is called \emph{symplectic} if $g^{[*]} \sigma = \sigma$; otherwise, it is called \emph{non-symplectic}. 

            \item A non-symplectic automorphism $g \in \Aut(X)$ is called an \emph{anti-symplectic involution} if $g^2 = \Id_X$ and $g^{[*]} \sigma = - \sigma$.
              
            \item A non-symplectic automorphism $g \in \Aut(X)$ is called \emph{purely non-symplectic} if there exists a primitive root of unity $\zeta \neq 1$ of the same order as $g$ such that $g^{[*]} \sigma = \zeta \sigma$.

            \item We say that a finite subgroup $G$ of $\Aut(X)$ acts \emph{symplectically} (resp.\ \emph{non-symplec\-tically}) on $X$ if every element $g \in G \setminus \{\Id\}$ is a symplectic (resp.\ non-symplectic) automorphism of $X$.
        \end{enumerate}
    \end{dfn}

    In the remainder of this subsection we study non-symplectic finite group actions on primitive symplectic varieties. We begin with some general observations.
    Let $(X,\sigma)$ be a primitive symplectic variety and let $G$ be a finite subgroup of $\Aut(X)$. Consider the natural map
    \[
        \rho\colon \Aut(X)\to \Aut \left( H^0 \big( X,\Omega_{X}^{[2]} \big) \right), \ g \mapsto g^{[*]}.
    \]
    Due to the facts that $H^0 \big( X,\Omega_{X}^{[2]} \big) \simeq \C$ and $\GL(1,\C)\simeq \C^*$, for any $g\in\Aut(X)$ there exists a root of unity $\zeta$ such that $g^{[*]} \sigma = \zeta \sigma$. We obtain thus a group homomorphism
    $\Aut(X) \to \C^* $, which induces a surjective group homomorphism
    \[ \rho_G \colon G \twoheadrightarrow \mu_m , \]
    where $m \coloneqq |\Ima(\rho_G)|$ and $\mu_m$ is a cyclic group 
    generated by a primitive $m$-th root of unity. By construction, $G_\text{sym} \coloneqq \ker(\rho_G)$ consists of all symplectic automorphisms of $X$ that are contained in $G$, and clearly there is a short exact sequence
    \[ 0 \longrightarrow G_\text{sym} \longrightarrow G \longrightarrow \mu_m \longrightarrow 0 . \]
    
    The previous observations lead to the next useful lemma.
    
    \begin{lem}
        \label{lem:group_acting_nonsymplectically_cyclic}
        Let $X$ be a primitive symplectic variety and let $G$ be a finite subgroup of $\Aut(X)$. Then $G$ acts non-symplectically on $X$ if and only if $G$ is a cyclic group generated by a purely non-symplectic automorphism of $X$.
    \end{lem}
    
    \begin{proof}
        One direction is clear, so we prove the converse direction. Specifically, assume that $G$ acts non-symplectically on $X$, consider the induced surjective group homomorphism 
        $ \rho_G \colon G \twoheadrightarrow \mu_m $ as above, and note that $\ker (\rho_G) = \{ \Id_X \}$ by assumption. Hence, $\rho_G$ is an isomorphism, and the assertion follows.
    \end{proof}

    In the special case that $g \in \Aut(X)$ is an anti-symplectic involution, its fixed locus $\Fix(g)$ is a finite union of Lagrangian subvarieties of $X$.  Indeed, the proof given in \cite[Lemma 1]{Beauville11} works on $X_{\text{reg}} \cap \Fix(g)$. 
    We now generalize Beauville's argument in order to investigate the non-free locus of a non-symplectic group action. 
    
    \begin{lem}
        \label{lem:prime_order_non-symplectic}
        Let $X$ be a primitive symplectic variety of dimension $2n \geq 4$. If $g \in \Aut(X)$ is a non-symplectic automorphism of prime order, then $\codim_X \big( \Fix(g) \big) \geq 2$.
    \end{lem}

    \begin{proof}
        We may assume that $\Fix(g) \neq \emptyset$, since otherwise we are done, and we may also restrict to the regular locus of $X$, since $\codim_X (X_\text{sing}) \geq 2$. In other words, we need to show that $\codim_{X_\text{reg}} \big( \Fix(g) \cap X_{\text{reg}} \big) \geq 2$. 
        To this end, pick an irreducible component $F$ of $\Fix(g) \cap X_{\text{reg}}$. If it is zero-dimensional, then we are done, so we may assume that it is positive-dimensional. Consider the natural inclusion $j \colon F \hookrightarrow X_\text{reg}$ and fix a point $x \in F_\text{reg}$.
        Note that $g$ acts on the tangent space $T_x X$ of $X$ at the point $x$ via its differential $dg_x \in \Aut(T_x X)$.  
        Since $(g |_{X_\text{reg}}) \circ j = \Id_{F}$, we infer that $(dg_x) |_{T_x F} = \Id_{T_x F}$. It follows that $1$ is an eigenvalue of $dg_x$ with corresponding eigenspace $W$,
        and we have 
        \[ 
            T_x F \subseteq W \subseteq T_x X \simeq T_x (X_\text{reg}) , 
        \] 
        so the local dimension of $F$ at $x$ is bounded above by $\dim_\C W$. We will next estimate this dimension.
        Denote by $\sigma$ the symplectic form on $X_{\text{reg}}$ and by $\sigma_x$ the induced symplectic form on the tangent space $T_x X$. By assumption we have $g^{[*]} \sigma = \zeta_p \, \sigma$, where $p \coloneqq \ord(g)$ and $\zeta_p$ is a primitive $p$-th root of unity. In particular, for any $u,v \in W$ we obtain
        \[ \big( \zeta_p \, \sigma_x \big) (u,v) = \big( g^* \sigma_x \big) (u,v) = \sigma_x \big( dg_x(u), \, dg_x(v) \big) = \sigma_x(u,v) , \]
        which implies that $\sigma_x(u,v)=0$. Therefore, $W$ is isotropic; in particular, we have
        $ \dim_\C W \leq \frac{\dim X}{2} $. Since $\dim X = 2n \geq 4$, we infer that $F$ 
        has codimension at least two in $\Fix(g) \cap X_{\text{reg}}$. In conclusion,
        $ \codim_{X_\text{reg}} \big( \Fix(g) \cap X_{\text{reg}} \big) \geq 2 $, as desired.
    \end{proof}
   
    \begin{prop}
        \label{prop:group_acting_nonsymplectically_free_codim_one}
        Let $X$ be a primitive symplectic variety of dimension $2n \geq 4$. If $g \in \Aut(X)$ is a purely non-symplectic automorphism, then $G \coloneqq \langle g \rangle$ acts on $X$ freely in codimension one.
    \end{prop}

    \begin{proof}
        We have to show that $\codim_X (X^\text{nf}) \geq 2$. We may assume that $X^\text{nf} \neq \emptyset$, since otherwise we are done. We now claim that for any $x \in X^\text{nf}$ there exists a non-symplectic automorphism $h \in G$ of prime order such that $x \in \Fix(h)$. Indeed, given $x \in X^\text{nf}$, we have $m \coloneqq |G_x| >1$. Since $G = \langle g \rangle$ is a cyclic group, its subgroup $G_x$ is also a cyclic group, generated by $g^{\ord(g)/m}$. By picking any prime divisor $p$ of $m$, we infer that 
        \[ 
            h \coloneqq g^{\ord(g) / p} = \big(g^{\ord(g) / m}\big)^{m/p} \in \Aut(X) 
        \]
        is an automorphism of prime order $p$ which fixes the point $x$; this proves the claim.
        It follows now that
        \[ 
            X^\text{nf} = \bigcup_h \Fix(h) , 
        \]
        where the union is taken over all prime order non-symplectic automorphisms $h \in G$. As this is clearly a finite union, \autoref{lem:prime_order_non-symplectic} yields $\codim_X (X^\text{nf}) \geq 2$.
    \end{proof}

    \subsection{MMP for symplectic pairs}
    \label{subsection:MMP_symplectic}
    
    In this short subsection we first demonstrate that being primitive symplectic is preserved under both the usual and the equivariant MMP, since this is essential for our purposes. At the end we show that any equivariant MMP starting from an IHS pair terminates, see \autoref{prop:IHS_G-equivariant_termination}. This result is crucial for the proof of \autoref{thm:termination_MMP_Enriques}.
    
    \begin{lem}
        \label{lem:MMP_step_symplectic}
        Let $(X,B)$ be a projective log canonical pair. Assume that $K_X + B$ is not nef and consider a step of a $(K_X+B)$-MMP:
        \begin{center}
            \begin{tikzcd}
                (X, B) \arrow[dr, "\eta" swap] \arrow[rr, "\varphi", dashed] && (X^+, B^+) \arrow[dl, "\eta^+"] \\
                & Z
            \end{tikzcd}
        \end{center}
        where $B^+ = \varphi_* B$. The following statements hold:
        \begin{enumerate}[\normalfont (i)]
            \item If $X$ is a symplectic variety, then $Z$ and $X^+$ are also symplectic varieties.
            
            \item If $X$ is a primitive symplectic variety, then $Z$ and $X^+$ are also primitive symplectic varieties.

            \item If $X$ is a projective IHS manifold and if the given $(K_X+B)$-MMP step is a $B$-flop, then $X^+$ is also a projective IHS manifold.
        \end{enumerate}
    \end{lem}
    
    \begin{proof}
        Observe that all varieties involved are projective by construction. Part (i) follows from \autoref{lem:symplectic_under_maps}(ii)(iii), taking also the second paragraph of \autoref{subsection:MMP_step} into account. Part (ii) is an immediate consequence of (i) and \autoref{prop:dom_rational_map_from_PSV}. Part (iii) follows from (ii) and \cite[Corollary 1]{Nam06}, taking into account \cite[Lemma 3.38]{KM98}, \cite[Proposition 4.8.20]{Fuj17book} and \autoref{rem:contraction_from_K-trivial}.
    \end{proof}
    
    \begin{lem}
        \label{lem:G-equivariant_MMP_step_symplectic}
        Let $(X,B)$ be a projective log canonical $G$-pair, where $G$ is a finite subgroup of $\Aut(X)$. Assume that $K_X+B$ is not nef and consider a step of a $G$-equiva\-riant $(K_X+B)$-MMP:
        \begin{center}
            \begin{tikzcd}
                (X, B) \arrow[dr, "\eta" swap] \arrow[rr, "\varphi", dashed] && (X^+, B^+) \arrow[dl, "\eta^+"] \\
                & Z
            \end{tikzcd}
        \end{center}
        where $B^+ = \varphi_* B$. The following statements hold:
        \begin{enumerate}[\normalfont (i)]
            \item If $X$ is a symplectic variety, then $Z$ and $X^+$ are also symplectic varieties.
            
            \item If $X$ is a primitive symplectic variety, then $Z$ and $X^+$ are also primitive symplectic varieties.

            \item Assume now that $X$ is a primitive symplectic variety. If $G$ acts non-symplectically (resp.\ freely in codimension one) on $X$, then it also acts non-symplectically (resp.\ freely in codimension one) on both $Z$ and $X^+$.
        \end{enumerate}
    \end{lem}

    \begin{proof}
        Parts (i) and (ii) are proved by arguing as in the proofs of parts (i) and (ii) of \autoref{lem:MMP_step_symplectic}, respectively, taking also \autoref{rem:G-equivariant_K-trivial_canonical_sing}(iii) into account for the former. 
        Now, regarding part (iii), we have already explained in \autoref{subsection:MMP_step} that if $G$ acts freely in codimension one on $X$, then it also does so on both $Z$ and $X^+$. 
        We thus assume henceforth that $G$ acts non-symplectically on $X$. Then $G = \langle g \rangle$ by \autoref{lem:group_acting_nonsymplectically_cyclic},
        where $g \in \Aut(X)$ is a purely non-symplectic automorphism of order $m \geq 2$. Since $X$ is a primitive symplectic variety with symplectic form $\sigma$, by (ii) we know that both $Z$ and $X^+$ are also primitive symplectic varieties with induced symplectic forms $\sigma_Z$ and $\sigma^+$, respectively, which agree on nonempty, dense, open subsets by construction. Since 
        \[ H^0 \big( X,\Omega_{X}^{[2]} \big) = \C \cdot \sigma , \quad H^0 \big( Z,\Omega_{Z}^{[2]} \big) = \C \cdot \sigma_Z , \quad H^0 \big( X^+,\Omega_{X^+}^{[2]} \big) = \C \cdot \sigma^+ , \]
        and since by assumption we have $g^{[*]} \sigma = \zeta \sigma $, where $\zeta$ is a primitive $m$-th root of unity, we conclude that $G = \langle g \rangle$ acts non-symplectically on both $Z$ and $X^+$.
    \end{proof}
    
    \begin{prop}
        \label{prop:IHS_G-equivariant_termination}
        Let $(X,B)$ be a log canonical $G$-pair, where $X$ is a projective IHS manifold and $G$ is a finite subgroup of $\Aut(X)$. Then any $G$-equivariant $(K_X+B)$-MMP terminates with a $G$-minimal model $(X^m, B^m)$ of $(X,B)$, where $X^m$ is a $G\Q$-factorial primitive symplectic variety and $B^m$ is a nef, effective, $G$-invariant, $\R$-Cartier $\R$-divisor on $X^m$.
    \end{prop}
    
    \begin{proof}
        By \autoref{thm:termination_MMP_symplectic} and \autoref{prop:G-equivariant_termination} we infer that any $G$-equivariant $(K_X+B)$-MMP terminates with a $G$-minimal model $(X^m, B^m)$ of $(X,B)$. The assertions about the variety $X^m$ and the $\R$-divisor $B^m$ follow from \autoref{lem:G-equivariant_MMP_step_symplectic}(ii) and the discussion in \autoref{subsection:MMP_step}.
    \end{proof}

    \section{Enriques varieties}
    \label{section:Enriques_varieties}
    
    \subsection{Definitions and basic properties}
    \label{subsection:def_Enriques}
    
    An \emph{Enriques manifold} is a connected complex manifold $Y$ which is not simply connected and whose universal covering $X$ is an IHS manifold, see \cite[Definition 2.1]{OS11a}. Thus, $Y$ is compact, of even dimension $\dim Y = 2n$, and with finite fundamental group.
    In fact, according to \cite[Proposition 2.4]{OS11a}, $\pi_1(Y)$ is a cyclic group whose order $d \geq 2$, called the \emph{index} of $Y$, is a divisor of $n+1$. 
    By \cite[Corollary 2.7]{OS11a} both the Enriques manifold $Y$ and its universal covering $X$ are necessarily projective. Finally, according to \cite[Proposition 2.8]{OS11a}, the torsion subgroup $\Pic^{\text{tor}}(Y)$ of the Picard group of $Y$ is a finite cyclic group of order $d = |\pi_1(Y)| \geq 2$, generated by the canonical bundle $\omega_Y \in \Pic(Y)$.
    
    Note that there exist Enriques manifolds of index $d \in \{2, 3, 4\}$, see \cite{BNWS11,OS11a}. In particular, Enriques surfaces are Enriques manifolds of dimension $2$ and index $2$. For two higher-dimensional examples of index $2$ we also refer to \autoref{example:PEV_from_Catanese}.
    
    \medskip
    
    In the remainder of this subsection we aim to define and to investigate singular analogs of Enriques manifolds. To this end, we first need to discuss a class of varieties defined by Greb, Guenancia and Kebekus \cite[Remark 13.5]{GGK19}. To quote the authors, one might call these \enquote{Enriques varieties}, but we coin the term \enquote{GGK varieties} instead for reasons that will become apparent below.

    \subsubsection{GGK varieties}
    
    \begin{dfn}
        \label{dfn:GGK_variety}
        A \emph{GGK variety} is a normal projective variety $Y$ with klt singularities and numerically trivial canonical divisor which admits a quasi-\'etale cover $\gamma \colon X \to Y$ such that $X$ is an irreducible symplectic variety.
    \end{dfn}

    
    
    Observe that the class of GGK varieties contains the class of irreducible symplectic varieties. In particular, we have the following result in the smooth case.
    
    \begin{lem}
        \label{lem:smooth_GGK}
        Let $Y$ be a GGK variety and assume that $Y$ is smooth. The following statements hold:
        \begin{enumerate}[\normalfont (i)]
            \item If $\pi_1(Y) = \{ 1 \}$, then $Y$ is an IHS manifold.

            \item If $\pi_1(Y) \neq \{ 1 \}$, then $Y$ is an Enriques manifold.
        \end{enumerate}
    \end{lem}    

    \begin{proof}
        Let $\gamma \colon X\to Y$ be a quasi-\'etale cover of $Y$ such that $X$ is an irreducible symplectic variety. By \autoref{rem:Galois_and_quasi-etale}(i) and the assumption that $Y$ is smooth we infer that $\gamma$ is \'etale, so $X$ is smooth, 
        and thus an IHS manifold; see \autoref{rem:irreducible_vs_primitive_symplectic}(i). Since $\pi_1(X) = \{ 1 \}$, $\gamma$ is the universal covering of $Y$, and the assertions follow.
    \end{proof}
    
    Our first goal is to identify the complementary class to the one of irreducible symplectic varieties in the (larger) class of GGK varieties, using primarily results from \cite{GKP16a,GKP16b,GGK19}.
    
    \begin{notation}
        We denote by $\Hol^\circ (Z)$ (resp.\ $\Hol(Z)$) the \emph{restricted holonomy group} (resp.\ the \emph{holonomy group}) of a normal projective variety $Z$ with klt singularities and numerically trivial canonical divisor, referring to \cite{GGK19} for further details.
    \end{notation}

    According to \cite[Proposition F]{GGK19}, a normal projective variety $X$ of dimension $2n$ with klt singularities and numerically trivial canonical divisor is \emph{irreducible symplectic} if and only if $\Hol(X)$ is isomorphic to the \emph{unitary symplectic group} $\Sp(n)$. More generally, the following result from \cite{GGK19} characterizes GKK varieties in terms of holonomy.
    
    \begin{prop}
        \label{prop:characterization_GGK_holonomy}
        Let $Y$ be a normal projective variety of dimension $2n$ with klt singularities and numerically trivial canonical divisor. Then $Y$ is a GGK variety if and only if $\Hol^\circ(Y) \simeq \Sp(n)$.
    \end{prop}

    \begin{proof}
        Assume first that $Y$ is a GGK variety and consider a quasi-\'etale cover $X \to Y$ of $Y$ such that $X$ is an irreducible symplectic variety. It follows immediately from \cite[Remark 5.2 and Proposition F, (F.4)--(F.6)]{GGK19} that
        \[ \Hol(Y)^\circ = \Hol(X)^\circ = \Hol(X) \simeq \Sp(n) . \]
        
        Assume now that $\Hol^\circ(Y) \simeq \Sp(n)$. By \cite[Theorem 12.1]{GGK19} the tangent sheaf of $Y$ is strongly stable and there exists a quasi-\'etale cover $\gamma \colon X \to Y$ of $Y$ such that 
        \[ \Hol(X) = \Hol^\circ(X) = \Hol^\circ(Y) \simeq \Sp(n) . \]
        Therefore, $X$ is an irreducible symplectic variety by \cite[Proposition F, (F.4)--(F.6)]{GGK19}, and hence $Y$ is a GGK variety by definition.
    \end{proof}
    
    We collect below some basic properties of GGK varieties, which were essentially established in \cite{GKP16b,GGK19}. 
    
    \begin{prop}
        \label{prop:basic_prop_GGK}
        Let $Y$ be a GGK variety of dimension $2n$. The following statements hold:
        \begin{enumerate}[\normalfont (i)]
            \item $\widetilde{q}(Y) = 0$.
            
            \item $\pi_1(Y)$ is finite.
            
            \item For any quasi-\'etale cover $\gamma \colon Y' \to Y$ of $Y$, the variety $Y'$ is also a GGK variety.
            
            \item The tangent sheaf of $Y$ is strongly stable in the sense of \cite[Definition 7.2]{GKP16b}.

            \item $1 \leq \chi(Y, \OO_Y) \leq n+1$.
        \end{enumerate}
    \end{prop}

    \begin{proof}~

        \medskip

        \noindent (i) 
        Follows immediately from \autoref{rem:augmented_irreg}(iii) and \autoref{prop:basic_prop_ISV}(i).
        
        \medskip

        \noindent (ii) 
        Let $X \to Y$ be a quasi-\'etale cover of $Y$ such that $X$ is an irreducible symplectic variety. By \cite[Proposition 1.3]{Campana91} the image of $\pi_1(X)$ in $\pi_1(Y)$ has finite index. Since $\pi_1(X) = \{ 1 \}$ by \autoref{prop:basic_prop_ISV}(ii), we infer that $\pi_1(Y)$ is finite.

        \medskip

        \noindent (iii) 
        Follows from \autoref{prop:characterization_GGK_holonomy}, taking \cite[Reminder 2.14 and Remark 5.2]{GGK19} into account.
        %
        
        \medskip
        
        \noindent (iv) 
        This was shown earlier in the proof of \autoref{prop:characterization_GGK_holonomy}.
        
        \medskip
        
        \noindent (v) 
        By part (iv) and \cite[Theorem 12.2 (12.2.2)]{GGK19} we obtain
        \[
            h^0 \big( Y,\Omega_Y^{[p]} \big) \leq 
            \begin{cases}
                1 , \text{ if } 0 \leq p \leq 2n \text{ and if }p\text{ is even,}\\
                0 , \text{ otherwise} . 
            \end{cases}
        \]
        This, together with \cite[Proposition 6.9]{GKP16b}, yield the assertion.
    \end{proof}
    
    We would like to thank Mirko Mauri for pointing out to us the following result, which gives a practical description of GGK varieties.
    
    \begin{lem}
        \label{lem:GGK_cover}
        If $Y$ is a GGK variety, then there exists a quasi-\'etale cover $\gamma \colon X \to Y$ such that $X$ is an irreducible symplectic variety and $Y\simeq X/G$, where $G\subseteq \Aut(X)$ is a finite cyclic group acting non-symplectically on $X$, unless $Y$ itself is an irreducible symplectic variety and in this case we have $G=\{ \Id_X \}$.
    \end{lem}
    
    \begin{proof}
        By definition there exists a quasi-\'etale cover $\gamma\colon X\to Y$ of $Y$ such that $X$ is an irreducible symplectic variety. By the existence of Galois closures, see \cite[Theorem 3.7]{GKP16a}, there exists a Galois quasi-\'etale cover $\widetilde{\gamma}\colon \widetilde{X}\to X$ of $X$ such that the composite map $\mu\coloneqq \gamma \circ \widetilde{\gamma} \colon \widetilde{X}\to Y$ is a Galois quasi-\'etale cover of $Y$. Note that $\widetilde{X}$ is also an irreducible symplectic variety by \autoref{prop:basic_prop_ISV}(iii).
        Denote by $G \subseteq \Aut(\widetilde{X})$ the finite group defining $\mu$, consider the induced surjective group homomorphism $\rho_G\colon G\to \mu_m$ and set $ G_\text{sym} \coloneqq \ker(\rho_G)$, see \autoref{subsection:finite_group_actions}.
        If $\pi_{G_\text{sym}} \colon \widetilde{X}\to X' \simeq \widetilde{X}/{G_\text{sym}}$ is the Galois cover defined by $G_\text{sym}$, then $\mu$ factors through $\pi_{G_\text{sym}}$, i.e., there exists a Galois cover $\pi_{G/G_\text{sym}}\colon X' \simeq \widetilde{X}/{G_\text{sym}} \to Y \simeq \widetilde{X}/G$ defined by the cyclic group $G/{G_\text{sym}}\simeq \mu_m$
        such that the following diagram commutes:
        \begin{center}
            \begin{tikzcd}
            & \widetilde{X} \arrow[ld, "\widetilde \gamma"'] \arrow[rd, "\pi_{G_\text{sym}}"] \arrow[dd, "\mu"]& \\
            X \arrow[rd, "\gamma"'] &  & X' \arrow[ld, "\pi_{G/G_\text{sym}}"] \\
            & Y &                                      
            \end{tikzcd}
        \end{center}
        As $\mu$ is quasi-\'etale, the Galois covers $\pi_{G_\text{sym}}$ and $\pi_{G/{G_\text{sym}}}$ are also quasi-\'etale. Since $\widetilde{X}$ is an irreducible symplectic variety and since $G_\text{sym}$ acts symplectically on $\widetilde{X}$, by \cite[Proposition 3.17]{BGMM25} we conclude that $X' \simeq \widetilde{X}/{G_\text{sym}}$ is also an irreducible symplectic variety.
        Therefore, the map $\pi_{G/{G_{\text{sym}}}} \colon X'\to Y $ is a Galois quasi-\'etale cover defined by a cyclic group that is generated by a purely non-symplectic automorphism $g \in \Aut(X')$ of order $m=| G / G_\text{sym} |$.
    \end{proof}
    
    With the aid of the previous results we can now characterize those GGK varieties that are not irreducible symplectic.
    
    \begin{prop}
        \label{prop:characterization_IEV}
        Let $Y$ be a GGK variety. The following statements are equivalent:
        \begin{enumerate}[\normalfont (i)]
            \item $\Hol^\circ(Y) \subsetneq \Hol(Y) $.

            \item $Y \simeq X/G$, where $X$ is an irreducible symplectic variety and $G \subseteq \Aut(X)$ is a nontrivial finite cyclic group whose action is non-symplectic and free in codimension one on $X$.

            \item $ H^0 \big( Y,\Omega_Y^{[2]} \big) = \{ 0 \}$.
        \end{enumerate}
    \end{prop}

    \begin{proof}~

        \medskip

        \noindent (i) $\implies$ (ii): Since $Y$ is a GGK variety, by assumption and by \cite[Proposition F, (F.4)--(F.6)]{GGK19} we know that $Y$ cannot be an irreducible symplectic variety, so the statement follows now from \autoref{lem:GGK_cover}, taking \autoref{rem:Galois_and_quasi-etale}(ii) into account.
        
        \medskip

        \noindent (ii) $\implies$ (iii): By assumption and by \autoref{rem:Galois_and_quasi-etale}(iv) we obtain
        \[ 
            H^0 \big(Y,\Omega_Y^{[2]} \big) \simeq H^0 \big(X,\Omega_X^{[2]} \big)^G ,
        \]
        whence 
        \[ 
            0 \leq h^0 \big(Y,\Omega_Y^{[2]} \big) \leq h^0 \big(X,\Omega_X^{[2]} \big) = 1 . 
        \]
        Assume now that $h^0 \big( Y,\Omega_Y^{[2]} \big) = 1$. Then the symplectic form $\sigma$ on $X$ satisfies $g^{[*]} \sigma = \sigma$ for every $g \in G$, which contradicts the hypothesis that $G$ acts non-symplectically on $X$. Thus, 
        $ H^0 \big(Y,\Omega_Y^{[2]} \big) = \{ 0 \} $.
        
        \medskip

        \noindent (iii) $\implies$ (i): Since $Y$ is a GGK variety, by \autoref{prop:characterization_GGK_holonomy} we know that $\Hol^\circ(Y) \simeq \Sp(n)$. But $Y$ is not a symplectic variety by assumption, so we obtain $\Hol^\circ(Y) \neq \Hol(Y)$ by \cite[Proposition F, (F.4)--(F.6)]{GGK19}.
    \end{proof}
    
    \subsubsection{Irreducible Enriques varieties}

    Our previous considerations lead naturally to the definition of the following class of singular Enriques varieties.
    
    \begin{dfn}
        \label{dfn:IEV}
        An \emph{irreducible Enriques} variety is a GGK variety which satisfies any of the equivalent conditions of \autoref{prop:characterization_IEV}.
    \end{dfn}
    
    We first claim that Enriques manifolds coincide with smooth irreducible Enriques varieties. Indeed, it follows from \autoref{lem:smooth_GGK} and \autoref{prop:characterization_IEV}(iii) that \autoref{dfn:IEV} reduces to \cite[Definition 2.1]{OS11a} in the smooth case. Conversely, any Enriques manifold $Y$ is an \'etale quotient $Y \simeq X / \pi_1(Y)$, where $X$ is a projective IHS manifold, as the finite cyclic group $\pi_1(Y)$ acts freely on $X$. By the holomorphic Lefschetz fixed point formula we deduce that $\pi_1(Y)$ acts non-symplectically on $X$, see \cite[Subsection 2.2]{BNWS11}. Hence, $Y$ is a smooth irreducible Enriques variety by \autoref{prop:characterization_IEV}(ii).
    
    We present below some examples of singular irreducible Enriques varieties, which also indicate some fundamental differences between Enriques manifolds and their singular analogs. Specifically, in constrast to the smooth case, singular irreducible Enriques varieties may also be simply connected (see \autoref{example:IEV_via_anti-symplectic_involution}, \autoref{example:uniruled_Enriques} and \autoref{example:PEV_from_CCCF}) or uniruled (see \autoref{example:uniruled_Enriques} and \autoref{example:PEV_from_CCCF}) or have trivial canonical divisor (see \autoref{example:IEV_via_anti-symplectic_involution}).
    
    \begin{exa}
        \label{example:IEV_via_anti-symplectic_involution}
        Let $X$ be an IHS manifold of dimension $\dim X = 2n \geq 4$ endowed with an anti-symplectic involution $\tau \in \Aut(X)$. If the fixed locus $\Fix(\tau)$ of $\tau$ is not empty,
        then it is a finite disjoint union of Lagrangian submanifolds of $X$ by \cite[Lemma 1]{Beauville11}. 
        Therefore, the quotient $Y \coloneqq X/ \langle \tau \rangle$ is an irreducible Enriques variety whose singularities are concentrated along the image of $\Fix(\tau)$ in $Y$. Note that $Y$ is simply connected by \autoref{rem:fundamental_group_PEV}(ii), $\Q$-factorial by \cite[Proposition 5.15]{KM98}, 
        and has canonical singularities by \cite[Corollary 2.31(2) and Proposition 5.20(2)]{KM98}.
        In particular, $Y$ is not uniruled.
        In view of \autoref{rem:Galois_and_quasi-etale}(iv), we also observe that if $n$ is even, then $K_Y \sim 0$, whereas if $n$ is odd, then $K_Y \not \sim 0$.
        
        We now reproduce \cite[Example 14.9]{GGK19} to illustrate the case $n=2$ of the above construction. Let $S$ be a K3 surface endowed with an anti-symplectic involution $\tau$. Then $S$ is projective by \cite[Proposition 6(i)]{Beauville83b} and the punctual Hilbert scheme $X \coloneqq S^{[2]}$ is a projective IHS fourfold endowed with an anti-symplectic involution $\tau^{[2]}$ 
        such that $\Fix\big( \tau^{[2]} \big) \neq \emptyset$, since $\tau^{[2]}$ fixes subspaces of $S^{[2]}$ of the form $ \big \{ x , \tau^{[2]}(x) \big \}$, where $x \in S$.
        Denote by $Y$ the associated irreducible Enriques variety $X / \langle \tau^{[2]} \rangle$.
        If $\sigma$ is a holomorphic symplectic form on $X$, then $\big( \tau^{[2]} \big)^* \sigma = -\sigma$, so $\big( \tau^{[2]} \big)^* \sigma^2 = \sigma^2$. By \autoref{rem:Galois_and_quasi-etale}(iv) 
        we deduce that
        \[ 
            h^0 \big(Y,\Omega_Y^{[0]} \big) = h^0 \big(Y,\Omega_Y^{[4]} \big) = 1 \ \text{ and } \ h^0 \big(Y,\Omega_Y^{[1]} \big) = h^0 \big(Y,\Omega_Y^{[2]} \big) = h^0 \big(Y,\Omega_Y^{[3]} \big) = 0 . 
        \]
        In particular, $K_Y \sim 0$. 
        
        Finally, it is worthwhile to mention that, according to \cite[Theorem 3.6]{CGM19} and \cite[Theorem 1.1]{Tak03}, the simply connected variety $Y$ has canonical but not terminal singularities and admits a crepant resolution by a strict Calabi--Yau fourfold in the sense of \cite[Proposition 2]{Beauville83a}.
    \end{exa}

    \begin{exa}
        \label{example:uniruled_Enriques}
        Consider the quotient $Y_{3,0} \coloneqq X_{3,0} / \langle \sigma_{3,0} \rangle$ from \cite[Proposition 4.7]{AS08}, where $X_{3,0} \subset \P^4$ is a K3 surface and $\sigma_{3,0} \in \Aut(X_{3,0})$ is a purely non-symplectic automorphism of order $3$ whose fixed locus consists of $3$ isolated points: 
        $ \Fix(\sigma_{3,0}) = \{ p_1, p_2 , p_3 \} $. The action of $\sigma$ can be locally linearized and diagonalized at each $p_i$ and is given each time by the matrix
        \[
            \begin{pmatrix}
                \zeta_3^2 & 0 \\[0.25em]
                0 & \zeta_3^2 
            \end{pmatrix} 
            \in \GL(2,\C) \setminus \SL(2,\C) , 
        \]
        where $\zeta_3$ is a primitive third root of unity, see \cite[Section 2]{AS08}.
        Note that $Y_{3,0}$ is an irreducible Enriques variety with klt, but not canonical, singularities, see \cite[Proposition 4.18, Theorem 4.20 and Remark 4.21(1)]{KM98}.
        Hence, $Y_{3,0}$ is uniruled. 
        By \autoref{rem:fundamental_group_PEV}(ii) we also infer that $\pi_1 (Y_{3,0}) = \{ 1 \}$, while by \autoref{rem:Galois_and_quasi-etale}(iv) we obtain
        $h^0 \big(Y_{3,0},\Omega_{Y_{3,0}}^{[2]} \big) = 0$, and thus $K_{Y_{3,0}} \not\sim 0$.
        We thank Mirko Mauri for bringing this example to our attention and we refer to \cite{AST11,BrandhorstHofmann23} for further similar examples.
    \end{exa}

    We finally generalize \cite[Proposition 2.6]{OS11a} to the setting of irreducible Enriques varieties.
	
	\begin{prop}
		\label{prop:algebra_reflexive_forms_IEV}
		If $Y \simeq X / G$ is an irreducible Enriques variety of dimension $2n$, where $|G| = d \geq 2$, then for any integer $ p \in \{ 0, \dots, 2n \}$ we have
		\[  
			h^0 \big( Y, \Omega_Y^{[p]} \big) =
			\begin{cases}
				1 , & \text{if } p \text{ is even and } d \mid \dfrac{p}{2} \, , \\[0.4em]
				0 , & \text{otherwise} .
			\end{cases}
		\]
		In particular,
		\[ 
			1 \leq \chi(Y, \OO_Y) \leq n .
		\]
	\end{prop}
	
	\begin{proof}
		By \autoref{prop:characterization_IEV}(ii) there exist an irreducible symplectic variety $X$ and a finite group $G \subseteq \Aut(X)$ acting non-symplectically and freely in codimension one on $X$ such that $Y \simeq X/G$ (as in the statement). In particular, the map $\gamma \colon X \twoheadrightarrow X/G \simeq Y$ is a Galois quasi-\'etale cover of $Y$. By \autoref{lem:group_acting_nonsymplectically_cyclic} we know that $G$ is cyclic, generated by a purely non-symplectic automorphism $g \in \Aut(X)$ of order $d \geq 2$.
		We now distinguish two cases: 

        \medskip
        \emph{Case A}: 
        If $0 \leq p \leq 2n$ is odd, then $H^0 \big( X, \Omega_X^{[p]} \big) = \{ 0 \}$, so by \autoref{rem:Galois_and_quasi-etale}(iv) we also obtain $H^0 \big( Y, \Omega_Y^{[p]} \big) = \{ 0 \}$.

        \medskip
        \emph{Case B}: 
        If $0 \leq p \leq 2n$ is even, then we denote by $\sigma \in H^0 \big( X, \Omega_X^{[2]} \big)$ the (unique up to scalar) symplectic form on $X$ and we recall that $H^0 \big( X, \Omega_X^{[p]} \big) = \C \cdot \sigma^{p/2}$. Since by assumption we have $g^{[*]} \sigma = \zeta_d \, \sigma$, where $\zeta_d$ is a primitive $d$-th root of unity, we deduce that 
		\[ 
			g^{[*]} \, \sigma^{p/2} = \big( g^{[*]} \sigma \big)^{p/2} = \zeta_d^{p/2} \, \sigma^{p/2} , 
		\]
		and hence $\sigma^{p/2}$ is $G$-invariant if and only if $d \mid \dfrac{p}{2}$, so the assertion follows immediately from \autoref{rem:Galois_and_quasi-etale}(iv).
	\end{proof}

    \subsubsection{Primitive Enriques varieties}

    We now introduce another class of singular Enriques varieties. Our definition was originally motivated by our MMP considerations, but it is also a natural generalization of \autoref{dfn:IEV}.
    
    \begin{dfn}
        \label{dfn:PEV}
        A \emph{primitive Enriques variety} $Y$ is the quotient of a primitive symplectic variety $X$ by a finite group $G \subseteq \Aut(X)$ whose action is non-symplectic and free in codimension one.
    \end{dfn}

    Some comments about our \autoref{dfn:PEV} are now in order. 
    The group $G$ is cyclic by \autoref{lem:group_acting_nonsymplectically_cyclic}, and if $\dim X \geq 4$, then $G$ acts automatically freely in codimension one on $X$ by \autoref{prop:group_acting_nonsymplectically_free_codim_one}.
    The reason why we nonetheless kept in \autoref{dfn:PEV} the condition that $G$ acts \enquote{freely in codimension one}  
    is clarified by the following facts in the case of surfaces, which also illustrate why the construction presented in \autoref{example:IEV_via_anti-symplectic_involution} does not have, in general, the desired outcome when $\dim X = 2$. Namely, if $S$ is a K3 surface carrying an anti-symplectic involution $\tau \in \Aut(S)$, then either $\Fix(\tau) = \emptyset$ and the quotient $S / \langle \tau \rangle$ is an Enriques surface, or $\codim_S \big( \Fix(\tau) \big) = 1$ and the quotient $S / \langle \tau \rangle$ is a smooth rational surface; see \cite[Section 3]{AST11}. In the latter case we have the following example: 
    If $f \colon S \to \P^2$ is the double cover of $\P^2$ ramifying along a sextic, then $S$ is a K3 surface, the involution $\tau \in \Aut(S)$ associated with $f$ is anti-symplectic, and the fixed locus of $\tau$ is a curve. In particular, we have $\P^2 \simeq S / \langle \tau \rangle$. However, this quotient is certainly \emph{not} an Enriques surface.
    
    Primitive Enriques varieties have klt singularities and torsion canonical class due to Remark \ref{rem:singularities_quasi-etale_cover}(i). According to \autoref{lem:uniruled_K-trivial}, they are not uniruled if and only if they have canonical singularities. In addition, as we demonstrate below, \autoref{dfn:PEV} reduces to \cite[Definition 2.1]{OS11a} in the smooth case.

    \begin{lem}
        \label{lem:smooth_PEV}
        Let $Y$ be a primitive Enriques variety. If $Y$ is smooth, then it is an Enriques manifold.
    \end{lem}
    
    \begin{proof}
        By definition of a primitive Enriques variety there exists a Galois quasi-\'etale cover $X \to Y$ of $Y \simeq X / G$, where $X$ is a primitive symplectic variety and $G \subseteq \Aut(X)$ is a finite nontrivial subgroup acting non-symplectically on $X$.
        By \autoref{rem:Galois_and_quasi-etale}(i) and the assumption that $Y$ is smooth, we conclude that $\gamma$ is \'etale, and thus $X$ is smooth, 
        so it is an IHS manifold; see \autoref{rem:irreducible_vs_primitive_symplectic}(i). Since $\pi_1(X) = \{ 1 \}$, we infer that $\gamma$ is the universal covering of $Y$, and it cannot be an isomorphism, since $G \neq \{ \Id_X \}$. Therefore, $Y$ is an Enriques manifold.
    \end{proof}

    In what follows we compare Definitions \ref{dfn:IEV} and \ref{dfn:PEV}.
    
    \begin{rem}~
        \label{rem:irreducible_vs_primitive_Enriques}
        \begin{enumerate}[(i)]
            \item In the smooth case, Lemmata \ref{lem:smooth_GGK} and \ref{lem:smooth_PEV} show that being irreducible Enriques is equivalent to being primitive Enriques: these are precisely the Enriques manifolds.
        
            \item In the singular setting, \autoref{rem:irreducible_vs_primitive_symplectic}(ii) and \autoref{prop:characterization_IEV}(ii) imply that irreducible Enriques varieties are primitive Enriques, but the converse does not hold in general; see \autoref{example:PEV_from_Catanese}(b).
        \end{enumerate}
    \end{rem}

    A quasi-\'etale cover of a primitive Enriques variety is not necessarily a primitive symplectic variety, as a quasi-\'etale cover of a primitive symplectic variety is not necessarily a primitive symplectic variety; see Examples \ref{example:symmetric_product_K3s} and \ref{example:PEV_from_Catanese}(b). 
    Similarly, a crepant resolution of a primitive Enriques variety need not be an Enriques manifold, cf.\ \cite[Remark 3.3(2)]{BL21}. For instance, \cite[Section 4]{CGM19} contains several explicit examples of crepant resolutions of irreducible Enriques varieties which are strict Calabi--Yau manifolds; see also \autoref{example:IEV_via_anti-symplectic_involution}. 
    On the other hand, we construct in \autoref{example:PEV_from_Catanese}(b) a crepant resolution of a primitive (but not irreducible) Enriques variety by an Enriques manifold.
    Finally, in view of \autoref{rem:Q-fact_very_sing}, \autoref{example:PEV_from_CCCF} demonstrates that $\Q$-factorial terminal modifications (e.g., crepant resolutions) of primitive Enriques varieties are not primitive Enriques varieties in general, cf.\ \cite[Proposition 11]{Schwald20}.

    \medskip
    
    Recall now that if $Y$ is an Enriques manifold, then $q(Y) = 0$ and $\Pic(Y)$ is finitely generated by \cite[Propositions 2.6 and 2.8]{OS11a}, respectively. These properties remain valid also for primitive Enriques varieties by \autoref{prop:Hodge_duality} and \autoref{rem:Galois_and_quasi-etale}(iv),
    since they have rational singularities by \cite[Theorem 5.22]{KM98}. In the next result we establish further basic properties of primitive Enriques varieties, describing in particular partially their algebra of global reflexive forms, cf.\ \autoref{prop:algebra_reflexive_forms_IEV}.
    
    \begin{prop}
        \label{prop:basic_prop_PEV}
        Let $Y \simeq X / G$ be a primitive Enriques variety of dimension $2n$, where $X$ is a primitive symplectic variety and $G \subseteq \Aut(X)$ is generated by a purely non-symplectic automorphism $g \in \Aut(X)$ of finite order $d \geq 2$ and acts freely in codimension one on $X$. The following statements hold:
        \begin{enumerate}[\normalfont (i)]
            \itemsep 2pt

            \item
            $
                H^0 \big(Y,\Omega_Y^{[1]} \big) = H^0 \big(Y,\Omega_Y^{[2]} \big) = H^0 \big(Y,\Omega_Y^{[2n - 1]} \big) = \{ 0 \} .
            $
            
            \item
            $h^0 \big(Y, \Omega^{[2n-2]}_Y \big) = 1 \iff d \mid n-1$,

            \item
            $h^0 \big(Y, \Omega^{[2n]}_Y \big) = 1 \iff d \mid n$. In this case, $K_Y \sim 0$ and $Y$ has canonical singularities, but it is not smooth.
            
            \item 
            $H^0(Y,\mathscr{T}_Y) = \{ 0 \}$, where $\mathscr{T}_Y$ denotes the tangent sheaf of $Y$.
            
            \item 
            $\Aut(Y)$ is a discrete group.
        \end{enumerate}
    \end{prop}
    
    \begin{proof}~

        \medskip

        \noindent (i) 
        Denote by $\sigma \in H^0 \big( X, \Omega^{[2]}_X \big)$ the (unique up to scalar) symplectic form on $X$ and recall that $g^{[*]} \sigma = \zeta_d \, \sigma$, where $\zeta_d$ is a primitive $d$-th root of unity.
        This observation, together with \autoref{prop:basic_prop_PSV}(ii) and \autoref{rem:Galois_and_quasi-etale}(iv) for $k \in \{1, 2, 2n-1 \}$, yield part (i). 

        \medskip

        \noindent (ii)
        Omitted, as it is similar to the proof of part (iii), which is given below; see also the proof of \autoref{prop:algebra_reflexive_forms_IEV}.

        \medskip

        \noindent (iii)
        We first deal with the first part of the statement, keeping the same notation as in part (i). Specifically, if $d \mid n$, then 
		\[ 
			g^{[*]} \, \sigma^n = \zeta_d^n \, \sigma^n = \sigma^n ,
		\]
		which implies that the (unique up to scalar) volume form $\sigma^n$ on $X$ is $G$-invariant, and hence descends to the quotient $Y \simeq X / G$ by \autoref{rem:Galois_and_quasi-etale}(iv). Therefore, $h^0 \big(Y, \Omega^{[2n]}_Y \big) = 1$. The converse follows similarly using \autoref{rem:Galois_and_quasi-etale}(iv) for $p=2n$.
		
		We now treat the second part of the statement, so we assume that $h^0 \big(Y, \Omega^{[2n]}_Y \big) = 1$. As $K_Y \sim_\Q 0$ by construction and $K_Y \sim G \geq 0$ by assumption, we conclude that $K_Y \sim 0$. Observe now that $Y$ cannot be smooth due to \autoref{lem:smooth_PEV} and \cite[Proposition 2.8]{OS11a}. We next claim that $Y$ has canonical singularities. Indeed, if $\rho \colon W \to Y$ is a resolution of singularities of $Y$, then we may write
		\[ 
			K_W \sim \rho^* K_Y + \sum a_i E_i \sim \sum a_i E_i , 
		\]
		where the $E_i$ are $\rho$-exceptional prime divisors on $W$ and the coefficients satisfy $a_i > -1$, since $Y$ has klt singularities by construction. But since $K_W$ is a Cartier $\Z$-divisor on $W$, the same must be true for $\sum a_i E_i$, which implies that $a_i \geq 0$ for all $i$. This proves the above claim, and completes the proof of part (iii).

        \medskip

        \noindent (iv) 
        Follows from \cite[Lemma 4.6]{BL22} and \cite[Lemma 5.3]{Graf18}.

        \medskip

        \noindent (v)
        Follows from part (iv) and \cite[Lemma 3.4]{MO67}.
    \end{proof}

    Next, we briefly turn out attention to primitive Enriques surfaces. We note in passing that the paper \cite{GPP24} provides a complete classification of primitive symplectic surfaces, so it should also be possible to classify primitive Enriques surfaces. The following remark contains several observations towards this direction.

    \begin{rem}
        Let $Y \simeq X / G$ be a primitive Enriques surface,
        where $X$ is a primitive symplectic surface and $G \subseteq \Aut(X)$ is generated by a purely non-symplectic automorphism of finite order $d \geq 2$ and acts freely in codimension one on $X$; in particular, $X^\mathrm{nf}$ is a finite subset of $X$. By Propositions \ref{prop:Hodge_duality} and \ref{prop:basic_prop_PEV} we obtain 
        \[ 
            q(Y) = 0 \quad \text{ and } \quad h^0(Y, K_Y) = 0 . 
        \]
        The latter shows that $Y$ has strictly torsion canonical divisor $K_Y$ whose order is equal to $d \geq 2$ by \cite[Proposition 5.4]{BCS24}. As explained previously, $Y$ has klt, hence rational, singularities. So we now consider the following two cases:

        \medskip

        \emph{Case A}:
        If $Y$ is smooth, then it is an \emph{Enriques surface} by \autoref{lem:smooth_PEV}. In particular, $\pi_1(Y) \simeq \Z / 2\Z$.

        \medskip

        \emph{Case B}:
        If $Y$ is singular, then taking also \cite[Proposition 4.18, Theorem 4.20 and Remark 4.21(1)]{KM98} into account, we infer that $Y$ is a \emph{logarithmic Enriques surface} in the sense of \cite{Zhang91,Zhang93}. 
        A detailed description of these surfaces via their canonical coverings and the minimal resolutions thereof is given in op.\ cit. 
        We add below some comments for \emph{singular} primitive Enriques surfaces and we also distinguish two subcases to elaborate on their minimal resolution.
        
        The singular Beauville--Bogomolov decomposition theorem \cite[Theorem 1.5]{HoerPet19} yields a quasi-\'etale cover $\widetilde{Y} \to Y$ of $Y$ such that $\widetilde{Y}$ is either an abelian surface or an irreducible symplectic surface. The latter is possible if and only if $Y$ is an irreducible Enriques surface. This follows immediately from \cite[Proposition F, (F.4)--(F.6)]{GGK19} (see also \autoref{prop:characterization_IEV}) and the fact \cite[Lemma 4.7]{GGK19} that the restricted holonomy is invariant under quasi-\'etale covers.

        \medskip

        \emph{Subcase B1}:
        If $Y$ is not smooth and has canonical singularities, then its minimal resolution $W \to Y$ is crepant by assumption and by \cite[Theorem 4-6-2]{Mat02}. It follows now from the Enriques--Kodaira classification of compact complex surfaces (see, for example, \cite[Theorem 1-7-1]{Mat02}) that $W$ is an \emph{Enriques surface}. Indeed, since $K_W \sim_\Q 0$ by construction, $W$ is a minimal smooth projective surface with $\kappa(W) = 0$, $q(W) = 0$ and $p_g(W) = h^0(W,K_W) = 0$. By \cite[Theorem 1.1]{Tak03} we also obtain
        \[ 
            \pi_1(Y) \simeq \pi_1(W) \simeq \Z / 2\Z .
        \]

        \emph{Subcase B2}:
        If $Y$ has klt but not canonical singularities, then it is uniruled, and if we consider its minimal resolution $V \to Y$, then the smooth projective surface $V$ is also uniruled, so $\kappa(V) = - \infty$ by \cite[Corollary 4.12]{Debarre01book}. Hence, $V$ is a rational surface with $q(V) = 0$ and $p_g(V) = h^0(V,K_V) = 0$ by construction and by \cite[Theorem 4-6-2]{Mat02}. It follows that $Y$ itself a singular rational surface, which is simply connected by \cite[Theorem 1.1]{Tak03} and by the birational invariance of the fundamental group for complex manifolds; that is,
        \[ 
            \pi_1(Y) \simeq \pi_1(V) \simeq \pi_1(\mathbb{P}^2) \simeq \{ 1 \} .
        \]
    \end{rem}

    The following slight generalization of the construction presented in \autoref{example:IEV_via_anti-symplectic_involution} yields a possible way to construct some primitive Enriques varieties in higher (even) dimensions.

    \begin{exa}
        \label{example:PEV_via_prime_order_non-symplectic}
        Let $X$ be a primitive symplectic variety of dimension $2n \geq 4$. If $X$ carries a non-symplectic automorphism $g \in \Aut(X)$ of prime order $p$, then the quotient $Y \coloneqq X/ \langle g \rangle$ is a primitive Enriques variety by \autoref{lem:prime_order_non-symplectic}, and by \autoref{prop:basic_prop_PEV}(iii) we have $K_Y \sim 0$ if and only if $p \mid n$. When $p \nmid n$, using \cite[Proposition 5.20(1)]{KM98} we see that the quotient map $X \twoheadrightarrow Y$ is the index-one cover of $Y$.
    \end{exa}
    
    We discuss now the fundamental group of primitive Enriques varieties.
    
    \begin{rem}
        \label{rem:fundamental_group_PEV}
        Let $Y$ be a primitive Enriques variety and consider a Galois quasi-\'etale cover $\gamma \colon X \to Y$ of $Y$ such that $Y\simeq X/G$, where $X$ is a primitive symplectic variety and $G = \langle g \rangle$ is generated by a purely non-symplectic automorphism $g \in \Aut(X)$ of finite order. Recall that $\gamma$ is \'etale if and only if $X^\text{nf} = \emptyset$.
        \begin{enumerate}[(i)]
            \item Assume that $X^\text{nf} = \emptyset$. If $\pi_1(X) = \{ 1 \}$, then $\pi_1(Y) \simeq \pi_1(X/G) \simeq G$.
            
            \item Assume that $X^\text{nf} \neq \emptyset$. If $g$ is of prime order, then $\Fix(g) = X^\text{nf} \neq \emptyset$, so the natural homomorphism $\pi_1(X) \twoheadrightarrow \pi_1(Y)$ is surjective by \cite[Lemma 1.2]{Fujiki83b}. In particular, if $X$ is simply connected (e.g., if it is smooth or if it admits a symplectic resolution), then $Y$ is also simply connected.
        \end{enumerate}
    \end{rem}
    
    According to \cite[Proposition 2.10]{OS11a}, birational Enriques manifolds have isomorphic fundamental groups.
    It is straightforward to check using \cite[Theorem 1.1]{Tak03} that this remains true for birational primitive Enriques varieties as well.

    The following result was inspired by \cite[Proposition 5.11]{BCS24}, and since its proof is essentially the same, we omit it and refer to op.\ cit.\ for the details.
    
    \begin{lem}
        Let $Y$ be a primitive Enriques variety such that $\pi_1(Y) \neq \{ 1 \}$. Assume that the universal covering $X$ of $Y$ is a primitive symplectic variety. If $\pi_1(Y)$ acts non-symplectically on $X$, then $\pi_1(Y)$ is a finite cyclic group.
    \end{lem}

    
    We conclude this subsection by examining the behavior of primitive Enriques varieties under birational morphisms. In contrast to the subclass of irreducible Enriques varieties which is not preserved under birational contractions, as \autoref{example:PEV_from_Catanese}(b) demonstrates, 
    the class of primitive Enriques varieties is well-behaved. Specifically:
    
    \begin{prop}
        \label{prop:PEV_closed_under_bir_contraction}
        If $\theta \colon Y\to Y'$ is a birational morphism from a primitive Enriques variety $Y$ to a normal projective variety $Y'$, then $Y'$ is also a primitive Enriques variety.
    \end{prop}

    \begin{proof}
        Since $Y$ is a primitive Enriques variety, there exists a Galois quasi-\'etale cover $\gamma \colon X \to Y$ of $Y \simeq X / G$, where $X$ is a primitive symplectic variety and $G \neq \{ \Id_X \}$ is a finite cyclic subgroup of $\Aut(X)$ acting non-symplectically and freely in codimension one on $X$. Considering the Stein factorization of the composite map $f' \coloneqq \theta \circ \gamma$, we obtain the following commutative diagram:
        \begin{center}
            \begin{tikzcd}[row sep = large, column sep = large]
                X \arrow[r, "f"] \arrow[d, "\gamma", swap]  \arrow[dr,"f'"]& X' \arrow[d, "\gamma'"] \\
                Y \arrow[r, "\theta" swap] & Y' ,
            \end{tikzcd}
        \end{center}
        where $ f \colon X \to X'$ is a projective birational morphism, $\gamma' \colon X' \to Y'$ is a finite surjective morphism, and $X' = \Spec_{Y'} ( f'_* \OO_X )$ by \cite[Proposition 37.53.4(5)]{Stacks}. Then $X'$ is a primitive symplectic variety by \autoref{lem:symplectic_under_maps}(ii) and \autoref{prop:dom_rational_map_from_PSV}, and $\gamma'$ is a quasi-\'etale cover of $Y'$ by \autoref{rem:contraction_from_K-trivial} and \cite[Paragraph 1.41]{Debarre01book}.
        Since $f' \colon X \to Y'$ is $G$-invariant by construction,
        working locally one can check that there exists an induced $G$-action on $\Aut_{\OO_{Y'}} ( f'_* \OO_X )$, and hence an induced $G$-action on $X' = \Spec_{Y'} \big( f'_* \OO_X \big)$ such that $f \colon X \to X'$ is $G$-equivariant and $\OO_{Y'}\simeq (\gamma'_* \OO_{X'})^G$, 
        which implies that $Y' \simeq X'/G$. 
        Now, the same argument as in the proof of \autoref{lem:G-equivariant_MMP_step_symplectic}(iii) 
        shows that the $G$-action on $X'$ is non-symplectic, while arguing as in \autoref{subsection:MMP_step}
        we infer that the $G$-action on $X'$ is free in codimension one. In conclusion, $Y' \simeq X'/G$ is a primitive Enriques variety, as desired.
    \end{proof}

    \subsection{MMP for Enriques pairs}
    \label{subsection:MMP_Enriques}

    We first verify that the class of primitive Enriques varieties is preserved under the operations of the MMP, 
    which is also the case for the class of primitive symplectic varieties.
    
    \begin{lem}
        \label{lem:MMP_step_primitive_Enriques}
        Let $(Y,B_Y)$ be a log canonical pair such that $Y$ is $\Q$-factorial. Assume that $K_Y + B_Y$ is not nef and consider a step of a $(K_Y + B_Y)$-MMP:
        \begin{center}
            \begin{tikzcd}
                (Y, B_Y) \arrow[dr, "\theta" swap] \arrow[rr, "\psi", dashed] && (Y^+, B_Y^+) \arrow[dl, "\theta^+"] \\
                & W
            \end{tikzcd}
        \end{center}
        where $B_Y^+ = \psi_* B_Y$. If $Y$ is a primitive Enriques variety, then both $W$ and $Y^+$ are also primitive Enriques varieties, and $Y^+$ is $\Q$-factorial as well.
    \end{lem}

    \begin{proof}
        By assumption there exist a primitive symplectic variety $X$ and a finite subgroup $G$ of $\Aut(X)$ acting non-symplectically and freely in codimension one on $X$ such that $Y \simeq X / G$ and $X$ is $G\Q$-factorial; in particular, the quotient map $\gamma \colon X \to Y$ is a Galois quasi-\'etale cover of $Y$. By \autoref{lem:lifting_MMP_step} there exists a commutative diagram:
        \begin{center}
            \begin{tikzcd}
                (X, B) \arrow[dd, "\gamma" swap] \arrow[rr, "\varphi", dashed] \arrow[dr, "\eta" swap] && (X^+, B^+) \arrow[dl, "\eta^+"] \arrow[dd, "\gamma^+"] \\
                & Z \arrow[dd, "\mu", pos=0.3, swap] \\
                (Y,B_Y) \arrow[rr, "\psi", pos=0.65, dashed, crossing over] \arrow[dr, "\theta" swap] && (Y^+,B_Y^+) \arrow[dl, "\theta^+"] \\
                & W ,
            \end{tikzcd}
        \end{center}
        where 
        \begin{itemize}
            \item $\varphi \colon (X,B) \dashrightarrow (X^+, B^+)$ is a step of a $G$-equivariant $(K_X+B)$-MMP,

            \item $\gamma^+ \colon X^+ \to Y^+ \simeq X^+ / G$ is a Galois quasi-\'etale cover between normal projective varieties with numerically trivial canonical divisor,

            \item $\mu \colon Z \to W \simeq Z / G$ is a Galois quasi-\'etale cover between normal projective varieties,
            
            \item $B = \gamma^* B_Y$ and $B^+ = \varphi_* B = (\gamma^+)^* B_Y^+$.
        \end{itemize}
        In addition, by \autoref{lem:G-equivariant_MMP_step_symplectic} we know that both $Z$ and $X^+$ are primitive symplectic varieties and also that $G$ acts non-symplectically and freely in codimension one on them. The statement now follows, since $W \simeq Z / G$ and $Y^+ \simeq X^+ / G$, taking also the discussion about $(G)\Q$-factoriality from \autoref{subsection:MMP_step} into account.
    \end{proof}

    We next address the termination problem for Enriques pairs. Let $Y_1$ be an Enriques manifold and let $B_{Y_1}$ be an effective $\R$-divisor on $Y_1$ such that $(Y_1,B_{Y_1})$ is a log canonical pair. Denote by $\gamma_1 \colon X_1 \to Y_1$ the universal covering of the Enriques manifold $Y_1$, where $X_1$ is a projective IHS manifold, and set $B_1 \coloneqq (\gamma_1)^* B_{Y_1}$, $(X_1',B_1') \coloneqq (X_1, B_1)$ and $f_1 \coloneqq \Id_{X_1}$. Running any $(K_{Y_1}+B_{Y_1})$-MMP and invoking Propositions \ref{prop:lifting_MMP} and \ref{prop:G-equivariant_termination}, we obtain the following commutative diagram:
    \begin{equation}
        \label{diagram:Enriques_MMP_lifting}
        \begin{tikzcd}[column sep = large, row sep = large, /tikz/column 1/.append style={column sep = 0pt, inner xsep = 0pt},]
        	(X_1, B_1) \hspace{1pt} = 
            & (X_1',B_1') \arrow[d, "\Id_{X_1} \; = \; f_1" swap] \arrow[r, dashed, "\rho_1"] & (X_2',B_2') \arrow[d, "f_2" swap] \arrow[r, dashed, "\rho_2"] & (X_3',B_3') \arrow[d, "f_3" swap] \arrow[r, dashed, "\rho_3"] & \dots 
        	\\
        	& (X_1,B_1) \arrow[d, "\gamma_1" swap] \arrow[r, dashed, "\varphi_1"] & (X_2,B_2) \arrow[d, "\gamma_2" swap] \arrow[r, dashed, "\varphi_2"] & (X_3,B_3) \arrow[d, "\gamma_3" swap] \arrow[r, dashed, "\varphi_3"] & \dots 
        	\\
        	& (Y_1, B_{Y_1}) \arrow[r, dashed, "\psi_1"] & (Y_2, B_{Y_2}) \arrow[r, dashed, "\psi_2"] & (Y_3, B_{Y_3}) \arrow[r, dashed, "\psi_3"] & \dots
        \end{tikzcd}
    \end{equation}
    where 
    \begin{itemize}
        \item the bottom row is the given $(K_{Y_1}+B_{Y_1})$-MMP of $\Q$-factorial primitive Enriques varieties with canonical singularities,
            
        \item the middle row is a $\pi_1(Y)$-equivariant $(K_{X_1} + B_1)$-MMP of $G\Q$-factorial primitive symplective varieties,

        \item the top row is a $(K_{X_1'} + B_1')$-MMP of $\Q$-factorial primitive symplectic varieties,
    \end{itemize}
    and for each $i \geq 1$,
    \begin{itemize}
        \item $\gamma_i \colon (X_i,B_i) \to (Y_i,B_{Y_i})$ is a Galois quasi-\'etale cover of $Y_i$ of degree $|\pi_1(Y_1)|$ such that 
        $ K_{X_i} + B_i \sim_\R (\gamma_i)^* (K_{Y_i} + B_{Y_i}) $,

        \item $f_i \colon (X_i',B_i') \to (X_i,B_i)$ is a projective birational morphism such that 
        $ K_{X_i'} + B_i' \sim_\R (f_i)^* (K_{X_i} + B_i) $.

            
            
    \end{itemize} 
    Bearing the above in mind, we are now ready to prove our main result, \autoref{thm:termination_MMP_Enriques}.
    
    \begin{proof}[\textbf{Proof of \autoref{thm:termination_MMP_Enriques}}]
        Run any $(K_Y+B_Y)$-MMP: 
        \begin{center}
            \begin{tikzcd}
                (Y,B_Y) \coloneqq (Y_1,B_{Y_1}) \arrow[r, dashed] & (Y_2,B_{Y_2}) \arrow[r, dashed] & (Y_3,B_{Y_3}) \arrow[r, dashed] & \dots 
            \end{tikzcd}
        \end{center}
        Consider the universal covering $\pi \colon X \to Y$ of the Enriques manifold $Y \simeq X / \pi_1(Y)$. 
        Set $B \coloneqq \pi^* B_Y$ and observe that $(X,B)$ is a log canonical $\pi_1(Y)$-pair by \autoref{rem:singularities_quasi-etale_cover}(ii). By \autoref{prop:lifting_MMP} the given $(K_Y + B_Y)$-MMP can be lifted to a $\pi_1(Y)$-equivariant $(K_X+B)$-MMP:
        \begin{center}
            \begin{tikzcd}
                (X,B) \coloneqq (X_1,B_1) 
                \arrow[r, dashed] & (X_2,B_2) 
                \arrow[r, dashed] & (X_3,B_3) 
                \arrow[r, dashed] & \dots 
            \end{tikzcd}
        \end{center}
        Since the above MMP terminates by \autoref{prop:IHS_G-equivariant_termination}, the given $(K_Y+B_Y)$-MMP also terminates with a minimal model $(Y',B_{Y'})$ of $(Y,B_Y)$. 
        To complete the proof, it only remains to show that the pair $(Y',B_{Y'})$ satisfies the claimed properties, but they follow immediately from \autoref{lem:MMP_step_primitive_Enriques}, see also \autoref{subsection:MMP_step}, so we are done.
    \end{proof}

    
    In view of \autoref{lem:MMP_step_symplectic}(iii), we would finally like to investigate the following question in the remainder of this subsection.
    
    \begin{question}
        \label{question:flop_of_Enriques_manifold}
        Let $(Y,B_Y)$ be a log canonical pair such that $Y$ is an Enriques manifold. Suppose that $\psi \colon Y \dashrightarrow Y^+$ is the $B_Y$-flop of $Y$. It is true that the primitive Enriques variety $Y^+$ is also smooth, and thus an Enriques manifold?
    \end{question}
   
    Let $X$ be a projective IHS manifold of dimension $\dim X = 2n \geq 4$ and let $P$ be a closed submanifold of $X$ such that $P \simeq \mathbb{P}^n$. An \emph{elementary Mukai flop} of $X$ is a Mukai flop $X\dashrightarrow X'$ along $P$. In the next paragraph we recall its construction from \cite[Example 21.7]{GHJ03book}. 
    
    Let $Z \to X$ be the blow-up of $X$ along $P$ and let $E \subset Z$ be the exceptional divisor. Note that the symplectic form on $X$ induces an isomorphism $\mathcal{N}_{P/X} \to \Omega_P$. The projection $E \to P$ is isomorphic to the projective bundle $\mathbb{P} \big( \mathcal{N}_{P/X} \big) \simeq \mathbb{P} (\Omega_P) \to P $, which is in turn canonically isomorphic to the incidence variety 
    $ \big\{(x,H) \mid x\in H \big\} \subseteq P \times P^\vee$ (as a projective bundle over $P$), where $P^\vee$ is the dual of $P$. In particular, there exists the other projection $E \to P^\vee$. Now, since the incidence variety $E\subseteq P\times P^\vee$ is a hypersurface with normal bundle $\OO_P (1)\boxtimes\OO_{P^\vee} (1)$, it follows from the adjunction formula that the restriction of 
    $\OO_E(E)$ to every fiber of $E\to P^\vee$ is isomorphic to $\OO_{\P^{n-1}} (-1)$. We can now apply the Nakano-Fujiki criterion to conclude that there exists a blow-down $Z\to X'$ to a complex manifold $X'$ such that $E \subset Z$ is the exceptional divisor and $E \to X'$ is the projection $E\to P^\vee \subset X'$. Finally, note that $X'$ is holomorphic symplectic with $h^0 (X', \Omega_{X'}^2) = 1$ and $\pi_1(X') = \{ 1 \}$, since $X$ has the same properties, but $X'$ need not even be K\"ahler.
    In fact, $X'$ is an IHS manifold if and only if it is K\"ahler.
    
    Let $Y$ be an Enriques manifold of dimension $2n \geq 4$ and let $Q$ be a closed submanifold of $Y$ such that $Q \simeq \P^n$. We can similarly define an \emph{elementary Mukai flop} of $Y$ to be a Mukai flop $Y \dashrightarrow Y'$ along $Q$ as constructed in \cite[Section 4, p.\ 1645-1646]{OS11b}; see also the third paragraph of the proof of \autoref{prop:Mukai_flop_EM} below.
    
    In any of the above two cases we say that an elementary Mukai flop is \emph{algebraic} if it is the flop of a small contraction associated to an extremal ray of the Mori cone. We may now provide the following partial answer to \autoref{question:flop_of_Enriques_manifold}, cf.\ \cite[Proposition 4.2]{OS11b}.
    
    \begin{prop}
        \label{prop:Mukai_flop_EM}
        Let $Y$ be an Enriques manifold of dimension $2n$ and let $\gamma \colon X\to Y$ be the universal covering of $Y$. 
        If $\psi \colon Y \dashrightarrow Y^+$ is an algebraic elementary Mukai flop, then $Y^+$ is an Enriques manifold. 
        Moreover, the lifted $\pi_1(Y)$-equivariant flop $\varphi \colon X \dashrightarrow X^+$ as in \autoref{lem:lifting_MMP_step} is a composition of algebraic elementary Mukai flops.
    \end{prop}

    \begin{proof}
        Set $G \coloneqq \pi_1(Y)$ and $d \coloneqq \deg(\gamma) = |G|$. By assumption the indeterminacy locus $Q$ of $\psi$ is isomorphic to $\mathbb{P}^n$. Then the preimage $P \coloneqq \gamma^{-1} (Q) = \bigcup_{i=1}^d P_i$ is $G$-invariant. 
        We claim that $P_i \simeq \P^n$ for any $i \in \{ 1, \dots, d \}$ and that their union $P$ is disjoint. To prove this claim, we first recall that the base change of an \'etale morphism is still \'etale. Hence, $P \to Q$ is an \'etale cover; in particular, $P$ is a topological covering space of $Q\simeq \P^n$. 
        Since $\P^n$ is simply connected, the only connected covering space of $\P^n$ is $\P^n$ itself. Picking now a connected component $V$ of $\gamma^{-1}(\P^n)$ and considering the restriction of $\gamma$ to $V$, we obtain a connected covering space $\gamma|_V \colon V \to \P^n$, and thus $V \simeq \P^n$. In conclusion, $P$ is a disjoint union $P = \bigsqcup_{i=1}^d P_i$ of $d$ projective $n$-spaces $P_i \simeq \P^n$, as claimed.

        Let $\ell$ be a line in $Q$. By assumption, its class $[\ell]$ spans the extremal ray of $\NEb(Y)$ that defines the flop:
        \begin{center}
            \begin{tikzcd}
                Y \arrow[rd, "\theta"'] \arrow[rr, "\psi", dashed] & & Y^+ \arrow[ld, "\theta^+"] \\
                & W .
            \end{tikzcd}
        \end{center}
        For each $i\in \{1,\dots,d\}$ let $\ell_i \subset X$ be a line in $P_i$. Since $[\ell]$ spans an extremal ray and $\gamma$ is a finite morphism, each class $[\ell_i]$ spans an extremal ray of $\NEb(X)$. 
        These $d$ extremal rays form a $G$-orbit and generate an extremal face $F$ of $\NEb(X)$, which is simultaneously an extremal ray of $\NEb(X)^G$.
    
        We now recall the construction of an elementary Mukai flop of an Enriques manifold from \cite[Section 4, p.\ 1645-1646]{OS11b}. Specifically, we have a commutative diagram:
        \begin{center}
            \begin{tikzcd}[row sep = 2em, column sep = 2em]
                & \widehat{X} \ar[ld,swap,"p"] \ar[rd,"q"] \\
                X \arrow[d, "\gamma" swap, pos=0.4] \arrow[rr, dashed, "\rho"] & & X' \arrow[d, "\gamma'" pos=0.4] \\
                 Y \arrow[rd, "\theta" swap] \arrow[rr, dashed, "\psi"] & & Y^+ \arrow[ld, "\theta^+"] \\
                & W
            \end{tikzcd}
        \end{center}
        where 
        \begin{itemize}
            \item $p$ is the blow-up of $X$ at the disjoint union $P = \bigsqcup_{i=1}^d P_i$ with exceptional divisor $E \coloneqq \bigsqcup_{i=1}^d E_i$, where each $E_i$ is isomorphic to the incidence variety 
            \[ 
                \big\{ (x_i,H_i)\mid x_i\in H_i \big\} \subseteq P_i\times P_i^\vee ,
            \]
            
            \item $q$ is the contraction of $E$ along the second projection,
            
            \item $\rho$ is the Mukai flop of $X$ along $P = \bigsqcup_{i=1}^d P_i$, and
            
            \item $\gamma' \colon X' \to Y^+$ is the universal covering of the Enriques manifold $Y^+$; see below for the justification.
        \end{itemize}
        To verify that $\gamma' \colon X'\to Y^+$ is the universal covering of $Y^+$, 
        we have to check that $X'$ is a projective IHS manifold and that there exists an induced $G$-action on $X'$ which is free. To this end, note first that $X'$ is projective, as $\rho \colon X \dashrightarrow X'$ is a composition of $d$ algebraic elementary Mukai flops corresponding to the $d$ extremal rays spanned by the classes $[\ell_i]$, where $i \in \{ 1, \dots, d \}$. Hence, $X'$ is indeed a projective IHS manifold. Now, since $P = \bigsqcup_{i=1}^d P_i$ is $G$-invariant, $\widehat{X}$ is a $G$-variety
        and the contraction $p \colon \widehat{X} \to X$ is $G$-equivariant,
        and since every automorphism of $\P^n$ has fixed points, we infer that the action of $G$ on $P = \bigsqcup_{i=1}^d P_i$ is given by permuting the factors $P_i$.
        Therefore, there exists a canonically induced free action on $P^\vee = \big( \bigsqcup_{i=1}^d P_i \big)^\vee = \bigsqcup_{i=1}^d P_i^\vee$, and thus the contraction $q \colon \widehat{X} \to X'$ is $G$-equivariant. 
        %
        Hence, $\rho \colon X \dashrightarrow X'$ is a $G$-equivariant map and the action of $G$ on $X'$ is free. 
        Consequently, $\gamma'\colon X'\to Y^+ \simeq X'/G$ is the universal covering of the Enriques manifold $Y^+$ by construction. 

        Next, we deal with the second part of the statement. Let $\eta \colon X \to Z$ be the small contraction of the extremal face $F \subseteq \NEb(X)$. By the Rigidity lemma \cite[Proposition 1.14]{Debarre01book}, the morphism $\eta\circ p\colon \widehat{X}\to Z$ factors through $q \colon \widehat{X} \to X'$ via $\eta' \colon X' \to Z$.
        We obtain thus the following commutative diagram, where every map is $G$-equivariant:
        \begin{center}
            \begin{tikzcd}
                & \widehat{X} \arrow[ld, "p" swap] \arrow[rd, "q"] \\
                X \arrow[rd, "\eta"'] \arrow[rr, "\rho", dashed] & & X' \arrow[ld, "\eta'"] \\
                & Z .
            \end{tikzcd}
        \end{center}
        On the other hand, since the flop $\psi$ corresponds to the extremal ray $\R_{\geq 0} [\ell] \subseteq \NEb(Y)$, by \autoref{lem:lifting_MMP_step} and its proof we also obtain the following commutative diagram:
        \begin{center}
            \begin{tikzcd}
                X \arrow[dd, "\gamma"'] \arrow[rd, "\eta"'] \arrow[rr, "\varphi", dashed] & & X^+ \arrow[dd, "\gamma^+"] \arrow[ld, "\eta^+"] \\
                & Z \arrow[dd, "\mu", pos=0.3, swap] \\
                Y \arrow[rd, "\theta"'] \arrow[rr, "\psi", pos=0.65, dashed, crossing over] & & Y^+ \arrow[ld, "\theta^+"] \\
                & W \simeq Z/G ,
            \end{tikzcd}
        \end{center}
        where 
        \begin{itemize}
            \item $\varphi \colon X \dashrightarrow X^+$ is a $G$-equivariant flop corresponding to the $G$-equivariant extremal ray spanned by
            \[ 
                [\ell]^G \coloneqq \frac{1}{d} \sum_{g\in G} g_* \big( [\ell_1] \big) = \frac{1}{d}\sum_{i=1}^d [\ell_i] ,
            \]
            and

            \item $\gamma^+\colon X^+\to Y^+\simeq X^+/G$ is a Galois quasi-\'etale cover.
        \end{itemize}
        
        Let $B_Y$ be an effective $\Q$-divisor on $Y$ such that the pair $(Y,B_Y)$ is klt and the map $\psi$ is a $(K_Y+B_Y)$-flip. Let $B_{Y^+}$ be the strict transform of $B_Y$ on $Y^+$, set $B_X \coloneqq \gamma^* B_Y$ and $B_{X^+} \coloneqq (\gamma^+)^* B_{Y^+}$. 
        According to \autoref{lem:lifting_MMP_step}, the $G$-equivariant flop $\varphi$ is also a $G$-equivariant $(K_X+B_X)$-flip. In particular,
        \begin{equation}
            \label{eq:ample_X+}
            B_{X^+} \text{ is ample over } Z.
        \end{equation}
        Let $B_{X'}$ be the strict transform of $B_X$ on $X'$ and note that $B_{X'} = (\gamma')^* B_{Y^+}$.
        Since $B_{Y^+}$ is ample over $W$ and since we have the commutative diagram:
        \begin{center}
            \begin{tikzcd}[row sep = normal, column sep = normal, /tikz/column 1/.append style={column sep = 0pt, inner xsep = 0pt}]
                & & X' \arrow[dl, "\eta'"'] \arrow[dd, "\gamma'"] \\
                & Z \arrow[dd, "\mu" swap] \\
                & & X'/G \arrow[dl, "\theta^+"] & \hspace{-2.8em} \simeq \hspace{1pt} Y^+ \\
                W \simeq & Z/G , 
            \end{tikzcd}
        \end{center}
        where the vertical morphisms are finite, we infer that 
        \begin{equation}
            \label{eq:ample_X'}
            B_{X'} = (\gamma')^* B_{Y^+} \text{ is ample over } Z.
        \end{equation}
        By \cite[Lemma 5.3]{Pro21}, \eqref{eq:ample_X+} and \eqref{eq:ample_X'} 
        we conclude that $\varphi$ coincides with $\rho$ up to isomorphism. This yields the second statement.
    \end{proof}

    \subsection{Further examples}
    \label{subsection:examples_Enriques}
    
    We present here some more examples of primitive Enriques varieties, which also illustrate various phenomena that were mentioned earlier. We begin with an example 
    of an irreducible Enriques surface with canonical singularities, which is not simply connected.
        
    \begin{exa}
        \label{example:Catanese}
        Let $X \subset \P^3$ be the Cefal\'u quartic surface from \cite[Theorem 1]{Catanese2023}. It has $16$ singular points of type $A_1$ by \cite[Theorem 1(ii)]{Catanese2023} and its minimal resolution $X_1$ is a K3 surface by \cite[Theorem 1 (ix)]{Catanese2023}, which carries a fixed-point-free anti-symplectic involution $\tau_1 \in \Aut(X_1)$ according to \cite[Theorem 1 (v) and (vii)]{Catanese2023}. Thus, the quotient $Y_1 \coloneqq X_1 / \langle \tau_1 \rangle$ is an Enriques surface.
        %
        %
        
        The K3 surface $X_1$ contains $16$ disjoint $(-2)$-curves 
        $E_1,\dots, E_{16}$, which are obtained by resolving the $16$ singular points of $X$. 
        There are also $16$ disjoint $(-2)$-curves $D_1,\dots,D_{16}$ on $X_1$, called \emph{tropes}, which satisfy $\tau_1(E_i)=D_i$ for every $i \in \{1, \dots, 16\}$; for their precise construction we refer to \cite[Section 7]{Catanese2023}. Therefore, the Enriques surface $Y_1$ contains sixteen $(-2)$-curves $E_i'$ as well, where $E_i' = \gamma_1(E_i)$ and $\gamma_1 \colon X_1 \to Y_1 = X_1 / \langle \tau_1 \rangle$ is the quotient map. In fact, according to \cite[Corollary 30]{Catanese2023}, there exists a set of four pairwise disjoint $(-2)$-curves $ \big\{E_{i_1}',E_{i_2}',E_{i_3}',E_{i_4}' \big\}$ on $Y_1$. By contracting this set and by invoking \autoref{prop:PEV_closed_under_bir_contraction}, we obtain a singular, $\Q$-factorial, primitive Enriques surface $Y_2$. It has canonical singularities by \autoref{rem:contraction_from_K-trivial},
        so it is not uniruled, while \cite[Theorem 1.1]{Tak03} implies that
        \[ 
            \pi_1(Y_2) \simeq \pi_1(Y_1) \simeq \Z / 2\Z . 
        \]
        
        Observe that $Y_2$ can also be described as the quotient $X_2 / \langle \tau_2 \rangle$, where $X_2$ is obtained from $X_1$ by contracting the set of $8$ pairwise disjoint $(-2)$-curves 
        \[ 
            \big\{ E_{i_1},E_{i_2},E_{i_3},E_{i_4}, D_{i_1},D_{i_2},D_{i_3},D_{i_4} \big\} \subset X_1 ,
        \] 
        and $\tau_2$ is the induced involution on $X_2$. By construction, the involution $\tau_2 \in \Aut(X_2)$ is still fixed-point-free on $X_2$, and hence the quotient map $ \gamma_2 \colon X_2 \to Y_2 = X_2 / \langle \tau_2 \rangle$ is \'etale. We thus obtain the following commutative diagram:
        \begin{center}
            \begin{tikzcd}[row sep = 3em, column sep = large, /tikz/column 1/.append style={column sep = 0pt, inner xsep = 0pt}, /tikz/column 7/.append style={column sep = 0pt}]
                & X_1 \arrow[d, "\gamma_1" swap] \arrow[rrrrr, "{\text{contracts } 
                \big\{E_{i_1},E_{i_2},E_{i_3},E_{i_4},D_{i_1},D_{i_2},D_{i_3},D_{i_4} \big\} }"] &  &  &  &  & X_2 \arrow[d, "\gamma_2"]  \\
                Y_1 \hspace{2pt} = 
                & X_1 / \langle \tau_1 \rangle \arrow[rrrrr, "{\text{contracts } \big\{E'_{i_1},E'_{i_2},E'_{i_3},E'_{i_4} \big\} }"] &  &  &  &  & X_2 / \langle \tau_2 \rangle \hspace{-4pt} & = \hspace{2pt} Y_2 .
            \end{tikzcd}
        \end{center}

        Finally, the alternative description of $Y_2$ as the quotient $X_2 / \langle \tau_2 \rangle$ allows us to verify that it is actually an \emph{irreducible} Enriques surface. Indeed, in the terminology of the paper \cite{GPP24}, the map $X_1 \to X_2$ is the contraction of the $(-2)$-curves in the ADE configuration $B = A_1^{\oplus 8}$ on the K3 surface $X_1$, which does not appear in the statement of \cite[Theorem 3.6]{GPP24}. By \cite[Theorem 3.7]{GPP24}, $X_2$ is thus an irreducible symplectic surface, and hence $Y_2 = X_2 / \langle \tau_2 \rangle$ is an irreducible Enriques surface, as asserted.
    \end{exa}

    Building on \autoref{example:Catanese} we give next some examples of higher-dimensional primitive Enriques varieties, including smooth ones. In particular, we construct a primitive, but not irreducible, Enriques variety that admits a crepant resolution by an Enriques manifold.
    
    \begin{exa}
        \label{example:PEV_from_Catanese}
        Let $X_1$ and $\tau_1 \in \Aut(X_1)$ be as in \autoref{example:Catanese}. Let $n \geq 3$ be an odd integer. Consider the punctual Hilbert scheme $X_1^{[n]} \coloneqq \Hilb^n(X_1)$ together with the induced anti-symplectic involution $ \tau_1^{[n]} \in \Aut \big( X_1^{[n]} \big)$ on $X_1^{[n]}$. According to \cite[Proposition 4.1]{OS11a}, $\tau_1^{[n]}$ is fixed-point-free, so the quotient $Y_1^{[n]} \coloneqq X_1^{[n]} / \langle \tau_1^{[n]} \rangle$ is an Enriques manifold of dimension $2n$ and index $d=2$.

        \medskip
        
        \noindent (a) Let $2\delta$ be the class of the diagonal divisor $\Delta^{[n]}\subset X_1^{[n]}$ parametrizing non-reduced subschemes of $X_1$. 
        By \cite[Example 4.2]{HT10b} we know that $\delta^\vee$ generates an extremal ray of $\NEb \big( X_1^{[n]} \big)$ and the corresponding contraction $g \colon X_1^{[n]} \to X_1^{(n)}$, the Hilbert--Chow morphism, is a crepant resolution of singularities of the symmetric product $X_1^{(n)} \coloneqq \Sym^n(X_1)$.
        By \cite[Example 4.11]{HT10b} any smooth rational curve $C \subset X_1$ induces $\mathbb{P}^n \simeq C^{[n]} \subset X_1^{[n]}$ and the lines in this projective $n$-space correspond to pencils of binary forms of degree $n$, e.g., families of subschemes
        \[  
            l = e_1 + \ldots + e_{n-1} + C =
            \{e_1 + \ldots + e_{n-1} + e_n \mid e_n \in C \} ,
        \]
        where the points $e_1, \dots, e_{n-1} \in X_1$ are fixed. Abusing notation, we denote by $[C]$ the class of the divisor on $X_1^{[n]}$
        parametrizing the length-$n$ subschemes of $X_1$ with some support on the $(-2)$-curve $C \subset X_1$. 
        We now claim that the class $\ell \coloneqq [C]^{\vee} - (n-1)\delta^\vee$ of a line $l$ in $C^{[n]}\simeq \P^n$
        generates an extremal ray of $\NEb \big( X_1^{[n]} \big)$, see also \cite[Example 4.7]{HT10b}. Indeed, the $(-2)$-curve $C \subset X_1$ can be contracted via a morphism $X_1 \to X_1'$, where $X_1'$ is a singular K3 surface with a node. This morphism induces a morphism $X_1^{(n)} \to (X_1')^{(n)}$ which contracts $C^{(n)} \subset X_1^{(n)}$. In turn the composite map $X_1^{[n]} \to X_1^{(n)} \to (X_1')^{(n)}$ contracts both the exceptional divisor $E$
        of the Hilbert--Chow morphism and $C^{[n]} \simeq \P^n$, and hence the extremal face generated by $[E]^{\vee}$ and $\ell$. This yields the claim.
            
        Consider now two $(-2)$-curves $E_{i_1}$ and $D_{i_1}$ on $X_1$ 
        and the corresponding two Lagrangian submanifolds $E_{i_1}^{[n]}\simeq \mathbb{P}^n$ and $D_{i_1}^{[n]}\simeq \mathbb{P}^n$ of $X_1^{[n]}$, respectively. 
        Let $l$ be a line in $E_{i_1}^{[n]}$ and let $l'$ be the image of $l$ under $\tau_1^{[n]}$ in $D_{i_1}^{[n]}$. Then $R=\R_{\geq 0}(\ell+\ell')$ is an extremal ray of $\NEb \big( X_1^{[n]} \big)^{\langle \tau_1^{[n]} \rangle}$, where $\ell$ and $\ell'$ denote the classes of the lines $l \subset E_{i_1}^{[n]}$ and $l' \subset D_{i_1}^{[n]}$, respectively.
        Let $\theta \colon X_1^{[n]} \to Z$ be the $\langle \tau_1^{[n]}\rangle$-equivariant flopping contraction of $R$.
        Let $A$ be an integral ample divisor on $Z$, let $H$ be an integral ample divisor on $X_1^{[n]}$, and consider the divisor $D \coloneqq \theta^*(A)-\varepsilon \big( H+\tau_1^{[n]}(H) \big)$ on $X_1^{[n]}$, where $0< \varepsilon \ll 1$. We may assume that the $\langle \tau_1^{[n]} \rangle$-pair $\big( X_1^{[n]},D \big)$ is klt and perform the $\langle \tau_1^{[n]} \rangle$-equivariant $(K_{X_1^{[n]}}+D)$-flip 
        of $R$, which is also a composition of algebraic elementary Mukai flops:
        \begin{center}
            \begin{tikzcd}
                X_1^{[n]} \arrow[rr, dashed, "\varphi"] \arrow[dr, "\theta" swap] && X' \arrow[dl, "\theta'"] \\
                & Z .
            \end{tikzcd}
        \end{center}
        By \autoref{prop:Mukai_flop_EM}, see also the proof of \cite[Proposition 4.2]{OS11b}, we know that $X'$ is a projective IHS manifold and that the induced $\langle \tau_1^{[n]} \rangle$-action on $X'$ remains free, so the quotient $Y' \coloneqq X'/ \langle \tau_1^{[n]} \rangle$ is an Enriques manifold of dimension $2n$ and index $d=2$.

        \medskip
        
        \noindent (b) Consider again the Hilbert--Chow morphism 
        $ g \colon X_1^{[n]} \to X_1^{(n)} $.
        By \autoref{example:symmetric_product_K3s}, $X_1^{(n)}$ is a primitive, but not irreducible, symplectic variety such that
        \[ 
            \widetilde{q} \big( X_1^{(n)} \big) = 0 \quad \text{ and } \quad \pi_1 \big( X_1^{(n)} \big) = \{ 1 \} .
        \]
        Since the induced anti-symplectic involution $\tau_1^{(n)}$ on $X_1^{(n)}$ is fixed-point-free,
        the quotient $Y_1^{(n)} \coloneqq X_1^{(n)} / \langle \tau_1^{(n)} \rangle$ is a singular primitive Enriques variety with torsion, but not trivial, canonical divisor, as $n$ is odd.
        Since the quotient map 
        $X_1^{(n)} \to Y_1^{(n)}$
        is an \'etale double cover, by Remarks \ref{rem:augmented_irreg}(iii) and \ref{rem:fundamental_group_PEV}(i), respectively, we obtain
        \[ 
            \widetilde{q} \big( Y_1^{(n)} \big) = 0 \quad \text{ and } \quad \pi_1 \big( Y_1^{(n)} \big) \simeq \Z / 2\Z .
        \]
        Note also that $Y_1^{(n)}$ is not an irreducible Enriques variety, since the primitive symplectic variety $X_1^{(n)}$ is not irreducible.
        
        Finally, as the morphism $g$ is $(\Z / 2\Z)$-equivariant,
        we get a unique birational morphism 
        \[ 
            h \colon Y_1^{[n]} = X_1^{[n]} / \langle \tau_1^{[n]} \rangle \to Y_1^{(n)} = X_1^{(n)} / \langle \tau_1^{(n)} \rangle ,
        \] 
        such that the following diagram commutes:
        \begin{center}
            \begin{tikzcd}[column sep = huge, row sep = large, /tikz/column 1/.append style={column sep = 0pt, inner xsep = 0pt}]
                & X_1^{[n]} \arrow[d] \arrow[r, "g"] & X_1^{(n)} \arrow[d] \\
                Y_1^{[n]} \hspace{2pt} = 
                & X_1^{[n]} / \langle \tau^{[n]} \rangle \arrow[r, "h"] & X_1^{(n)} / \langle \tau^{(n)} \rangle & \hspace{-6.5em} = \hspace{1pt} Y_1^{(n)} .
            \end{tikzcd}
        \end{center}
        In particular, $h$ is a crepant resolution of singularities of $Y_1^{(n)}$ by \autoref{rem:contraction_from_K-trivial}, so $Y_1^{(n)}$ has canonical singularities,
        and thus it is not uniruled.
    \end{exa}

    Finally, we briefly recall the construction from \cite[Section 4]{BHS25} in order to produce $4$-dimensional examples of primitive Enriques varieties which are simply connected and uniruled, and hence admit $\Q$-factorial terminal modifications whose underlying variety cannot be primitive Enriques according to \autoref{rem:Q-fact_very_sing}.
    
    \begin{exa}
        \label{example:PEV_from_CCCF}
        Given a nodal cubic threefold $C \subset \P^4$, the triple covering of $\P^4$ ramified along $C$ defines a cuspidal cyclic cubic fourfold $Z \subset \P^5$ with one cusp $\varsigma$. The singular locus $\Sigma = F(Z)_\text{sing}$ of its Fano variety of lines $F(Z)$ is a K3 surface.
        Assume henceforth that $Z$ does not contain any planes through $\varsigma$. By \cite[Theorem 4.1(1)(3)]{BHS25} we have the following commutative diagram:
        \begin{center}
            \begin{tikzcd}[row sep = large, column sep = large]
                \Bl_{\Sigma}\big(F(Z)\big) \arrow[d, "\rho" swap] \arrow[rd, "\widetilde{\varphi^{-1}}"] 
                \\
                F(Z) \arrow[r, "\varphi^{-1}" swap, dashed]  & {\Sigma^{[2]}} ,
            \end{tikzcd}
        \end{center}
        where $\rho$ is the blow-up of $F(Z)$ along $\Sigma$ and resolves the indeterminacy locus of the (bi)rational map $\varphi^{-1}$ and $\widetilde{\varphi^{-1}}$ is an isomorphism that maps the exceptional divisor of $\rho$ isomorphically onto a divisor on $\Sigma^{[2]}$, 
        which is contracted to $\Sigma$ by the birational morphism $\varphi \colon \Sigma^{[2]} \to F(Z)$. 
        In particular, $F(Z)$ is a primitive symplectic variety by \autoref{prop:dom_rational_map_from_PSV}, since $\Bl_{\Sigma} \big( F(Z) \big) \simeq \Sigma^{[2]}$ yields a crepant resolution of $F(Z)$ by \cite[Theorem 4.1(2)]{BHS25}.
        
        Now, by \cite[Theorem 4.1(4)]{BHS25} the non-symplectic automorphism $\sigma \in \Aut\big(F(Z)\big)$ of order $3$, which is induced by the covering automorphism $\iota \in \Aut(Z)$, lifts to a non-symplectic automorphism $\widetilde{\sigma} \in \Aut \big( \Bl_\Sigma \big( F(Z) \big) \big)$ of order $3$. By \cite[Remark 4.3]{BHS25} we also know that $\Fix(\sigma) \subseteq F(Z)$ is the surface $F(C)$ 
        and $\Fix(\widetilde{\sigma}) \subseteq \Bl_\Sigma \big( F(Z) \big)$ is the strict transform $\widetilde{F(C)}$ of $F(C)$.
        Therefore, the quotient
        \[ 
            \widetilde{Y} \coloneqq \Bl_\Sigma \big(F(Z)\big) / \langle \widetilde{\sigma} \rangle 
        \]
        is an irreducible Enriques variety of dimension $4$ such that
        \[ 
            K_{\widetilde{Y}} \nsim 0 \ \text{ but } \ 3K_{\widetilde{Y}} \sim 0, \ \text{ and } \ \pi_1 \big( \widetilde{Y} \big) = \{ 1 \} , 
        \]
        see \cite[Proposition 5.4]{BCS24} and \autoref{rem:fundamental_group_PEV}(ii),
        which is singular along the image of $\widetilde{F(C)}$ in $\widetilde{Y}$. In fact, it follows from the Reid--Shepherd-Barron--Tai criterion for quotient singularities that $\widetilde{Y}$ has klt but not canonical singularities, see \cite[\S 6.2.1]{BCS24},
        so it is uniruled. Similarly, the quotient
        \[ 
            Y \coloneqq F(Z) / \langle \sigma \rangle
        \]
        is a primitive Enriques variety of dimension $4$ such that
        \[ 
            K_Y \nsim 0 \ \text{ but } \ 3K_Y \sim 0, \ \text{ and } \ \pi_1 (Y) = \{ 1 \} , 
        \]
        see \cite[Proposition 5.4]{BCS24} and \autoref{rem:fundamental_group_PEV}(ii), 
        which is singular along the image of $F(C) \cup \Sigma$ in $Y$. We finally claim that $Y$ has klt but not canonical singularities as well.
        Indeed, since both the birational morphism $\varphi \colon \Sigma^{[2]} \to F(Z)$ and the isomorphism $\widetilde{\varphi^{-1}} \colon \Bl_\Sigma \big(F(Z)\big) \to \Sigma^{[2]}$ are $(\Z / 3\Z)$-equivariant by \cite[Theorem 4.1(1)(4)]{BHS25}, we conclude that $\rho$ itself is $(\Z / 3\Z)$-equivariant, so we obtain 
        the following commutative diagram:
        \begin{equation}
            \label{diagram:PEVs_from_BHS25}
            \begin{tikzcd}[row sep = large, column sep = large, /tikz/column 1/.append style={column sep = 0pt, inner xsep = 0pt}]
                & \Bl_\Sigma \big(F(Z)\big) \arrow[r, "\rho"] \arrow[d] & F(Z) \arrow[d] \\
                \widetilde{Y} = & \Bl_\Sigma \big(F(Z)\big) / \langle \widetilde{\sigma} \rangle \arrow[r, "h"] \hspace{1em} & \hspace{1em} F(Z) / \langle \sigma \rangle = Y , 
            \end{tikzcd}
        \end{equation}
        where $h$ is a birational morphism. By \autoref{rem:contraction_from_K-trivial}, $h$ is actually crepant, so $Y$ has klt but not canonical singularities, as claimed, and hence it is uniruled.

        We conclude this example with a comment regarding the Picard numbers of all the varieties appearing in \eqref{diagram:PEVs_from_BHS25}. According to \cite[Corollary 4.2]{BHS25} and its proof, we have 
        $\rho \left( \Bl_\Sigma \big(F(Z)\big) \right) = 3$
        and
        $\rho \big( F(Z) \big) = 1$. 
        The latter, together with \autoref{lem:NeronSeveri_GaloisCovering},
        imply that $\rho (Y) = 1$, which in turn yields $\rho \big( \widetilde{Y} \big) = 3$ by \eqref{diagram:PEVs_from_BHS25} and by construction.
    \end{exa}

    \section{Asymptotic theory for Enriques manifolds}
    \label{section:asymptotic_theory_EM}

    We begin this section by recalling first the following concept, which will be needed in Subsections \ref{subsection:volume_function_Enriques} and \ref{subsection:duality}. Given a smooth projective variety $W$ and a pseudoeffective $\R$-divisor $D$ on $W$, Naka\-yama \cite{Nak04book} defined a decomposition 
	\[ D = P_\sigma (D) + N_\sigma (D) , \]  
	known as the \emph{Nakayama--Zariski decomposition} of $D$; see also Boucksom's analytic approach \cite{Boucksom04}. The divisors $P_\sigma(D)$ and $N_\sigma(D)$ are called \emph{the positive part} and \emph{the negative part}, respectively, of the Nakayama--Zariski decomposition of $ D $. Recall also that $ P_\sigma(D) \in \Movb(W) $ by \cite[Lemma III.1.8 and Proposition III.1.14(1)]{Nak04book} and that $ N_\sigma(D) \geq 0 $ by construction.

    \medskip
    
    Let $X$ be an IHS manifold of dimension $2n$. 
    Due to the works of Beauville, Bogomolov and Fujiki, there exists a non-degenerate quadratic form $q_X$ on $H^2(X,\C)$, called the \emph{Beauville--Bogomolov--Fujiki (BBF) form}, which behaves like the intersection form on a surface. There also exists a positive constant $c_X \in \Q_{>0}$, called the \emph{Fujiki constant} of $X$, such that
    \[ 
        c_X \cdot q_X(\alpha)^n= \int_X \alpha^{2n} \ \text{ for all } \alpha \in H^2(X,\C) .
    \]
    For the precise definition and the basic properties of the BBF form we refer to \cite[Sections 8 and 9]{Beauville83a} and \cite[Section 23]{GHJ03book}.
    
    An effective $\R$-divisor $D = \sum_i a_i D_i$ on an IHS manifold $X$ is called \emph{exceptional} if the Gram matrix of its irreducible components $\big( q_X (D_i, D_j) \big)_{i,j}$ is negative definite. In particular, a prime divisor $E$ on $X$ is exceptional if and only if $q_X(E) < 0$. According to \cite[Theorem 4.5]{Boucksom04}, $D$ is exceptional if and only if $D = N_{\sigma} (D)$. Taking this fact into account, we say that an effective $\R$-divisor $D$ on an Enriques manifold $Y$ is \emph{exceptional} if $D = N_\sigma(D)$.

    \subsection{The ample and nef cones of an Enriques manifold}

    Let $X$ be an IHS manifold. The \emph{K\"ahler cone} of $X$ is the open convex cone 
    \[ 
        \mathcal{K}_X \subset H^{1,1}(X,\R) \coloneqq H^{1,1}(X) \cap H^2(X,\R) 
    \]
    of all K\"ahler classes on $X$, and the \emph{positive cone} of $X$ is the connected component $\mathcal{C}_X$ of the set 
    \[
        \big\{ \alpha \in H^{1,1}(X,\R) \mid q_X(\alpha) > 0 \big\}
    \] 
    that contains the K\"ahler cone of $X$. We denote by $\overline{\mathcal{K}}_X$ (resp.\ $\overline{\mathcal{C}}_X$) the closure of $\mathcal{K}_X$ (resp.\ $\mathcal{C}_X$) in $H^{1,1}(X,\R)$. Finally, when $X$ is projective, we set $ \Pos(X) \coloneqq \mathcal{C}_X \cap N^1(X)_\R $ and we denote by $\Posb(X)$ its closure in $N^1(X)_\R$.

    

    We now derive an analog of \cite[Théorème 1.2]{Boucksom01} and \cite[Proposition 3.2]{Huy03}, which characterize ample and nef classes on IHS manifolds, for Enriques manifolds.
    
    \begin{prop}
        \label{prop:positivity_on_Enriques}
        Let $Y$ be an Enriques manifold and let $\pi \colon X \to Y$ be its universal covering. An $\R$-divisor $D$ on $Y$ is ample (resp.\ nef) if and only if $[\pi^*(D)] \in \Pos(X)$ and $D \cdot C > 0$ (resp.\ $[\pi^*(D)] \in \Posb(X)$ and $D \cdot C \geq 0$) for every rational curve $C$ in $Y$.
    \end{prop}

    \begin{proof}
        Note that $D$ is ample (resp.\ nef) if and only if $\pi^*(D)$ is ample (resp.\ nef), since $\pi$ is finite and surjective, see \cite[Proposition 1.2.13, Corollary 1.2.28 and Example 1.4.4]{Laz04book_I}. Therefore, if $D$ is ample (resp.\ nef), then $[\pi^*(D)] \in \Pos(X)$ and $D \cdot C_Y > 0$ (resp.\ $[\pi^*(D)] \in \Posb(X)$ and $D \cdot C_Y \geq 0$) for every rational curve $C_Y$ in $Y$. 
         
        Conversely, it suffices to show that $\pi^* D$ is ample (resp.\ nef). To this end, let $C_X$ be a rational curve on $X$ and note that $C_Y \coloneq \pi(C_X)$ is a rational curve on $Y$. We have $D \cdot C_Y > 0$ (resp.\ $D \cdot C_Y \geq 0$) by hypothesis, while by the projection formula we obtain 
        \[ 
            \pi^* D \cdot C_X = D \cdot (d C_Y) > 0 \quad \big( \text{resp. } \pi^* D \cdot C_X = D \cdot (d C_Y) \geq 0 \big) ,
        \] 
        where $d$ is the degree of the restriction $\pi |_{C_X} \colon C_X \to C_Y$. Therefore, by assumption and by 
        \cite[Théorème 1.2]{Boucksom01} (resp.\ \cite[Proposition 3.2]{Huy03}) 
        we deduce that
        \[  
            [\pi^* D] \in \mathcal{K}_X \quad \big( \text{resp.\ } [\pi^* D] \in \overline{\mathcal{K}}_X \big) .
        \]
        Since 
        \[ 
            \Amp(X) = \mathcal{K}_X \cap N^1(X)_\R \quad \big( \text{resp.\ } \Nef(X) = \overline{\mathcal{K}}_X \cap N^1(X)_\R \big),
        \]
        we conclude that $\pi^* D$ is ample (resp.\ nef), so we are done.
    \end{proof}
    

    \subsection{The volume function of an Enriques manifold}
    \label{subsection:volume_function_Enriques}

    In this subsection we will prove \autoref{thm:volume_function_Enriques}, which is an analog of \cite[Theorem 1.6(2)]{Denisi23} for Enriques manifolds. We begin with some reminders about the volume of divisors.
    
    Fix a projective variety $Z$ of dimension $d$. The \emph{volume} of $D \in \Div(Z)$ is defined as
    \[
        \vol_Z (D) \coloneqq \limsup_{k \to \infty} \frac{h^0 \big(Z, \OO_Z(kD) \big)}{k^d/d!} \in \R_{\geq 0} \, .
    \]
    This is an invariant of the Cartier divisor $D$, which measures the asymptotic rate of growth of the dimension of the spaces of global sections $H^0 \big( Z, \OO_Z(mD) \big)$. We have 
    \[ 
        D \, \text{ is big if and only if } \vol_Z(D)>0 . 
    \]

    The notion of volume can be extended to divisors with rational or real coefficients. Specifically, the volume of $D \in \Div_{\Q}(Z) \coloneqq \Div(Z) \otimes_{\Z} \Q$
    is defined by picking an integer $k \geq 1$ such that $kD$ is an integral Cartier divisor and by setting 
    \[ 
        \vol_Z(D) \coloneqq \frac{1}{k^{d}} \vol_Z (kD) .
    \]
    This definition does not depend on the chosen integer $k$ by \cite[Lemma 2.2.38]{Laz04book_I}. More generally, given $D \in \Div_{\R}(Z) \coloneqq \Div(Z)\otimes_{\Z} \R$ and taking into account the fact that two numerically equivalent Cartier divisors have the same volume, see \cite[Proposition 2.2.41]{Laz04book_I}, choose a sequence $\{\xi_k\}_{k=1}^\infty \subset N^1(Z)_\Q$ of $\Q$-divisor classes converging to $[D] \in N^1(Z)_\R$ and set
    \[
        \vol_Z(D) \coloneqq \lim_{k \to \infty} \vol_Z(\xi_k) .
    \]
    This definition is independent of the choice of the sequence $\{\xi_k\}_{k=1}^\infty$ by \cite[Theorem 2.2.44]{Laz04book_I}. 
    
    According to \cite[Corollary 2.2.45]{Laz04book_I}, the function 
    \[ 
        \alpha \in N^1(Z)_\Q \mapsto \vol_Z(\alpha) \in \R_{\geq 0}
    \]
    extends to a continuous function 
    \[ 
        \vol_Z(-) \colon N^1(Z)_\R \to \R_{\geq 0} , \ \alpha \mapsto \vol_Z(\alpha) , 
    \] 
    called the \emph{volume function} of $Z$. We have
    \[
        [D] \in N^1(Z)_\R \, \text{ is big if and only if } \vol_Z(D)>0 . 
    \]

    \begin{dfn}
        \label{dfn:piecewise_polynomial}
        The volume function $\vol_Z(-)$ of a projective variety $Z$ is said to be \emph{piecewise polynomial} if there exists a decomposition $N^1(Z)_\R = \bigsqcup_{i \in I} \Sigma_i$ into subsets $\Sigma_i$ with nonempty interior such that $\vol_Z |_{\Sigma_i}$ is a polynomial function for every $i \in I$.
    \end{dfn}
    %

    Note that the subsets $\Sigma_i \subset \BgCn(Z)$ are neither open nor closed in general; see \cite[Example 5.1]{Denisi23}. In the special case that $Z$ is an Enriques manifold, we will construct a decomposition of its N\'eron--Severi space $N^1(Z)_\R$ as in \autoref{dfn:piecewise_polynomial},
    where each subset $\Sigma_i \subset \BgCn(Z)$ will actually be a convex cone.
    This is the content of our \autoref{thm:volume_function_Enriques} whose proof will be given shortly and will be based on the following two results.

    \begin{lem}
        \label{lem:NZD_EM_vs_IHS}
        Let $Y$ be an Enriques manifold and let $\pi \colon X \to Y$ be its universal covering.
        If 
        \[ 
            \alpha = P_\sigma (\alpha) + N_\sigma(\alpha) 
        \]
        is the Nakayama--Zariski decomposition of $\alpha \in \BgCn(Y)$, then
        \[ 
            \pi^* \alpha = \pi^* P_\sigma(\alpha) + \pi^* N_\sigma(\alpha) .
        \]
        is the Nakayama--Zariski decomposition of its pullback $\pi^* \alpha \in \BgCn(X)$. In particular, $N_\sigma(\alpha)$ (resp.\ $N_\sigma(\pi^*\alpha) = \pi^* N_\sigma(\alpha)$) is an exceptional $\R$-divisor on $Y$ (resp.\ on $X$).
    \end{lem}

    \begin{proof}
        Since $\pi$ is finite and surjective, it follows immediately from \cite[Theorem III.5.16]{Nak04book} that $N_\sigma (\pi^* \alpha) = \pi^* N_\sigma(\alpha)$, which yields the first assertion.
        %
        The \enquote{in particular} part of the statement is settled by \cite[Theorem 3.12(i)]{Boucksom04}.
    \end{proof}
    
    \begin{prop}
        \label{prop:BZC_EM}
        Let $Y$ be an Enriques manifold and let $\pi \colon X \to Y$ be its universal covering. 
        If $S = \{ E_j \}_{j=1}^s $ is the set of the irreducible components of the negative part $N_\sigma(\beta)$ of the Nakayama--Zariski decomposition of some $\beta \in \BgCn(Y)$, then the set
        \[
            \Sigma_S \coloneqq \left\{ \alpha \in \BgCn(Y) \ \Big| \ \Supp \big( N_{\sigma}(\alpha) \big) = \bigcup_{j=1}^s \Supp(E_j) \right\} \neq \emptyset
        \]
        is a convex subcone of $\BgCn(Y)$ with nonempty interior. In particular, if $\alpha_1, \alpha_2 \in \Sigma_S$, then 
        \[ 
            \alpha_1 + \alpha_2 = \big( P_\sigma(\alpha_1) + P_\sigma(\alpha_2) \big) + \big( N_\sigma(\alpha_1) + N_\sigma(\alpha_2) \big) 
        \]
        is the Nakayama--Zariski decomposition of $\alpha_1 + \alpha_2 \in \Sigma_S$.
    \end{prop}
    
    \begin{proof}
        In what follows we use repeatedly the characterization of the Nakayama--Zariski decomposition of pseudoeffective $\R$-divisors on IHS manifolds due to Boucksom \cite[Theorem 4.8]{Boucksom04}. We also denote by $q_X$ the BBF form of the considered IHS manifold $X$ and we set $d \coloneqq | \pi_1(Y) | \geq 2$.
        
        First, we show that the cone $\Sigma_S$
        is convex. To this end, fix $\alpha_1, \alpha_2 \in \Sigma_S$ and consider their Nakayama--Zariski decompositions
        \[ 
            \alpha_1 = P_\sigma(\alpha_1) + N_\sigma(\alpha_1) \quad \text{and} \quad \alpha_2 = P_\sigma(\alpha_2) + N_\sigma(\alpha_2) , 
        \]
        respectively, where
        \begin{equation}
            \label{eq:1_supp_lem6.3}
            \Supp \big( N_{\sigma}(\alpha_1) \big) = \Supp \big( N_{\sigma}(\alpha_2) \big) = \bigcup_{j=1}^s \Supp(E_j) .
        \end{equation}
        By \autoref{lem:NZD_EM_vs_IHS} we know that the Nakayama--Zariski decompositions of $\pi^* (\alpha_1)$ and $\pi^* (\alpha_2)$ are
        \[ 
            \pi^* \alpha_1 = \pi^* P_\sigma(\alpha_1) + \pi^* N_\sigma(\alpha_1) \quad \text{and} \quad \pi^* \alpha_2 = \pi^* P_\sigma(\alpha_2) + \pi^* N_\sigma(\alpha_2) ,
        \]
        respectively, and by \eqref{eq:1_supp_lem6.3} we also deduce that
        \begin{equation}
            \label{eq:2_supp_lem6.3}
            \Supp \big( \pi^* N_{\sigma} (\alpha_1) \big) = \Supp \big( \pi^* N_{\sigma} (\alpha_2) \big) = \bigcup_{j=1}^s \pi^{-1} \big( \Supp (E_j) \big).
        \end{equation}
        
        We now claim that $\pi^* \big( P_\sigma(\alpha_1) + P_\sigma(\alpha_2) \big)$ and $\pi^* \big( N_\sigma(\alpha_1) + N_\sigma(\alpha_2) \big)$ are the positive and the negative part, respectively, of the Nakayama--Zariski decomposition of $\pi^*(\alpha_1 + \alpha_2)$. Indeed, using \eqref{eq:2_supp_lem6.3} and \cite[Theorem 4.8]{Boucksom04}, it is straightforward to verify that the Gram matrix with respect to $q_X$ of the irreducible components of $\pi^*N_\sigma(\alpha_1)+\pi^*N_\sigma(\alpha_2) \geq 0$ is negative definite (if nonzero) and also that $\pi^* P_\sigma(\alpha_1)+\pi^* P_\sigma(\alpha_2)$ is $q_X$-orthogonal to any irreducible component of $\pi^* N_\sigma(\alpha_1)+\pi^*N_\sigma(\alpha_2)$ and satisfies
        $\big[\pi^* P_\sigma(\alpha_1)+\pi^* P_\sigma(\alpha_2) \big] \in \Movb(X) \cap \BgCn(X)$.
        By the uniqueness of the Nakayama--Zariski decomposition in the setting of IHS manifolds, see \cite[Theorem 4.8]{Boucksom04}, we conclude that
        \begin{equation*}
            P_\sigma \big( \pi^* (\alpha_1 + \alpha_2) \big) = \pi^* \big( P_\sigma(\alpha_1) + P_\sigma(\alpha_2) \big)
        \end{equation*}
        and
        \begin{equation}
            \label{eq:3_Nsigma_lem6.3}
            N_\sigma \big( \pi^* (\alpha_1 + \alpha_2) \big) = \pi^* \big( N_\sigma(\alpha_1) + N_\sigma(\alpha_2) \big) ,
        \end{equation}
        as claimed. 
        
        We next claim that $P_\sigma(\alpha_1) + P_\sigma(\alpha_2)$ and $ N_\sigma(\alpha_1) + N_\sigma(\alpha_2)$ are the positive and the negative part, respectively, of the Nakayama--Zariski decomposition of $\alpha_1 + \alpha_2$. Indeed, by \eqref{eq:3_Nsigma_lem6.3} and 
        by \autoref{lem:NZD_EM_vs_IHS} for the class $\alpha_1 + \alpha_2 \in \BgCn(Y)$ we obtain 
        $\pi^* \big( N_\sigma(\alpha_1) +     N_\sigma(\alpha_2) \big) 
        = N_\sigma \big( \pi^* (\alpha_1 + \alpha_2) \big) 
        = \pi^* N_\sigma (\alpha_1 + \alpha_2) , $
        which yields
        \begin{equation}
            \label{eq:4_Nsigma_lem6.3}
            N_\sigma (\alpha_1 + \alpha_2) = N_\sigma(\alpha_1) + N_\sigma(\alpha_2) 
        \end{equation}
        in view of \cite[Chapter 7, Theorem 2.18]{Liu02book}, and hence 
        \[
            P_\sigma (\alpha_1 + \alpha_2) = P_\sigma(\alpha_1) + P_\sigma(\alpha_2) ,
        \]
        as claimed. In particular, \eqref{eq:1_supp_lem6.3} and \eqref{eq:4_Nsigma_lem6.3} imply that $\alpha_1 + \alpha_2 \in \Sigma_S$, as desired.

        Finally, we demonstrate that the interior of $\Sigma_S$ is not empty. In view of \autoref{lem:NZD_EM_vs_IHS}, we begin by considering the Boucksom--Zariski chamber $\Sigma_{S'} \subset \BgCn(X)$ associated with the set $S' = \{ E_i' \}_{i=1}^r$ of the irreducible components of $N_\sigma(\pi^*\beta) = \pi^* \big( N_{\sigma}(\beta) \big)$, whose interior is not empty; 
        we refer to \cite[Section 4]{Denisi23} for the details. We may assume for simplicity that $S$ consists of a single prime divisor $E$, which implies that the $E_i'$ lie in a unique $\pi_1(Y)$-orbit. Given any ample $\Q$-divisor $A$ on $Y$, we now construct a $\pi_1(Y)$-invariant $\Q$-divisor $D_1'$ on $X$ of the form 
        \begin{equation}
            \label{eq:BZC_int_divisor}
            D_1' = \pi^*A + a\sum_{i=1}^r E_i' + b\sum_{i=1}^r E_i' 
        \end{equation}
        which lies in the interior of the Boucksom--Zariski chamber $\Sigma_{S'}$ and the positive and negative parts of whose Nakayama--Zariski decomposition are
        \begin{equation}
            \label{eq:BZC_Psigma_Nsigma}
            P_\sigma(D_1') = \pi^*A + a\sum_i E_i' \quad \text{and} \quad N_\sigma(D_1') = b\sum_i E_i' \, ,
        \end{equation}
        respectively, for some $a, b \in \Q_{>0}$. Specifically, arguing as in the proof of \cite[Lemma 4.15]{Denisi23}, we can construct a big and movable $\Q$-divisor $M \coloneqq \pi^*A + \sum_i a_i E_i'$ on $X$ such that $\mathrm{Null}_{q_X}(M)=S'$, where $\mathrm{Null}_{q_X}(M)$ is the collection of prime divisors $E$ on $X$ such that $q_X(M,E)=0$. Then the $\Q$-divisor $M' \coloneqq \sum_{g\in G} g_*(M) = d \pi^*A + a' \sum_i E_i' $, where $a' = \sum_i a_i \in \Q_{>0}$, is also big and movable. Hence, for any $b \in \Q_{>0}$, the $\Q$-divisor $D_1' \coloneqq \frac{1}{d}M' + b \sum_i E_i'$ satisfies \eqref{eq:BZC_Psigma_Nsigma} with $a = a' / d = \big( \sum_i a_i \big) / d \in \Q_{>0}$ according to \cite[Remark 4.20]{Denisi23},
        it is big and $\pi_1(Y)$-invariant by construction, 
        and additionally it lies in the interior of $\Sigma_{S'}$ by \cite[Corollary 4.19]{Denisi23}.

        Consider next the $\Q$-divisor $D_{1,0} \coloneqq \frac{1}{d} \pi_* (D_1')$ on $Y$, where $D_1'$ is chosen as in \eqref{eq:BZC_int_divisor}. Since $\pi^* (D_{1,0}) = D_1'$,
        we infer that $D_{1,0}$ is big, and it follows now from \autoref{lem:NZD_EM_vs_IHS}, \eqref{eq:BZC_Psigma_Nsigma} and \cite[Theorem III.5.16]{Nak04book} that
        \[
            \Supp N_\sigma (D_{1,0}) = \bigcup_j \Supp E_j .
        \]
        Thus, $D_{1,0} \in \Sigma_S$. Now, pick any effective $\R$-divisor $G$ on $Y$ and consider for any $\varepsilon > 0$ the big $\R$-divisor $D_{1,\varepsilon}=D_{1,0} + \varepsilon G$ on $Y$. Since $D_1'$ lies in the interior of $\Sigma_{S'}$ by construction, $\pi^*(D_{1,\varepsilon}) = D_1'+\varepsilon \pi^*(G)$ also lies in $\Sigma_{S'}$ for every $0 \leq \varepsilon \ll 1$, so
        \[
            \Supp N_\sigma \big( \pi^*(D_{1,\varepsilon}) \big) = \bigcup_i \Supp E_i' \ \text{ for every } \, 0 \leq \varepsilon \ll 1 
        \]
        by definition of the chamber $S'$. Since $N_\sigma (D_{1,\varepsilon}) = \pi_* N_\sigma \big( \pi^*(D_{1,\varepsilon}) \big)$  by \cite[Theorem III.5.16]{Nak04book}, we conclude that 
        \[
            \Supp N_\sigma (D_{1,\varepsilon}) = \bigcup_j \Supp E_j \ \text{ for every } \, 0 \leq \varepsilon \ll 1 .
        \] 
        Thus, $D_{1,\varepsilon} \in \Sigma_S$ for every $0 \leq \varepsilon \ll 1$. In conclusion, the interior of $\Sigma_S$ is not empty.
    \end{proof}
    
    \begin{proof}[\textbf{Proof of \autoref{thm:volume_function_Enriques}}]
        Consider the universal covering $\pi \colon X \to Y$ of $Y$, where $X$ is a projective IHS manifold with BBF form $q_X$, and let $d \geq 2$ be its degree. It follows from \cite[Proposition 2.2.35]{Laz04book_I} that $\vol_Y(-)$ is homogeneous of degree $2n$, and to prove that it is piecewise polynomial, we may restrict to the big cone $\BgCn(Y)$ of $Y$, since the volume function vanishes on its complement in $N^1(Y)_\R$. 
        
        Let $S = \{ E_j \}_j$ be the set of the irreducible components of the negative part of the Nakayama--Zariski decomposition of some big class on $Y$ and consider the set
        \[
            \Sigma_S \coloneqq \left\{ \alpha \in \BgCn(Y) \ \Big| \ \Supp \big( N_{\sigma}(\alpha) \big) = \bigcup_j \Supp(E_j) \right\} ,
        \]
        which is a nonempty convex subcone of $\BgCn(Y)$ with nonempty interior by \autoref{prop:BZC_EM}.
        We will show below that $\vol_Y(-) |_{\Sigma_S}$ is polynomial. Since $\BgCn(Y)$ is the disjoint union 
        of subcones of type $\Sigma_S$, this will imply that $\vol_Y(-)$ is piecewise polynomial, as desired.
        
        Given $\alpha \in \Sigma_S$, consider the Boucksom--Zariski chamber $\Sigma_{S'} \subset \BgCn(X)$ associated with the set $S'$ of the irreducible components of $N_\sigma(\pi^*\alpha) = \pi^* \big( N_{\sigma}(\alpha) \big)$; see \cite[Definition 4.11]{Denisi23} and \autoref{lem:NZD_EM_vs_IHS}. Let $(S')^{\perp}$ be the subspace of $N^1(X)_\R$ of classes that are $q_X$-orthogonal to any element of $S'$ and let $\mathcal{B} = \{B_1,\dots,B_k\}$ be a basis for $S^{\perp}$. By the proof of \cite[Theorem 1.6(2)]{Denisi23} we know that the volume function of $X$ restricted to $\Sigma_{S'}$ is given by the polynomial
        \[
            \vol_{X}(-) \big\vert_{\Sigma_{S'}} = c_X \cdot q_X \left( \sum_{i=1}^k x_i B_i \right)^n ,
        \]
        where the $x_i$ are the coordinates with respect to the basis $\mathcal{B}$.
        Note that $\Sigma_S$ embeds into $\Sigma_{S'}$ via the inclusion $N^1(Y)_\R \hookrightarrow N^1(X)_\R$.
        It follows from \cite[Proposition 2.9(1)]{Kuronya06} and the above relation that the volume function of $Y$ restricted to $\Sigma_S$ is given by the polynomial
        \[
            \vol_{Y}(-) \big\vert_{\Sigma_S} = \frac{c_X}{d} \cdot q_X \left( \sum_{i=1}^k x_i B_i \right)^n .
        \]
        This finishes the proof.
    \end{proof}

    \subsection{Duality for cones of partially ample divisors}
    \label{subsection:duality}

    In this subsection we will prove \autoref{thm:cone_duality_Enriques}, which is an analog of \cite[Theorem B]{DRO25} for Enriques manifolds.
    
    We first recall the definitions of the various asymptotic base loci associated with an $\R$-Cartier $\R$-divisor on a normal projective variety. We refer to \cite{Nak04book,ELMNP06,BBP13,TX25} for their basic properties as well as for more information about them.
    
    \begin{dfn}
        \label{dfn:asymptoticbaseloci}
        Let $Z$ be a normal projective variety and let $D \in \Div_{\R}(Z)$.
        \begin{enumerate}[(i)]
            \item The \emph{stable base locus} of $D$ is defined as 
            \[
            \sbs(D) \coloneqq \bigcap \big\{ \Supp(E) \mid E \text{ is an effective $\R$-divisor such that $E \sim_\R D$} \big\} .
            \]
            
            \item The \emph{augmented base locus} of $D$ is defined as 
            \[
            \abs(D) \coloneqq \bigcap_{D=A+E} \Supp(E) ,
            \]
            where the intersection is taken over all decompositions of the form $D=A+E$ such that $A$ is an ample $\R$-divisor on $Z$ and $E$ is an effective $\R$-Cartier $\R$-divisor on $Z$.
            
            \item The \emph{restricted base locus} of $D$ (also referred to as the \emph{diminished base locus} in the literature) is defined as 
            \[
            \dbs(D) \coloneqq \bigcup_{A} \sbs(D+A),
            \]
            where the union is taken over all ample $\R$-divisors $A$ on $Z$.
        \end{enumerate}
    \end{dfn}
     
    The augmented and the restricted base locus of $D$ depend only on the numerical equivalence class $[D] \in N_1(Z)_\R$ of $D$, but this does not hold in general for the stable base locus of $D$, see \cite[Example 1.1]{ELMNP06}. One readily sees that we have the inclusions 
    $\dbs(D) \subseteq \sbs(D) \subseteq \abs(D)$,
    which are strict in general, even simultaneously; see \cite[Example 3.4]{TX25} for such an example.

    The next result is an analog of \cite[Theorem A(1)]{DRO25}. It estimates the dimension of the augmented and the restricted base loci of $\R$-Cartier $\R$-divisors on Enriques manifolds. It provides tacitly information about the dimension of their stable base locus as well. Indeed, for any \emph{big} $\R$-divisor $D$ on an Enriques manifold $Y$ we have $\dbs(D) = \sbs(D)$ by \cite[Corollary A]{BBP13}; see also \cite[Corollary B(ii)]{TX25}.
    
    \begin{prop}
        \label{prop:abl_Enriques}
        Let $Y$ be an Enriques manifold of dimension $2n$. If $D$ is a big (resp.\ pseudoeffective) $\R$-divisor on $Y$, then the irreducible components of $\abs(D)$ and $\dbs(D)$ have dimension greater than or equal to $n$, whenever $\abs(D) \neq \emptyset$ and $\dbs(D) \neq \emptyset$.
    \end{prop}
    
    \begin{proof}
        It suffices to show that the irreducible components of $\abs(D)$, where $D$ is any big and non-ample $\R$-divisor on $Y$, have dimension greater than or equal to $n$. Indeed, by \cite[Lemma 1.14]{ELMNP06} we have 
        \[ \dbs(D)= \bigcup_{A} \abs(D+A) , \] 
        where the union is taken over all ample $\R$-divisors $A$ on $Y$, so if $V$ is an irreducible component of $\dbs(D)$, then it is an irreducible component of $\abs(D+A)$ for some ample $\R$-divisor $A$ on $Y$. 
        
        Assume now that $D$ is a big, non-ample $\R$-divisor on $Y$, so that $\emptyset \subsetneq \abs(D) \subsetneq Y$. Consider the universal covering $\pi \colon X \to Y$ of $Y$, where $X$ is an IHS manifold, and note that the pullback $\pi^*(D)$ is again a big, non-ample $\R$-divisor on $X$ by \cite[Corollary 1.2.28]{Laz04book_I} and \cite[Proposition 2.9(1)]{Kuronya06}. By \cite[Theorem 1.1(1.1.2)]{Gomez25} we obtain
        \begin{equation}
            \label{eq:preimage_of_augmented_base_loci}
            \abs \big( \pi^*(D) \big) = \pi^{-1} \big( \abs(D) \big) . 
        \end{equation}
        Let $V$ be an irreducible component of $\abs(D)$. Then there exists an irreducible component $T$ of $\pi^{-1}(V)$ mapping onto $V$ via $\pi$. 
        By \eqref{eq:preimage_of_augmented_base_loci} and \cite[Theorem A(1)]{DRO25} we know that $\dim T \geq n$, and since $\pi$ is finite, we conclude that $\dim V \geq n$, as asserted.
    \end{proof}
    
    Using the augmented base locus of big divisor classes on a given projective variety, it is possible to characterize various cones in its N\'eron--Severi space. For example, the pseudoeffective cone $\Pseff(Z)$ of a projective variety $Z$ 
    can be viewed as the set of Cartier divisor classes on $Z$ whose augmented base locus has codimension at least one in $Z$; see \cite[Examples 1.7 and 1.18]{ELMNP06}. This construction can be generalized in a natural way as follows.
    
    \begin{dfn}
        \label{dfn:amp_k-cone}
        Let $Z$ be a normal projective variety. For any $1\leq k\leq \dim(Z)$ we denote by $\Amp^k(Z)$ the convex cone in the N\'eron--Severi space $N^1(Z)_\R$ of $Z$ spanned by those Cartier divisor classes $\alpha$ satisfying $\dim \big( \abs(\alpha) \big) \leq k-1$. A class lying in $\Amp^k(Z)$ is said to be \emph{$k$-ample}.
    \end{dfn}

    It is an immediate consequence of \cite[Examples 1.8 and 1.9]{ELMNP06} that $\Amp^k(Z)$ is a convex cone in $N^1(Z)_\R$. Moreover, all $k$-ample classes $\alpha$ on a normal projective variety $Z$ are automatically big by \cite[Example 1.7]{ELMNP06}.
    In particular, $1$-ample classes on $Z$ are ample, as the augmented base locus cannot have isolated points by \cite[Proposition 1.1]{ELMNP09} and we have $\Amp^1(Z) = \Amp(Z)$ by \cite[Example 1.7]{ELMNP06}.

    Our main objective in this subsection is to compute the duals of the cones of $k$-ample classes on Enriques manifolds, following closely the presentation from \cite[Section 4]{DRO25}. To this end, we begin with some reminders.
    
    Let $\zeta \colon Z \dashrightarrow Z'$ be a small map
    between $\Q$-factorial projective varieties. Since the N\'eron--Severi spaces $N^1(Z)_\R$ and $N^1(Z')_\R$ are isomorphic via $\zeta_*$, their dual spaces $N_1(Z)_\R$ and $N_1(Z')_\R$ are also isomorphic. Under this isomorphism, denoted by $\zeta^*_\mathrm{num}$ and called the \emph{numerical pullback of curves} \cite{Araujo10}, any curve class $\gamma' \in N_1(Z')_\R$ can be pulled back to a curve class $\zeta^*_\mathrm{num} (\gamma') \in N_1(Z)_\R$. 
    Note that if $\delta' \in N^1(Z')_\R$ and $\gamma' \in N_1(Z')_\R$, then 
    \begin{equation}
        \label{eq:numerical_pullback_of_curves}
        \zeta^* (\delta') \cdot \zeta^*_\mathrm{num} (\gamma') = \delta' \cdot \gamma' . 
    \end{equation}
    For some more information about the numerical pullback of curves we refer to \cite[Section 4]{Araujo10}.
    
    Given a map $\zeta \colon Z \dashrightarrow Z'$ as above, for any $1 \leq k \leq \dim Z$ we define
    \[ 
        \Mob_{k, \, \zeta}(Z,Z') \subseteq N_1(Z)_\R 
    \]
    to be the image under $\zeta^*_\mathrm{num}$ of the convex cone in $N_1(Z')_\R$ generated by numerical classes of irreducible curves $C'$ in $Z'$ moving in a family that sweeps out the birational image via $\zeta$ of a subvariety of $Z$ of dimension at least $k$.
    In the special case that $\zeta = \Id_Z$ we set 
    \[ 
        \Mob_k(Z) \coloneqq \Mob_{k, \, \Id_Z}(Z,Z) \subseteq N_1(Z)_\R
    \]
    and we call it the \emph{cone of $k$-movable curves} of $Z$.
    
    Now if $f \colon Y \dashrightarrow Y'$ is a birational map between primitive Enriques varieties with terminal singularities, then $f$ is small by \cite[Corollary 3.54]{KM98}.
    We can thus give the following definition.
    
    \begin{dfn}
        \label{dfn:birationally_k-movable_cone}
        Let $Y$ be an Enriques manifold. The \emph{cone of birationally $k$-movable classes} of $Y$ is defined as 
        \[
            \bMob_k(Y) \coloneqq
            \sum_{Y'} \left(\sum_{f \colon Y \dashrightarrow Y'} \Mob_{k,f}(Y,Y')\right) ,
        \]
        where the sum runs over all  $\Q$-factorial primitive Enriques varieties $Y'$ with terminal singularities and all birational maps $f \colon Y \dashrightarrow Y'$. Its closure in $N_1(Y)_\R$ is denoted by $\bMobc_k(Y)$.
    \end{dfn}

    We are finally ready to prove the main result of this section, \autoref{thm:cone_duality_Enriques}.
    
    \begin{proof}[\textbf{Proof of \autoref{thm:cone_duality_Enriques}}]
        The second part of the statement follows readily from \autoref{prop:abl_Enriques}, which shows that 
        \[
            \Amp^k(Y) = \Amp^1(Y) = \Amp(Y) \ \text{ for any } 1 \leq k \leq n ,
        \]
        so it remains to deal with the first part of the statement. We will show equivalently that 
        \[
            \overline{\Amp^k}(Y) = \bMob_k(Y)^\vee .
        \]

        
        \emph{Step 1}: We show that $ \overline{\Amp^k}(Y) \subseteq \bMob_k(Y)^\vee$. 
        
        Pick a class $\beta \in \Amp^k(Y)$ and a generating class $\gamma \in \bMob_k(Y)$. Up to passing to a birational model of $Y$, the curve class $\gamma$ is represented by an irreducible curve $C$ moving in a family that covers a subvariety $S$ of $Y$ with $\dim S \geq k$.
        On the other hand, we have $\dim \abs(\beta) < k$.
        Thus, the curve $C$ can be moved out of the locus $\abs(\beta)$, so 
        $\beta \cdot \gamma > 0$ by \cite[Remark 2.8]{DRO25}. It follows that $u \cdot v \geq 0$ for any $u \in \overline{\Amp^k}(Y)$ and any $v \in \bMob_k(Y)$, which yields the desired inclusion.

        \medskip

        \emph{Step 2}: We show that $ \bMob_k(Y)^\vee \subseteq \overline{\Amp^k}(Y)$. 
        
        It suffices to prove that
        \begin{equation}
            \label{eq:interiors}
            \operatorname{int} \left( \bMob_k(Y)^{\vee} \right) \subseteq \Amp^k(Y).
        \end{equation}
        We first note that
        \begin{equation}
            \label{eq:1_bMob_in_pseff}
            \bMob_k(Y)^{\vee} \subseteq \Pseff(Y),
        \end{equation}
        since we clearly have $\Mob_{2n}(Y) \subseteq \bMob_k(Y)$ and $\Mob_{2n}(Y)^{\vee} = \Pseff(Y)$ by \cite[Theorem 2.2]{BDPP13}.

        \medskip
        
        \emph{Step 2a}: In this step we show that
        \begin{equation}
            \label{eq:2_bMob_in_Movb}
            \bMob_{2n-1}(Y)^{\vee} \subseteq \Movb(Y) . 
        \end{equation}
        
        Since both of these cones are closed, it suffices to show that $\operatorname{int} \big( \bMob_{2n-1}(Y)^{\vee} \big) \subseteq \Movb(Y)$.
        Pick now $\alpha \in \operatorname{int} \big( \bMob_{2n-1}(Y)^{\vee} \big)$, which is a big class by \eqref{eq:1_bMob_in_pseff}, and consider its Nakayama--Zariski decomposition 
        \[ 
            \alpha = P_{\sigma}(\alpha) + N_{\sigma}(\alpha) = P + N .
        \]
        By \cite[Proposition 3.20]{Boucksom04} we have
        \begin{equation}
            \label{eq:3_volume_NZD}
            \vol_Y (P) = \vol_Y (\alpha) > 0 .
        \end{equation}
        Suppose by contradiction that $\alpha \notin \Movb(Y)$.
        We may assume that $(Y,P)$ is klt and run a $(K_Y+P)$-MMP with scaling of an ample divisor, which is a sequence of $P$-flops, as $[P] \in \Movb(Y)$. Since $P$ is big by \eqref{eq:3_volume_NZD}, \cite[Corollary 1.4.2]{BCHM10} implies that this MMP terminates with a minimal model $f \colon (Y,P) \dashrightarrow (Y',P')$, where $Y'$ is a $\Q$-factorial primitive Enriques variety with terminal singularities by \cite[Lemma 3.38]{KM98} and $P'$ is a nef and effective $\R$-divisor on $Y'$. 
        Moreover, by \eqref{diagram:Enriques_MMP_lifting} we obtain:
        \begin{center}
            \begin{tikzcd}
                X' \arrow[r,"h"] & X \arrow[r,"g"] & Y',
            \end{tikzcd}
        \end{center}
        where $X'$ is a projective IHS manifold by \autoref{lem:MMP_step_symplectic}(iii), $X$ is a $\pi_1(Y)\Q$-factorial primitive symplectic variety with terminal singularities,
        $h \colon X' \to X$ is a small crepant resolution, and $g \colon X \to Y'$ is a quasi-\'etale cover of $Y'$ with $d \coloneqq \deg(g) = |\pi_1(Y)| \geq 2$.
        Now, set $N' \coloneqq f_*(N)$ and consider the pullback 
        \[ 
            (g\circ h)^*(\alpha') = (g\circ h)^*(P') + (g\circ h)^*(N') \eqqcolon D .
        \]
        Clearly, $(g\circ h)^*(P')$ is nef.
        Since the volume of $\R$-Cartier $\R$-divisors is invariant under small $\Q$-factorial modifications, by \eqref{eq:3_volume_NZD} we obtain $\vol_{Y'}(P') = \vol_{Y'}(\alpha')$, and together with \cite[Proposition 2.2.43]{Laz04book_I} and \cite[Proposition 2.9(1)]{Kuronya06}, we infer that 
        \[ 
             \vol_{X'} \big( (g\circ h)^*(P') \big) = d \cdot \vol_{Y'}(P') = d \cdot \vol_{Y'}(\alpha') = \vol_{X'}(D) . 
        \] 
        Consequently, $(g\circ h)^*(P')$ (resp.\ $(g\circ h)^*(N')$) is the positive (resp.\ negative) part of the Nakayama--Zariski decomposition of $D$ by \cite[Remark 2.3]{DNFT17}. In particular, $D$ is not movable, and thus there exists an irreducible component $E$ of $(g\circ h)^*(N')$ such that $q_{X'}(D,E)<0$. Note that the prime divisor $E$ on the projective IHS manifold $X'$ is exceptional, as $q_{X'}(E) < 0$.
        Hence, by \cite[Proposition 3.1 and Corollary 3.6(1)]{Markman13} and \cite[Theorem 1]{Kaw91}, 
        the ray in $\NEb(X')$ spanned by
        \[ 
            [E]^{\vee} \coloneqq -2 \frac{q_{X'}(E,-)}{q_{X'}(E)} \in H_2(X', \Z)
        \]
        contains the class of a rational curve $C \subset X'$, which moves in a family sweeping out $E$ and satisfies $E \cdot C < 0$. In particular, $C$ cannot be contracted by $h$, 
        and $[C] \in \Mob_{2n-1}(X')$. Moreover, we have 
        \begin{equation}
            \label{eq:4_negative_intersection}
            D \cdot C = [E]^{\vee} (D) = -2 \frac{q_{X'}(E,D)}{q_{X'}(E)} < 0 . 
        \end{equation}
        Consider now the rational curve $C' = (g \circ h) (C)$ on $Y'$. Since $C'$ moves in a family sweeping out the prime divisor $(g\circ h)_* E$ in $Y'$, we have $[C'] \in \Mob_{2n-1}(Y')$, which yields $f^*_\mathrm{num} \big( [C] \big) \in \bMob_{2n-1}(Y)$. Since $\alpha' \cdot C' < 0$ by \eqref{eq:4_negative_intersection} and by the projection formula, by \eqref{eq:numerical_pullback_of_curves} we obtain $\alpha \cdot f^*_\mathrm{num} \big( [C] \big) < 0$, which contradicts our assumption that $\alpha \in \bMob_{2n-1}(Y)^{\vee}$ and establishes thus \eqref{eq:2_bMob_in_Movb}.

        \medskip
        
        \emph{Step 2b}: In this step we show \eqref{eq:interiors}, completing thus the proof.
        
        Arguing by contradiction, we suppose that $D$ is a big $\R$-divisor on $Y$ whose class lies in $\operatorname{int} \big( \bMob_k(Y)^{\vee} \big)$, but not in $\Amp^k(Y)$. Then there exists an irreducible component $V$ of $\abs(D)$ such that $\dim(V) \geq k$. As $[D] \in \operatorname{int} \big( \Movb(Y) \big)$ by \eqref{eq:2_bMob_in_Movb}, we can run a $(K_Y+D)$-MMP with scaling of an ample divisor $\psi\colon Y\dashrightarrow Y''$ such that $Y''$ is a terminal and $\Q$-factorial primitive Enriques variety over which the strict transform $V''$ of $V$ can be contracted; see \cite[Lemma 4.1]{TX25} and the proof of \cite[Theorem A(ii)]{TX25}. In particular, there exists a class $\gamma \in \NE(Y)$ such that $\psi_*(\gamma)$ is represented by an irreducible curve that moves in a family sweeping out $V''$, which satisfies $\psi_* \big( [D] \big) \cdot \psi_*(\gamma) \leq 0$.
        But then $\gamma \in \bMob_k(Y)$ by construction and $[D] \cdot \gamma \leq 0$ by \eqref{eq:numerical_pullback_of_curves}, which contradicts $[D] \in \operatorname{int} \big( \bMob_k(Y)^{\vee} \big)$.
    \end{proof}

    
	\bibliographystyle{amsalpha}
	\bibliography{BibliographyForPapers}
 
\end{document}